\documentclass{amsart}
\usepackage{subfigure}
\usepackage{graphicx,verbatim, amsmath, amssymb, amsfonts, ifthen,mathtools,enumerate,
placeins,longtable,hyperref,commath,caption, leftidx}	

\newtheorem{proposition}{Proposition}[section]
\newtheorem{theorem}[proposition]{Theorem}

\newtheorem{lemma}[proposition]{Lemma}
\newtheorem{prop}[proposition]{Proposition}

\newtheorem{conj}[proposition]{Conjecture}

\theoremstyle{definition}
\newtheorem{example}[proposition]{Example}

\newtheorem{remark}[proposition]{Remark}

\numberwithin{equation}{section}

\usepackage{color}

\newcommand{\margincolor}{red}      
\definecolor{darkgreen}{rgb}{0,0.7,0}

\newcounter{margincounter}
\newcommand{\marginnum}{
\ifnum\value{margincounter}<10
\textcolor{\margincolor}{\begin{picture}(0,0)\put(2.2,2.4){\circle{9}}\end{picture}\footnotesize\arabic{margincounter}}
\else\ifnum\value{margincounter}<100
\textcolor{\margincolor}{\begin{picture}(0,0)\put(4.256,2.5){\circle{11}}\end{picture}\footnotesize\arabic{margincounter}}
\else
\textcolor{\margincolor}{\begin{picture}(0,0)\put(6.8,2.5){\circle{14}}\end{picture}\footnotesize\arabic{margincounter}}
\fi\fi
}

\newcommand{\newword}[1]{\textbf{\emph{#1}}}

\newcommand{\R}{\mathbb R}
\newcommand{\g}{\mathbf{g}}
\newcommand{\x}{\mathbf{x}}

\renewcommand{\Join}{\bigvee}
\newcommand{\Meet}{\bigwedge}
\newcommand{\inv}{\operatorname{inv}}
\newcommand{\rev}{\operatorname{rev}}
\newcommand{\des}{\operatorname{des}}
\newcommand{\Can}{\operatorname{Can}}
\newcommand{\Con}{\operatorname{Con}}
\newcommand{\supp}{\operatorname{supp}}
\renewcommand{\top}{\operatorname{top}}
\newcommand{\bottom}{\operatorname{bottom}}
\newcommand{\Cat}{\operatorname{Cat}}
\newcommand{\pp}{^{+\!\!+}}
\newcommand{\biCat}{\operatorname{biCat}}
\newcommand{\biNar}{\operatorname{biNar}}
\newcommand{\Nar}{\operatorname{Nar}}

\newcommand{\meet}{\wedge}
\newcommand{\join}{\vee}

\newcommand{\E}{\mathcal{E}}
\newcommand{\F}{\mathcal{F}}

\newcommand{\C}{\operatorname{\mathbf{Camb}}}
\newcommand{\bC}{\operatorname{\mathbf{biCamb}}}
\newcommand{\pidown}{\pi_\downarrow}
\newcommand{\piup}{\pi^\uparrow}
\newcommand{\covers}{{\,\,\,\cdot\!\!\!\! >\,\,}}
\newcommand{\covered}{{\,\,<\!\!\!\!\cdot\,\,\,}}
\newcommand{\cm}{\parallel}
\newcommand{\cov}{\mathrm{cov}}
\newcommand{\br}[1]{{\langle #1 \rangle}}

\newcommand{\reals}{\mathbb{R}}

\author{Emily Barnard}
\author{Nathan Reading}
\title{Coxeter-biCatalan combinatorics}
\address{Department of Mathematics, North Carolina State University, Raleigh, NC, USA}
\thanks{Emily Barnard was supported in part by NSF grants DMS-0943855, DMS-1101568, and DMS-1500949.  Nathan Reading was supported in part by NSF grants DMS-1101568 and DMS-1500949.}

\begin{document}
\begin{abstract}
We pose counting problems related to the various settings for Coxeter-Catalan combinatorics (noncrossing, nonnesting, clusters, Cambrian).
Each problem is to count ``twin'' pairs of objects from a corresponding problem in Coxeter-Catalan combinatorics.
We show that the problems all have the same answer, and, for a given finite Coxeter group $W$, we call the common solution to these problems  the $W$-biCatalan number.
We compute the $W$-biCatalan number for all $W$ and take the first steps in the study of Coxeter-biCatalan combinatorics.
\end{abstract}
\maketitle

\setcounter{tocdepth}{2}
\tableofcontents

\section{Introduction}\label{intro sec}  
This paper considers enumeration problems closely related to Coxeter-Catalan combinatorics.
(For background on Coxeter-Catalan combinatorics, see for example \cite{Armstrong,rsga}).
Each enumeration problem can be thought of as counting pairs of ``twin'' Coxeter-Catalan objects---twin sortable elements or twin nonnesting partitions, etc.  
Many of the terms used in this introductory section are new to this paper and will be explained in Section~\ref{defs sec}.

In the setting of sortable elements and Cambrian lattices/fans, the enumeration problem is to count the following families of objects:
\begin{itemize}
\item maximal cones in the bipartite biCambrian fan (the common refinement of two bipartite Cambrian fans); 
\item pairs of twin $c$-sortable elements for bipartite $c$;  
\item classes in the bipartite biCambrian congruence (the meet of two bipartite Cambrian congruences); 
\item elements of the bipartite biCambrian lattice;
\item $c$-bisortable elements for bipartite $c$.  
\end{itemize}
In type A, $c$-bisortable elements for bipartite $c$ are in bijection with permutations avoiding a set of four bivincular patterns in the sense of \cite[Section~2]{BCDK} and with alternating arc diagrams, as will be explained in Sections~\ref{pat av sec}--\ref{alt sec}.  
In type B, similar bijections exist with certain signed permutations and with centrally symmetric alternating arc diagrams, as described in Section~\ref{type B sec}. 

In the setting of nonnesting partitions (antichains in the root poset), the enumeration problem is to count two families of objects:
\begin{itemize}
\item antichains in the doubled root poset; 
\item pairs of twin nonnesting partitions.
\end{itemize}

In the setting of clusters of almost positive roots (in the sense of \cite{ga}), the problem is to count two families of objects:  
\begin{itemize}
\item maximal cones in the bicluster fan (the common refinement of the cluster fan, in the original bipartite sense of Fomin and Zelevinsky, and its antipodal opposite);
\item pairs of twin clusters, again in the bipartite sense.
\end{itemize}

In the setting of noncrossing partitions, the problem is to count the following families of objects:
\begin{itemize}
\item pairs of twin bipartite $c$-noncrossing partitions;
\item pairs of twin bipartite $(c,c^{-1})$-noncrossing partitions.
\end{itemize}
 
The main result of this paper is the following.
\begin{theorem}\label{main thm}
For each finite Coxeter group/root system, all of the enumeration problems posed above have the same answer.
\end{theorem}

In all of the settings above except the nonnesting setting, the objects described above can be defined for arbitrary choices of a Coxeter element.
However, the enumerations depend on the choice of Coxeter element, and we emphasize that Theorem~\ref{main thm} is an assertion about the enumeration in the case where the Coxeter element is chosen to be bipartite.
See Section~\ref{bicamb sec} for the definition of Coxeter elements and bipartite Coxeter elements.

The enumeration problems in the nonnesting setting require a crystallographic root system, but outside of the nonnesting setting, Theorem~\ref{main thm} still holds for noncrystallographic types.
For each noncrystallographic type except $H_4$, one can define a root poset, and Theorem~\ref{main thm} holds; see Remark~\ref{H I remark}.

We will see in Section~\ref{defs sec} that within each group of bullet points above, the various enumeration problems have the same answer essentially by definition.
Using known uniform correspondences from the usual Coxeter-Catalan combinatorics, it is straightforward to give (in Theorems~\ref{bi cl} and~\ref{bi nc}) uniform bijections connecting the Cambrian/sortable setting to the noncrossing and cluster settings.
The difficult part of the main result is the following theorem which connects the nonnesting setting to the other settings.

\begin{theorem}\label{hard part}  
For crystallographic $W$, $c$-bisortable elements for bipartite $c$ are in bijection with antichains in the doubled root poset.
\end{theorem}
More specifically, we have the following refined version of Theorem~\ref{hard part}.
\begin{theorem}\label{hard part finer}  
For crystallographic $W$ and for any $k$, the number of bipartite $c$-bisortable elements with $k$ descents equals the number of $k$-element antichains in the doubled root poset.
\end{theorem}

Our proof of Theorems~\ref{hard part} and~\ref{hard part finer} in Section~\ref{double pos sec} would be uniform if a uniform proof were known connecting the nonnesting setting to the other settings of the usual Coxeter-Catalan combinatorics.  
Indeed, the opposite is true:
A well-behaved uniform bijection proving Theorem~\ref{hard part} or Theorem~\ref{hard part finer} would imply a uniform proof of the analogous Coxeter-Catalan statement.
(See Remark~\ref{uniform uniform} for details.)
However, the proofs of these theorems are far from a trivial recasting of Coxeter-biCatalan combinatorics in terms of Coxeter-Catalan combinatorics.
Instead, it requires a count of antichains in the doubled root poset indirectly in terms of the Coxeter-Catalan numbers and a nontrivial proof that the same formula holds for bipartite $c$-bisortable elements.   
The formula uses a notion of ``double-positive'' Catalan and Narayana numbers, which already appeared in~\cite{Ath-Sav} as the local $h$-polynomials of the positive cluster complex.
(See Remark~\ref{Ath connection} and Theorem~\ref{Cat++ thm}.)

We propose the terms \newword{$W$-biCatalan number} and \newword{$W$-biNarayana number} and the symbols $\biCat(W)$ and $\biNar_k(W)$ for the numbers appearing in Theorems~\ref{main thm} and~\ref{hard part finer}.

\begin{theorem}\label{enum thm}
The $W$-biCatalan numbers for irreducible finite Coxeter groups are: {\small
\[\begin{array}{c||c|c|c|c|c|c|c|c|c|c}
\!W\!&\!A_n\!&\!B_n\!&\!D_n\!&\!E_6\!&\!E_7\!&\!E_8\!&\!F_4\!&\!H_3\!&\!H_4\!&\!I_2(m)\!\\\hline
&&&&&&&&&&\\[-9pt]
\!\biCat(W)\!&\!\binom{2n}n\!&\!2^{2n-1}\!&\!6\cdot4^{n-2} - 2\binom{2n-4}{n-2}\!&\!1700\!&\!8872\!&\!54066\!&\!196\!&\!56\!&\!550\!&\!2m\!
\end{array}\,.\]}
\end{theorem}

The type-A and type-B cases of Theorem~\ref{enum thm} are proved, in the nonnesting setting, in Section~\ref{nn sec} by recasting the antichain count as a count of lattice paths.
The same cases can also be established in the setting of $c$-bisortable elements by recasting the problem in terms of alternating arc diagrams.
Although the latter approach is more difficult, we carry out the type-A and type-B enumeration by the latter approach in Section~\ref{type A sec}, because the combinatorial models for bipartite $c$-bisortable elements in types A and B are of independent interest, and because the enumeration of alternating arc diagrams provides the crucial insight which leads to the recursive proof of Theorem~\ref{main thm}.
(See Remark~\ref{type A insight}.)  
The type-D case of Theorem~\ref{enum thm} is much more difficult, and involves solving the type-D case of the recursion used in the proof of Theorem~\ref{hard part}.
The formula in type D was first guessed using the package \texttt{GFUN} \cite{GFUN}.
The enumerations in the exceptional types were obtained using Stembridge's \texttt{posets} and \texttt{coxeter/weyl} packages~\cite{StembridgePackages}.

We also obtain formulas for the $W$-biNarayana numbers outside of type D.
In Section~\ref{double pos sec}, we write $\biCat(W, q)$ for the polynomial in the second column in the tables below.

\begin{theorem}\label{biNar thm}
The biNarayana numbers of irreducible finite Coxeter groups, except in type D, are given by the following generating functions.
{\small
\begin{longtable}{l|l}
$W$&$\sum_{k=0}^n\biNar_k(W)\,q^k$\\[2pt]\hline\\[-9pt]\hline\\[-10pt]
$A_n$&\raisebox{0pt}[10pt][5pt]{$\sum_{k=0}^n\binom n k^2q^k$}\\\hline
$B_n$&\raisebox{0pt}[10pt][5pt]{$\sum_{k=0}^n\binom{2n}{2k}q^k$}\\\hline
$E_6$&\raisebox{0pt}[10pt][4pt]{$1+66q+415q^2+736q^3+415q^4+66q^5+q^6$}\\\hline
$E_7$&\raisebox{0pt}[10pt][4pt]{$1+119q+1139q^2+3177q^3+3177q^4+1139q^5+119q^6+q^7$}\\\hline
$E_8$&\raisebox{0pt}[10pt][4pt]{$1+232q+3226q^2+13210q^3+20728q^4+13210q^5+3226q^6+232q^7+q^8$}\\\hline
$F_4$&\raisebox{0pt}[10pt][4pt]{$1+44q+106q^2+44q^3+q^4$}\\\hline
$G_2$&\raisebox{0pt}[10pt][4pt]{$1+10q+q^2$}\\\hline
$H_3$&\raisebox{0pt}[10pt][4pt]{$1+27q+27q^2+q^3$}\\\hline
$H_4$&\raisebox{0pt}[10pt][4pt]{$1+116q+316q^2+116q^3+q^4$}\\\hline
$I_2(m)$&\raisebox{0pt}[10pt][4pt]{$1+(2m-2)q+q^2.$}
\end{longtable}}
\end{theorem}

Generating functions for biNarayana numbers for some type-D Coxeter groups are shown here.
At present we have no conjectured formula for the $D_n$-biNarayana numbers.
See Section~\ref{type D biNar sec} for a modest conjecture.

{\small
\begin{longtable}{l|l}
$D_4$&\raisebox{0pt}[10pt][4pt]{$1+20q+42q^2+20q^3+q^4$}\\\hline
$D_5$&\raisebox{0pt}[10pt][4pt]{$1+35q+136q^2+136q^3+35q^4+q^5$}\\\hline
$D_6$&\raisebox{0pt}[10pt][4pt]{$1+54q+343q^2+600q^3+343q^4+54q^5+q^6$}\\\hline
$D_7$&\raisebox{0pt}[10pt][4pt]{$1+77q+731q^2+2011q^3+2011q^4+731q^5+77q^6+q^7$}\\\hline
$D_8$&\raisebox{0pt}[10pt][4pt]{$1+104q+1384q^2+5556q^3+8638q^4+5556q^5+1384q^6+104q^7+q^8$}\\\hline
$D_9$&\raisebox{0pt}[10pt][4pt]{$1+135q+2402q^2+13314q^3+29868q^4$}\\
&\hspace{130pt}\raisebox{0pt}[10pt][4pt]{$+29868q^5+13314q^6+2402q^7+135q^8+q^9$}\\\hline
$D_{10}$&\raisebox{0pt}[10pt][4pt]{$1+170q+3901q^2+28624q^3+87874q^4+126336q^5$}\\
&\hspace{130pt}\raisebox{0pt}[10pt][4pt]{$+87874q^6+28624q^7+3901q^8+170q^9+q^{10}$}\\
\end{longtable}}

Naturally, one would like a uniform formula for the $W$-biCatalan number, but we have not found one.
A tantalizing near-miss is the \emph{non-formula} $\prod_{i=1}^n\frac{h+e_i-1}{e_i}$, where $h$ is the Coxeter number and the $e_i$ are the exponents.
This expression captures the $W$-biCatalan numbers for $W$ of types $A_n$, $B_n$, $H_3$, and $I_2(m)$---the ``coincidental types'' of \cite{Williams}---but fails to even be an integer in some other types.
In every case, the expression is a surprisingly good estimate of the $W$-biCatalan number.

Section~\ref{defs sec} is devoted to filling in definitions and details for the discussion above and proving the easy parts of Theorem~\ref{main thm}.
In Section~\ref{type A sec}, we explain why, in type A, the bipartite bisortable elements are in bijection with alternating arc diagrams and carry out the enumeration of alternating arc diagrams.
We carry out a similar enumeration in type B, in terms of centrally symmetric alternating arc diagrams.  
We conjecture that the bipartite biCambrian fan is simplicial (and thus that its dual polytope is simple), and prove the conjecture in types A and~B.
In Section~\ref{double pos sec}, we discuss double-positive Coxeter-Catalan numbers and establish a formula counting antichains in the doubled root poset in terms of double-positive Coxeter-Catalan numbers.  
We then show that bipartite $c$-bisortable elements satisfy the same recursion, thus proving Theorem~\ref{hard part finer} and completing the proof of Theorem~\ref{main thm}.  Finally, we establish some additional formulas involving double-positive Coxeter-Catalan numbers, Coxeter-Catalan numbers, and Coxeter-biCatalan numbers and use them to prove the formula for $\biCat(D_n)$ and thus complete the proof of Theorem~\ref{enum thm}.

\section{BiCatalan objects}\label{defs sec}
In this section, we fill in the definitions and details behind the enumeration problems discussed in the introduction.
An exposition in full detail would require reviewing Coxeter-Catalan combinatorics in full detail, so we leave some details to the references.

\subsection{Antichains in the doubled root poset and twin nonnesting partitions}\label{nn sec}
The \newword{root poset} of a finite crystallographic root system $\Phi$ is the set of positive roots in $\Phi$, partially ordered by setting $\alpha\le\beta$ if and only if $\beta-\alpha$ is in the nonnegative span of the simple roots.
Recall that the dual of a poset $(X,\le)$ is the poset $(X,\ge)$.
That is, the dual has the same ground set, with $x\leq y$ in the dual poset if and only if $x\ge y$ in the original poset.
The \newword{doubled root poset} consists of the root poset, together with a disjoint copy of the dual poset, identified on the simple roots.
Figure~\ref{doubled} shows some doubled root posets.  
\begin{figure}
\begin{tabular}{ccccccccc}
\includegraphics{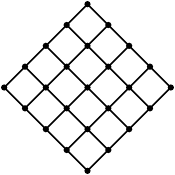}&&\includegraphics{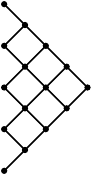}&&\includegraphics{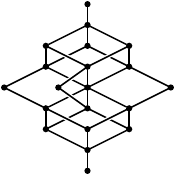}&&&&\\
$A_5$&&$B_3$\qquad\qquad\qquad&&$D_4$
\end{tabular}
\begin{tabular}{cccc}
\includegraphics{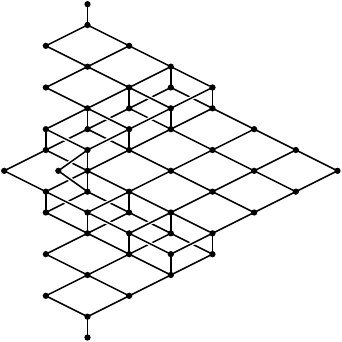}&&&\raisebox{0pt}[190pt][0pt]{\includegraphics{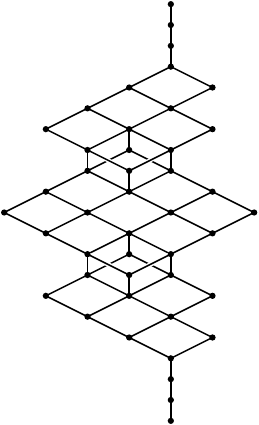}}\\
$D_6$\qquad\qquad\qquad\qquad\qquad&&&\qquad\qquad$F_4$
\end{tabular}
\caption{Some doubled root posets}
\label{doubled}
\end{figure}  

The antichain counts in types A and B are easy and known, in the guise of lattice path enumeration.
Antichains in the doubled root poset of type $A_n$ are in an easy bijection with lattice paths from $(0,0)$ to $(n,n)$ with steps $(1,0)$ and $(0,1)$.
The bijection can be made so that the number of elements in the antichain corresponds to the number of right turns in the path (the number of times a $(1,0)$-step immediately follows a $(0,1)$-step).
To specify a path with $k$ right turns, we need only specify where the right turns are. 
This means choosing $0\le x_1<\cdots<x_k\le n-1$ and $1\le y_1<\cdots<y_k\le n$ and placing right turns at $(x_1,y_1),\ldots,(x_k,y_k)$.
Thus, as is well-known, there are $\binom nk^2$ paths with $k$ right turns.

Antichains in the doubled root poset of type $B_n$ are similarly in bijection with lattice paths from $(-2n+1,-2n+1)$ to $(2n-1,2n-1)$ with steps $(2,0)$ and $(0,2)$ that are symmetric with respect to the reflection through the line $y=-x$.
 The $k$-element antichains correspond to paths with either $2k$ right turns, ($k$ of which are to the left of the line $y=-x$) or $2k-1$ right turns ($k-1$ of which are left of the line $y=-x$ and one of which is on the line $y=-x$). 
Each path is uniquely determined by its first $2n-1$ steps, whereupon the path intersects the line $y=-x$.
Thus, the paths map bijectively to words of length $2n-1$ in the letters $N$ and $E$ (for North steps $(0,2)$ and East steps $(2,0)$).
Appending the letter $E$ to the end of each word, the $k$-element antichains correspond to the words having exactly $k$ positions where an $E$ appears immediately after an $N$.
(The number of right turns in the path is odd if and only if one of these is position $2n$.)
The $2n$-letter words ending in $E$ and having exactly $k$ instances of an $E$ following an $N$ are in bijection with $2k$-element subsets of $\set{1,\ldots,2n}$.
(Given such a word, take the set of positions where the letter changes, with the convention that an $N$ in the first position is a change but an $E$ in the first position is not.
So, for example, $ENNEEE$ gives the subset $\set{2,4}$ and $NEEENE$ gives $\set{1,2,5,6}$.)
We see that there are $\binom{2n}{2k}$ $k$-element antichains, and $2^{2n-1}$ total antichains, in the doubled root poset of type~$B_n$.

\begin{remark}\label{H I remark}
It is not clear in general how one should define a ``root poset'' for a noncrystallographic root system.
See \cite[Section~5.4.1]{Armstrong} for a discussion.
In type $I_2(m)$, there is an obvious way to define an unlabeled poset generalizing the root posets of types $A_2$, $B_2$, and $G_2$.
We say ``unlabeled'' here because it is obvious how the poset should look but not obvious how the poset elements should correspond to roots.
There is also a type-$H_3$ root poset suggested in \cite[Section~5.4.1]{Armstrong}.  
For these choices of root posets, one can verify that Theorem~\ref{main thm} holds in these types as well.
\end{remark}

\begin{remark}\label{distr}
The doubled root poset, and similar posets, were probably first considered by Proctor (see \cite[Remark~4.8(a)]{StembridgeQuasi}) and then by Stembridge, as a tool for counting reduced expressions for certain elements of finite Coxeter groups. 
In the simply-laced types (A, D, and E), the doubled root poset corresponds to the \newword{smashed Cayley order} defined by Stembridge in \cite[Section~4]{StembridgeQuasi}.
In the non-simply laced types, the smashed Cayley order is disconnected and is a strictly weaker partial order than the doubled root poset.
Stembridge \cite[Theorem~4.6]{StembridgeQuasi} shows that the component whose elements are short roots is a distributive lattice.
Thus in particular the doubled root posets of types A, D, and E are distributive lattices.
One can easily check distributivity in the remaining crystallographic types B, F, and G (and in fact in types $H_3$ and $I_2(m)$).
By the Fundamental Theorem of Distributive Lattices \cite[Theorem~3.4.1]{EC1}, the doubled root poset is isomorphic to the poset of order ideals in its subposet of join-irreducible elements.
These posets of join-irreducible elements are shown in Figure~\ref{IrrDRP} for several types.
An explicit root-theoretic description of the poset of join-irreducible elements in the simply-laced types also appears in \cite[Theorem~4.6]{StembridgeQuasi}.

\end{remark}
\begin{figure}
\begin{tabular}{ccccccccccccccccccc}
\includegraphics{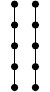}&\includegraphics{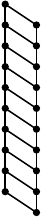}&\,\,\includegraphics{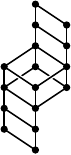}\,\,&\includegraphics{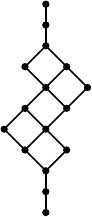}\\
$A_6$&$B_6$&$D_6$&$H_3$&&&&&&&&&&&&&&\\
\end{tabular}
\begin{tabular}{ccccccccc}
\includegraphics{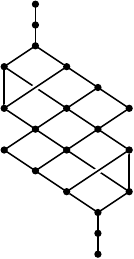}&\includegraphics{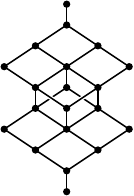}&\!\!\!\!\!\includegraphics{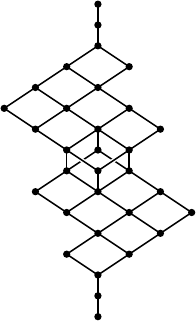}\!\!\!\!\!&\!\!\!\!\!\!\!\!\!\!\!\raisebox{0pt}[0pt][0pt]{\includegraphics{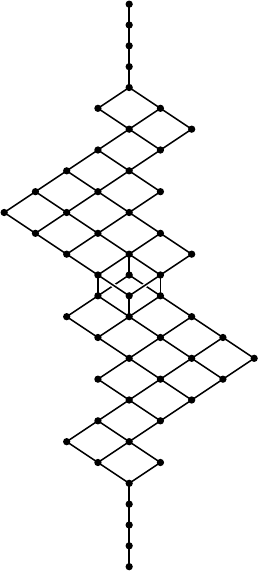}}\\
$F_4$\!\!\!\!\!\!\!\!\!\!\!\!\!\!\!\!\!\!&$E_6$&$E_7$&\!\!\!\!\!\!\!\!\!$E_8$
\end{tabular}
\caption{Some posets of join-irreducibles of doubled root posets}
\label{IrrDRP}
\end{figure}  

The \newword{support} of a root $\beta$ is the set of simple roots appearing with nonzero coefficient in the expansion of $\beta$ in the basis of simple roots.
The support of a set of roots is the union of the supports of the roots in the set.
We write $\Delta$ for the simple roots and, given a set $A$ of roots, we write $A^\circ$ for the set of non-simple roots in $A$.
If $A_1$ and $A_2$ are nonnesting partitions (i.e.\ antichains in the root poset), then $(A_1,A_2)$ is a pair of \newword{twin nonnesting partitions} if and only if $A_1\cap\Delta=A_2\cap\Delta$, and $\supp(A_1^\circ)\cap\supp(A_2^\circ)=\emptyset$. 

Given an antichain $A$ in the doubled root poset, define $\top(A)$ to be the intersection of $A$ with the root poset that forms the top of the doubled root poset.
Define $\bottom(A)$ to be the intersection of $A$ with the dual root poset that forms the bottom of the doubled root poset.  
Both $\top(A)$ and $\bottom(A)$ are sets of positive roots.
The following proposition is an immediate consequence of the observation that a root $\beta$ in the top part of the doubled root poset is related to a root $\gamma$ in the bottom part of the doubled root poset if and only if the supports of $\beta$ and $\gamma$ overlap.
\begin{prop}\label{anti twin nn}
The map $A\mapsto(\top(A),\bottom(A))$ is a bijection from antichains in the doubled root poset to pairs of twin nonnesting partitions.
\end{prop}

We pause to observe that the first biNarayana number (the number of elements of the doubled root poset) is the number of roots minus the rank of $W$.

\begin{prop}\label{biNar 1}
If $W$ is an irreducible finite Coxeter group with Coxeter number $h$ and rank $n$, then $\biNar_1(W)=n(h-1)$.
\end{prop}

\subsection{BiCambrian fans}  \label{bicamb sec}
The Cambrian fan is a complete simplicial fan whose maximal faces are naturally in bijection \cite{sortable,camb_fan} with seeds in an associated cluster algebra of finite type and with noncrossing partitions.
Furthermore, the Cambrian fan is the normal fan \cite{HohLan,HLT} to a simple polytope called the \newword{generalized associahedron} \cite{gaPoly,ga}, which encodes much of the combinatorics of the associated cluster algebra. 
More directly, the Cambrian fan is the $\g$-vector fan of the cluster algebra.
(This was conjectured, and proved in a special case, in \cite[Section~10]{camb_fan} and then proved in general in~\cite{YZ}.)

The defining data of a Cambrian fan is a finite Coxeter group $W$ and a Coxeter element $c$ of $W$.
We emphasize that the results discussed in Section~\ref{intro sec} concern a special ``bipartite'' choice of $c$, as explained below, but for now we proceed with a discussion for general $c$.
A \newword{Coxeter element} is the product of a permutation of the simple generators of $W$ and may be specified by an orientation of the Coxeter diagram.
Given a choice of $W$, we will assume the usual representation of $W$ as a reflection group acting with trivial fixed subspace.
The collection of reflecting hyperplanes in this representation is the \newword{Coxeter arrangement} of $W$.
The hyperplanes in the Coxeter arrangement cut space into cones, which constitute a fan called the \newword{Coxeter fan} $\F(W)$.
The maximal cones of the Coxeter fan are in bijection with the elements of $W$.
The Cambrian fan $\C(W,c)$ is the coarsening of the Coxeter fan obtained by gluing together maximal cones according to an equivalence relation on $W$ called the $c$-Cambrian congruence.
Further details on the $c$-Cambrian congruence appear in Section~\ref{twin sec}.    
For fixed $W$, all choices of $c$ give distinct but combinatorially isomorphic Cambrian fans.

For each Coxeter element $c$, the inverse element $c^{-1}$ is also a Coxeter element, corresponding to the opposite orientation of the diagram.
We define the \newword{biCambrian fan} $\bC(W,c)$ to be the coarsest common refinement of the Cambrian fans $\C(W,c)$ and $\C(W,c^{-1})$.  
Since $\C(W,c)$ and $\C(W,c^{-1})$ are coarsenings of $\F(W)$, so is $\bC(W,c)$.
Naturally, $\bC(W,c^{-1})=\bC(W,c)$.

\begin{example}\label{B2 example} 
To illustrate the definition, take $W$ of type $B_2$ with simple generators $s_1$ and $s_2$.
Figure~\ref{boring} shows, from left to right, the $s_1s_2$-Cambrian fan, the $s_2s_1$-Cambrian fan, and the $s_1s_2$-biCambrian fan.
Observe that the $s_1s_2$-biCambrian fan coincides with the $B_2$ Coxeter fan.
In general, when $W$ is rank 2, the $c$-biCambrian fan for any choice of Coxeter element $c$ is equal to the Coxeter fan $\F(W)$.
\end{example}
\begin{figure}
\scalebox{1}{\includegraphics{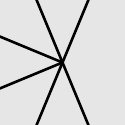}\qquad\qquad\includegraphics{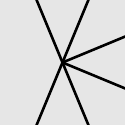}\qquad\qquad\includegraphics{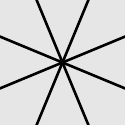}}
\caption{Cambrian fans and the biCambrian fan in type $B_2$}
\label{boring}
\end{figure}

\begin{example}\label{A3 biCamb example}  
For $W$ of type $A_3$, there are two non-isomorphic $c$-biCambrian fans, shown in Figures~\ref{A3 biCamb linear} and~\ref{A3 biCamb bipartite} respectively.
Each figure can be understood as follows:
Intersecting the $c$-biCambrian fan with a unit sphere about the origin, we obtain a decomposition of the sphere into spherical convex polygons.
The picture shows a stereographic projection of this polygonal decomposition to the plane.
In each case, the walls of one Cambrian fan are shown in red and the walls of the opposite Cambrian fan are shown in blue.
Walls that are in both Cambrian fans are shown dashed red and blue.  

\end{example}
\begin{figure}[p]
\scalebox{1.05}{\includegraphics{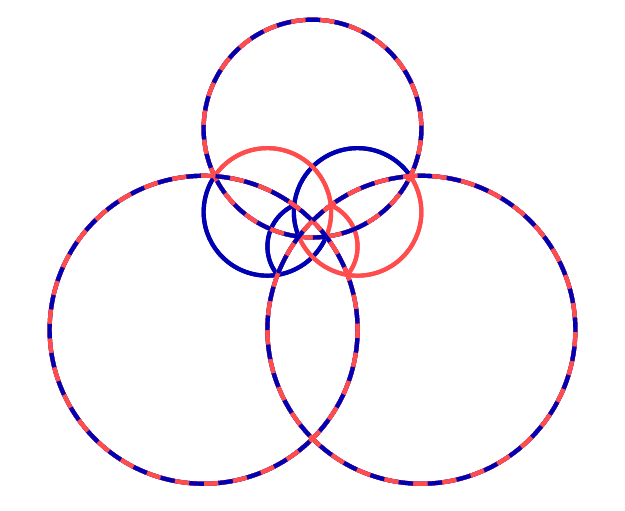}}
\caption{The linear biCambrian fan in type $A_3$}
\label{A3 biCamb linear}
\end{figure}
\begin{figure}[p]
\scalebox{1.05}{\includegraphics{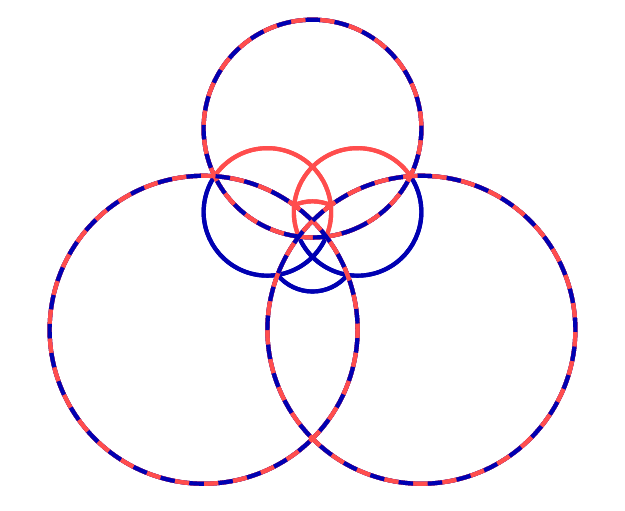}}
\caption{The bipartite biCambrian fan in type $A_3$}
\label{A3 biCamb bipartite}
\end{figure}

\begin{remark}\label{common walls}  
We observe that in Examples~\ref{B2 example} and~\ref{A3 biCamb example} that the common walls of $\C(W,c)$ and $\C(W,c^{-1})$ are exactly the reflecting hyperplanes for the simple generators of~$W$.
This fact true in general, and the simplest proof involves \newword{shards}.
We will not define shards here, but definitions and results can be found, for example, in~\cite{shardint}.
Assuming for a moment that background, we sketch a proof.
First, recast \cite[Theorem~8.3]{typefree} as the statement that the $c$-Cambrian congruence removes all but one shard from each reflecting hyperplane of $W$.
As explained in the argument for \cite[Proposition~1.3]{sort_camb} (located in \cite[Section~3]{sort_camb} just after the proof of \cite[Theorem~1.1]{sort_camb}), the antipodal map sends the shard that is not removed by the $c$-Cambrian congruence to the shard that is not removed by the $c^{-1}$-Cambrian congruence.
The only shards that are fixed by the antipodal map are shards that consist of an entire reflecting hyperplane, and \cite[Lemma~3.11]{shardint} says that these are exactly the reflecting hyperplanes for the simple generators.
\end{remark}

The $c^{-1}$-Cambrian fan $\C(W,c^{-1})$ coincides with $-\C(W,c)$, the image of the $c$-Cambrian fan under the antipodal map.
This is an immediate corollary of \cite[Proposition~1.3]{sort_camb}, which is a statement about the $c$-Cambrian congruence.
See also \cite[Remark~3.27]{afframe}.    
Thus we have the following proposition which amounts to an alternate definition of the biCambrian fan.
\begin{prop}\label{biCamb alt def} 
The biCambrian fan $\bC(W,c)$ is the coarsest common refinement of $\C(W,c)$ and $-\C(W,c)$.
\end{prop}

Since $\C(W,c)$ and $\C(W,c^{-1})$ are the normal fans of two generalized associahedra, a standard fact (see \cite[Proposition~7.12]{Ziegler}) yields the following result.

\begin{prop}\label{is poly}
For any $W$ and $c$, the fan $\bC(W,c)$ is the normal fan of a polytope, specifically, the Minkowski sum of the generalized associahedra dual to $\C(W,c)$ and $\C(W,c^{-1})$.
\end{prop}

The definition of $\bC(W,c)$ seems strange \textit{a priori}, but it is well-motivated \textit{a~posteriori} by enumerative results.
The first such result is Theorem~\ref{Baxter thm} below.
When $W$ is the symmetric group $S_n$ (i.e.\ when $W$ is of type $A_{n-1}$), the Coxeter diagram of $W$ is a path.
A \newword{linear Coxeter element} of $S_n$ is the product of the generators in order along the path.

\begin{theorem}\label{Baxter thm}
When $W$ is the symmetric group $S_n$ and $c$ is the linear Coxeter element, the number of maximal cones in $\bC(W,c)$ is the Baxter number
\[B(n)= {\binom{n+1}{1}}^{-1}{\binom{n+1}{2}}^{-1}\sum_{k=1}^n \binom{n+1}{k-1}\binom{n+1}{k}\binom{n+1}{k+1}.
\]
\end{theorem}
For more on the Baxter number, see \cite{Baxter,CGHK,FFNO}.  
Theorem~\ref{Baxter thm} was observed empirically (in the language of lattice congruences) in \cite[Section~10]{con_app} and then proven by J. West~\cite{West pers}. 
See also~\cite{Giraudo,rectangle}.  
The theorem is also related to the observation by Dulucq and Guibert \cite{DGStack} that pairs of twin binary trees are counted by the Baxter number.

Once one sees that the Baxter number counts maximal cones of $\bC(W,c)$ for $W$ of type~A and for a particular $c$, it is natural to look at other types of finite Coxeter group $W$, with the idea of defining a ``$W$-Baxter number'' for each finite Coxeter group $W$.
Indeed, there is a good notion of a ``type-B Baxter number'' discovered by Dilks~\cite{Dilks}.
The Coxeter diagram of type~B is also a path, and taking $c$ to be a linear Coxeter element, the maximal cones of $\bC(W,c)$ are counted by the type-B Baxter number.  
Despite the nice type-B result, there seems to be little hope for a reasonable definition of the $W$-Baxter number, because some types of Coxeter diagrams are not paths and thus it is not clear how to generalize the notion of a \emph{linear} Coxeter element.

There is, however, a choice of Coxeter element that can be made uniformly for all finite Coxeter groups.
Since the Coxeter diagram of any finite Coxeter group is acyclic, the diagram is in particular bipartite. 
Thus we can fix a bipartition $S_+\cup S_-$ of the diagram and orient each edge of the diagram from its vertex in $S_-$ to its vertex in $S_+$.
The resulting Coxeter element is called a \newword{bipartite Coxeter element}, and if $c$ is a bipartite Coxeter element of $W$, we call $\bC(W,c)$ a \newword{bipartite biCambrian fan}.
We emphasize that the case of bipartite $c$ is very special.
In particular, many of our results explicitly require that $c$ is bipartite.

Proposition~\ref{is poly} says that $\bC(W,c)$ is the normal fan of a polytope, but does not guarantee that this polytope is simple (equivalently, that this fan is simplicial).
In fact, simpleness fails for the linear Coxeter element of $S_n$, and this failure can be seen already in $S_4$. 
(See Figure~\ref{A3 biCamb linear}, and also \cite[Figure~13]{rectangle}.  The latter shows the $1$-skeleton of this polytope disguised as the Hasse diagram of a certain lattice.)  
We conjecture that the situation is better in the bipartite case.

\begin{conj}\label{simple poly}
If $W$ is a bipartite Coxeter element, then $\bC(W,c)$ is a simplicial fan.
(Equivalently, its dual polytope is simple.)
\end{conj}

We have verified Conjecture~\ref{simple poly}, with the aid of Stembridge's packages~\cite{StembridgePackages}, up to rank 6.
Also, in Section~\ref{A B simp sec}, we prove the following theorem using alternating arc diagrams, by appealing to some results of \cite{DIRRT} linking the lattice theory of the weak order to the representation theory of finite-dimensional algebras, and then applying a folding argument.  
\begin{theorem}\label{simple A B}
Conjecture~\ref{simple poly} holds in types A and B.
\end{theorem}

In Section~\ref{twin sec}, we will prove the following theorem.

\begin{theorem}\label{simple h}
If Conjecture~\ref{simple poly} holds for a Coxeter group $W$, then the $h$-vector of the simplicial sphere underlying $\bC(W,c)$, for $c$ bipartite,  has entries $\biNar_k(W)$.
\end{theorem}

In light of the evidence for Conjecture~\ref{simple poly} and in light of Theorem~\ref{simple h}, we propose the term \newword{simplicial $W$-biassociahedron} for the polytope whose face fan is $\bC(W,c)$ \emph{for $c$ bipartite}, and \newword{simple $W$-biassociahedron} for the polytope whose normal fan is $\bC(W,c)$ \emph{for $c$ bipartite}.

\begin{remark}\label{A biassoc}
Theorems~\ref{biNar thm}, \ref{simple A B}, and~\ref{simple h} imply that the $A_n$-biassociahedron has the same $h$-vector as the $B_n$-associahedron (also known as the cyclohedron).
One is naturally led to ask whether these two polytopes are combinatorially isomorphic.  
The answer is no already for $n=3$.
The normal fan to the $A_3$-biassociahedron is shown in Figure~\ref{A3 biCamb bipartite}.
The dual graph to this fan has a vertex that is incident to two hexagons and a quadrilateral.
The graph of the $B_3$-associahedron (shown for example in \cite[Figure~3.9]{rsga}) has no such vertex.
\end{remark}

\subsection{The biCambrian congruence, twin sortable elements, and bisortable elements}\label{twin sec}  
A \newword{congruence} $\Theta$ on a lattice $L$ is an equivalence relation respecting the meet and join operations.
We now quote some combinatorial facts about lattice congruences.  
Proofs can be found in \cite[Section~9-5]{regions9}.
In this paper, we consider only finite lattices, and some results quoted in this section can fail for infinite lattices.
On a finite lattice, congruences are characterized by three properties:
congruence classes are intervals;
the projection \raisebox{0pt}[0pt][0pt]{$\pidown^\Theta$}, mapping each element to the bottom element of its congruence class, is order preserving;
and the projection \raisebox{0pt}[0pt][0pt]{$\piup_\Theta$}, mapping each element to the top element of its congruence class, is order preserving.
The $\Theta$-classes are exactly the fibers of $\pidown^{\Theta}$.  
The quotient $L/\Theta$ of a finite lattice $L$ modulo a congruence $\Theta$ is a lattice isomorphic to the subposet induced by the set \raisebox{0pt}[0pt][0pt]{$\pidown^\Theta(L)$} of elements that are the bottoms of their congruence classes.
The congruence $\Theta$ is determined by the set $\pidown^\Theta(L)$: Specifically $x\equiv y$ modulo $\Theta$ if and only if the unique maximal element of $\pidown^\Theta(L)$ below $x$ equals the unique maximal element of $\pidown^\Theta(L)$ below $y$.

The map $\pidown^{\Theta}$ is a lattice homomorphism from $L$ onto the subposet \raisebox{0pt}[0pt][0pt]{$\pidown^\Theta(L)$}, but care must be taken to avoid misinterpreting this fact.
Literally, the fact that $\pidown^{\Theta}$ is a lattice homomorphism means that for any $U\subseteq L$, we have $\pidown^{\Theta}(\Join U) = \Join_{x\in U} \pidown^{\Theta}(x)$ and $\pidown^{\Theta}(\Meet U) = \Meet_{x\in U} \pidown^{\Theta}(x)$, but in each identity, the join on the left side occurs in $L$ while the join on the right side occurs in $\pidown^\Theta(L)$.
It is easy to check that $\pidown^\Theta(L)$ is also a join-sublattice of $L$, so the distinction between the join in $L$ and the join in $\pidown^\Theta(L)$ is unnecessary.
However, in general, $\pidown^{\Theta}(L)$ need not be a meet-sublattice of $L$, so in interpreting the identity $\pidown^{\Theta}(\Meet U) = \Meet_{x\in U} \pidown^{\Theta}(x)$, it is crucial to be clear on where the meets occur.

The maximal cones of the Coxeter fan $\F(W)$, partially ordered according to a suitable linear functional, form a lattice isomorphic to the weak order on $W$.
(This fact is true either for the right or left weak order.  We will work with the right weak order.)  
Any lattice congruence $\Theta$ on the weak order on $W$ defines a fan $\F_\Theta(W)$ coarsening $\F(W)$.
(See \cite[Theorem~1.1]{con_app} and \cite[Section~5]{con_app}.)
Specifically, for each $\Theta$-class, the union of the corresponding maximal cones in $\F(W)$ is itself a convex cone, and the collection of all these convex cones and their faces is the fan $\F_\Theta(W)$.
Each Coxeter element $c$ specifies a congruence $\Theta_c$ on the weak order called the \newword{$c$-Cambrian congruence}.
(See \cite{cambrian} for the definition.)
The fan $\F_{\Theta_c}(W)$ is the $c$-Cambrian fan $\C(W,c)$ described earlier.

The set $\Con(L)$ of all congruences on a given lattice $L$ is itself a sublattice of the lattice of set partitions of $L$.
In particular, the meet of two congruences is the coarsest set partition of $L$ refining both congruences.   
We define the \newword{$c$-biCambrian congruence} to be the meet, in $\Con(W)$, of the Cambrian congruences $\Theta_c$ and $\Theta_{c^{-1}}$.
The fan $\F_\Theta(W)$ for $\Theta=\Theta_c\meet\Theta_{c^{-1}}$ is the coarsest common refinement of $\F(\Theta_c(W))$ and $\F(\Theta_{c^{-1}}(W))$.
Thus the $c$-biCambrian fan $\bC(W,c)$ is the fan $\F_\Theta(W)$ for $\Theta=\Theta_c\meet\Theta_{c^{-1}}$.
In particular, the $c$-biCambrian congruence classes are in bijection with the maximal cones of $\bC(W,c)$.
We define the \newword{$c$-biCambrian lattice} to be the quotient of the weak order modulo the $c$-biCambrian congruence.
The elements of the $c$-biCambrian lattice are thus in bijection with the maximal cones of $\bC(W,c)$.

We write \raisebox{0pt}[0pt][0pt]{$\pidown^c$} for the projection taking each element of $W$ to the bottom element of its $c$-Cambrian congruence class, and similarly $\pidown^{c^{-1}}$.
(That is, $\pidown^c$ stands for $\pidown^\Theta$ where $\Theta=\Theta_c$.)
Consider the map that sends each $c$-biCambrian congruence class to the pair $(\pidown^c(w),\pidown^{c^{-1}}(w))$, where $w$ is any representative of the class.
Because the $c$-biCambrian congruence $\Theta$ is the meet $\Theta_c\meet \Theta_{c^{-1}}$, two elements $u$ and $v$ are congruent in the $c$-biCambrian congruence if and only if \raisebox{0pt}[0pt][0pt]{$\pidown^c(u)=\pidown^c(v)$} and \raisebox{0pt}[0pt][0pt]{$\pidown^{c^{-1}}(u)=\pidown^{c^{-1}}(v)$}.
Thus, the map from classes to pairs is a well-defined bijection from $c$-biCambrian congruence classes to its image.

The bottom elements of the $c$-Cambrian congruence are called \newword{$c$-sortable elements}.
(In fact $c$-sortable elements have an independent combinatorial definition \cite[Section~2]{sortable}, but were shown to be the bottom elements of $c$-Cambrian congruences in \cite[Theorems 1.1 and 1.4]{sort_camb}.)  
Given elements $u$ and $v$ of $W$, we define the pair $(u,v)$ to be a pair of \newword{twin $(c,c^{-1})$-sortable elements} of $W$ if there exists $w\in W$ such that $u=\pidown^c(w)$ and $v=\pidown^{c^{-1}}(w)$.
The map considered in the previous paragraph is a bijection between $c$-biCambrian congruence classes and pairs of twin $(c,c^{-1})$-sortable elements of $W$.
The twin sortable elements are similar in spirit to the twin binary trees of \cite{DGStack}, which were already mentioned in connection with Theorem~\ref{Baxter thm}.
Indeed, for $W$ of type~A and $c$ linear, the connection is implicit in the construction in \cite{rectangle} of a diagonal rectangulation from a pair of binary trees.
(See also \cite[Remark~6.6]{rectangle}.)
Also in type A, but for general $c$, the twin binary trees are generalized in~\cite{ChaPil} to \newword{twin Cambrian trees}, which correspond explicitly to pairs of twin $(c,c^{-1})$-sortable elements. 
Indeed, \cite[Proposition~36]{ChaPil} amounts to another computation of the type-A biCatalan number, quite different from the two given here (in Sections~\ref{nn sec} and~\ref{alt_arc_diagrams}).

Another set of objects naturally in bijection with $c$-biCambrian congruence classes are the bottom elements of $c$-biCambrian congruence classes.
We coin the term \newword{$c$-bisortable} elements for these bottom elements.
Although the $c$-sortable elements have a direct combinatorial characterization \cite[Section~2]{sortable}, we currently have no direct combinatorial characterization of $c$-bisortable elements.
We do offer the following indirect characterization of $c$-bisortable elements in terms of $c$-sortable elements and $c^{-1}$-sortable elements.
\begin{prop}\label{c join cinv}
For any $c$, an element $w\in W$ is $c$-bisortable if and only if there exists a $c$-sortable element $u$ and a $c^{-1}$-sortable element $v$ such that $w=u\join v$ in the weak order.
When $w$ is $c$-bisortable, we can take $u=\pidown^c(w)$ and $v=\pidown^{c^{-1}}(w)$.
\end{prop}
\begin{proof}
Given $c$-bisortable $w$, take $u=\pidown^c(w)$ and $v=\pidown^{c^{-1}}(w)$.
Then $u\le w$ and $v\le w$.
Since Cambrian congruence classes are intervals, any upper bound $w'$ for $u$ and $v$ with $w'\le w$ is congruent to $u$ modulo $\Theta_c$ and congruent to $v$ modulo $\Theta_{c^{-1}}$.
Thus $w'$ is congruent to $w$ in the $c$-biCambrian congruence.
Since $w$ is the bottom element of its $c$-biCambrian congruence class, we conclude that $w'=w$.
We have shown that $w=u\join v$.

Suppose $w=u\join v$ for some $c$-sortable element $u$ and some $c^{-1}$-sortable element $v$.
Since $\pidown^c(w)$ is the unique maximal $c$-sortable element below $w$, we have $\pidown^c(w)\ge u$.
Similarly, $\pidown^{c^{-1}}(w)\ge v$.
If there exists $w'<w$ in the same $c$-biCambrian congruence class as $w$, then $w'\ge\pidown^c(w')=\pidown^c(w)\ge u$ and $w'\ge\pidown^{c^{-1}}(w')=\pidown^{c^{-1}}(w)\ge v$.
This contradicts the fact that $w=u\join v$, and we conclude that $w$ is $c$-bisortable.
\end{proof}

Recall that for any congruence $\Theta$ on a finite lattice $L$, the set $\pidown^\Theta(L)$ is a join-sublattice of $L$.
The Cambrian congruences have a stronger property:
For any Coxeter element $c$, the $c$-sortable elements constitute a sublattice \cite[Theorem~1.2]{sort_camb} of the weak order on $W$.
It is natural to ask whether the same is true for $c$-bisortable elements, but the answer is no.
We give an example for $W=S_5$ and bipartite $c$:
The permutations $45312$ and $53142$ are both $c$-bisortable but their meet $31452$ is not.
(To check this example, Proposition~\ref{avoidance} will be very helpful.)

Each $c$-bisortable element $v$ covers some number of elements in the $c$-biCambrian lattice.
By a general fact on lattice quotients (see for example \cite[Proposition~6.4]{shardint}), $v$ covers the same number of elements in the weak order on $W$.
This number is $\des(v)$, the number of descents of $v$.  
(We will define descents in Section~\ref{sortable formula sec}.)
The \newword{descent generating function of $c$-bisortable elements} is the sum $\sum x^{\des(v)}$ over all $c$-bisortable elements $v$.
For bipartite $c$, its coefficients are the $W$-biNarayana numbers.
A general fact about lattice quotients of the weak order \cite[Proposition~3.5]{con_app} implies that, when $\bC(W,c)$ is simplicial, the descent generating function of $c$-bisortable elements equals the $h$-polynomial of $\bC(W,c)$.
In the bipartite case, Theorem~\ref{simple h} follows immediately. 

\subsection{Twin clusters and bicluster fans}\label{clus sec}
Clusters of almost positive roots were introduced in \cite{ga}, where they were used to define generalized associahedra.
In \cite{ca2}, clusters of almost positive roots were used to model cluster algebras of finite type.
Here, we will not need the cluster-algebraic background, which can be found in \cite{ca2}.
Instead, we define almost positive roots and $c$-compatibility and quote some results about $c$-clusters and their relationship to $c$-sortable elements.
We will also not need the more refined notion of ``compatibility degree.''

In a finite root system, the \newword{almost positive roots} are those roots which either are positive, or are the negatives of simple roots.  
The definition of compatibility in \cite{ga} is a special case (namely the bipartite case) of what we here call $c$-compatibility.
The general definition was given in \cite{MRZ}, but here we give a rephrasing found in \cite[Section~7]{sortable}, translated into the language of almost positive roots.

We write $\set{\alpha_1,\ldots,\alpha_n}$ for the simple roots and $\set{s_1,\dots,s_n}$ for the simple reflections.
For each $i$ in $\set{1,\ldots,n}$, we define an involution $\sigma_i$ on the set of almost positive roots by
\begin{equation}\label{sigma def}
\sigma_i(\beta):=\left\lbrace\begin{array}{ll}
\beta&\text{if }\beta=-\alpha_j\text{ with }j\neq i,\mbox{ or}\\
s_i\beta&\text{otherwise}.
\end{array}\right.
\end{equation}
We write $[\beta:\alpha_i]$ for the coefficient of $\alpha_i$ in the expansion of $\beta$ in the basis of simple roots.
A simple reflection $s_i$ is \newword{initial} in a Coxeter element $c$ if $c$ has a reduced word starting with $s_i$.
If $s_i$ is initial in $c$, then $s_ics_i$ is another Coxeter element.

The \newword{$c$-compatibility} relations are a family of symmetric binary relations $\cm_c$ on the almost positive roots.
They are the unique family of relations with 
\begin{enumerate}
\item[(i) ]For any $i$ in $\set{1,\ldots,n}$, and Coxeter element~$c$,
\[-\alpha_i\cm_c\beta\mbox{ if and only if }[\beta:\alpha_i]=0.\]
\item[(ii) ]For each pair of almost positive roots $\beta_1$ and $\beta_2$, each Coxeter element $c$, and each $s_i$ initial in $c$,   
\[\beta_1\cm_c \beta_2\mbox{ if and only if }\sigma_i(\beta_1)\,\cm_{s_ics_i}\,\sigma_i(\beta_2).\]
\end{enumerate}

The \newword{$c$-clusters} are the maximal sets of pairwise $c$-compatible almost positive roots.
By \cite[Theorem~1.8]{ga} and \cite[Proposition~3.5]{MRZ}, for fixed $W$, all $c$-clusters are of the same size, and furthermore, each is a basis for the root space (the span of the roots).
Write $\reals_{\ge0} C$ for the nonnegative linear span of a $c$-cluster $C$.
Then \cite[Theorem~1.10]{ga} and \cite[Theorem~3.7]{MRZ} state that the cones $\reals_{\ge0} C$, for all $c$-clusters $C$, are the maximal cones of a complete simplicial fan.
We call this fan the \newword{$c$-cluster fan}.

We define the \newword{$c$-bicluster fan} to be the coarsest common refinement of the $c$-cluster fan and its antipodal opposite.
A pair $(C_1,C_2)$ of $c$-clusters is called a pair of \newword{twin $c$-clusters} if the cones $\reals_{\ge0} C_1$ and $-\reals_{\ge0} C_2$ (the nonpositive linear span of $C_2$) intersect in a full-dimensional cone.
It is immediate that maximal cones in the $c$-bicluster fan are in bijection with pairs of twin $c$-clusters.

\begin{example}\label{A3 biclus example}  
For $W$ of type $A_3$, up to symmetry there are two different $c$-bicluster fans:  one for linear $c$ and one for bipartite $c$, shown in Figures~\ref{A3 biclus linear} and~\ref{A3 biclus bipartite} respectively.
These are again stereographic projections as explained in Example~\ref{A3 biCamb example}.
\end{example}
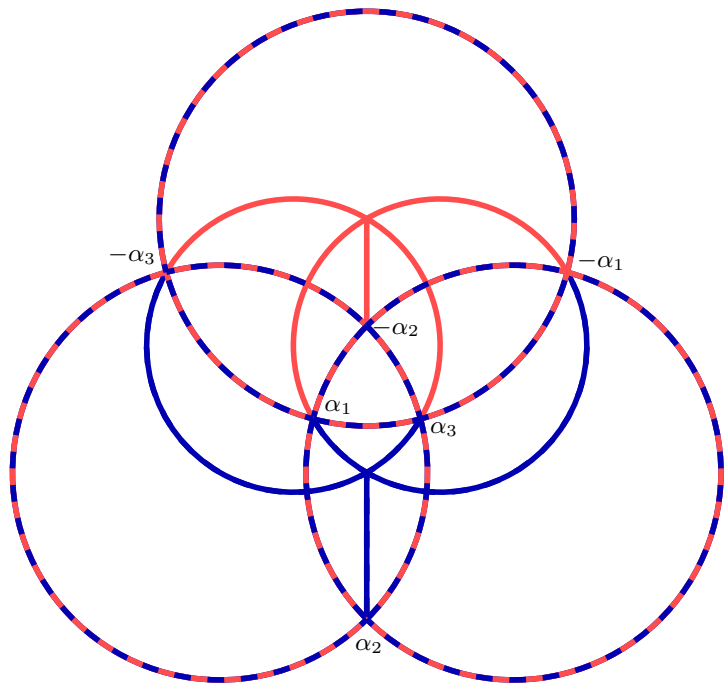
\begin{figure}[p]
\scalebox{1.05}{\includegraphics{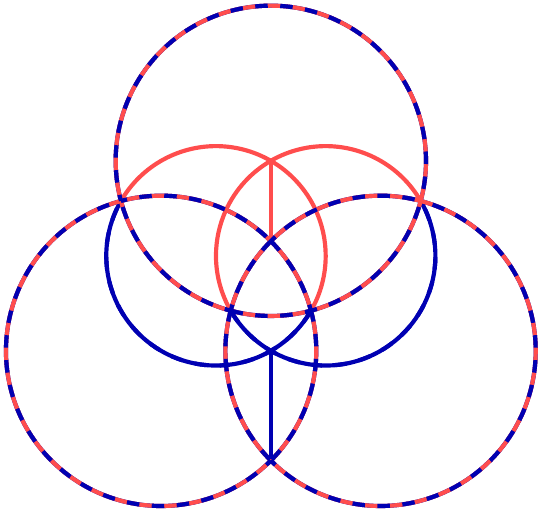}
\begin{picture}(0,0)(130,-135)
\put(-19,-36){\small$\alpha_1$}
\put(72,16){\small$-\alpha_1$}
\put(-8,-122){\small$\alpha_2$}
\put(-2,-8){\small$-\alpha_2$}
\put(19,-44){\small$\alpha_3$}
\put(-97,18){\small$-\alpha_3$}
\end{picture}
}
\caption{The linear bicluster fan in type $A_3$}
\label{A3 biclus linear}
\end{figure}
\begin{figure}[p]
\scalebox{1.05}{\includegraphics{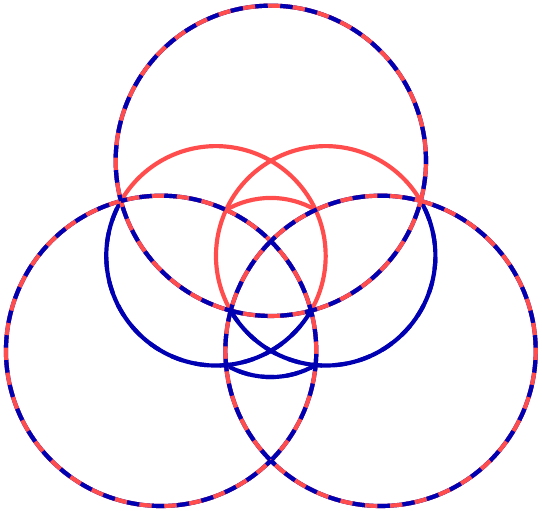}
\begin{picture}(0,0)(130,-135)
\put(-19,-36){\small$\alpha_1$}
\put(72,16){\small$-\alpha_1$}
\put(-8,-122){\small$\alpha_2$}
\put(-2,-8){\small$-\alpha_2$}
\put(19,-44){\small$\alpha_3$}
\put(-97,18){\small$-\alpha_3$}
\end{picture}
}
\caption{The bipartite bicluster fan in type $A_3$}
\label{A3 biclus bipartite}
\end{figure}
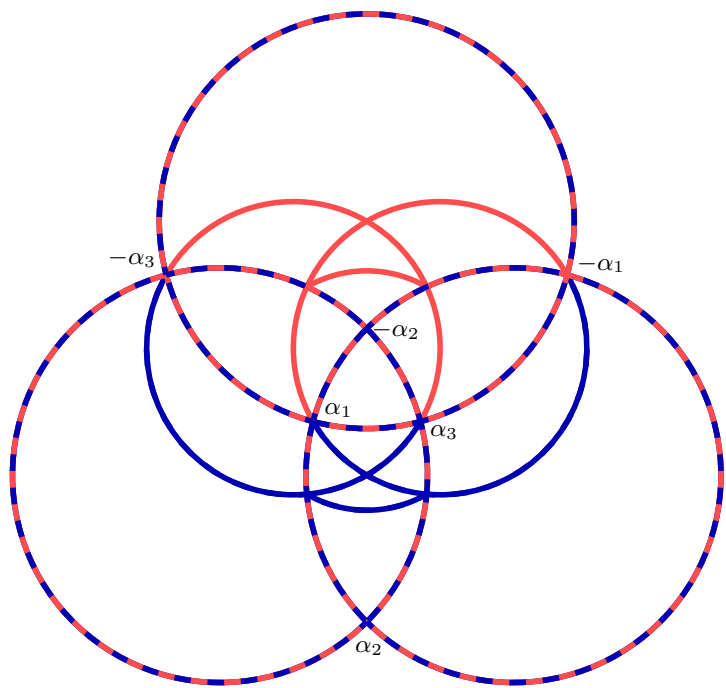

The two $c$-bicluster fans in Example~\ref{A3 biclus example} are combinatorially isomorphic.
Despite this tantalizing fact, in this paper, we only explore bicluster fans in the special case of bipartite Coxeter elements (the original setting of \cite{ga,ca2}), where they are easily related to biCambrian fans.  
For the bipartite choice of $c$, \cite[Theorem~9.1]{camb_fan} says that the $c$-Cambrian fan is \emph{linearly} isomorphic to the cluster fan.
Combining this fact with Proposition~\ref{biCamb alt def}, we have the following theorem.
\begin{theorem}\label{bi cl fans}
For all finite Coxeter groups $W$ and for bipartite $c$, the $c$-bicluster fan is linearly isomorphic to the $c$-biCambrian fan.
\end{theorem}

\begin{remark}
We emphasize that Theorem~\ref{bi cl fans} requires the hypothesis that $c$ is bipartite.
In contrast, when $W$ is of type $A_3$ and $c$ is the linear Coxeter element, the $c$-bicluster fan and the $c$-biCambrian fan don't even have the same number of regions.
\end{remark}

Because of the bijection between $c$-bisortable elements and maximal cones in $\bC(W,c)$ and the bijection between maximal cones in the $c$-bicluster fan and pairs of twin $c$-clusters, we have the following immediate consequence of Theorem~\ref{bi cl fans}.

\begin{theorem}\label{bi cl}
For all finite Coxeter groups $W$, $c$-bisortable elements for \emph{bipartite} $c$ are in bijection with pairs of twin $c$-clusters.
\end{theorem}

Combining Theorems~\ref{simple h} and~\ref{bi cl fans}, we obtain the following theorem.
\begin{theorem}\label{simple h cl}
If Conjecture~\ref{simple poly} holds for a Coxeter group $W$, then the bipartite $c$-bicluster fan is simplicial and the $h$-vector of the underlying simplicial sphere has entries $\biNar_k(W)$.
\end{theorem}

\subsection{Twin noncrossing partitions}\label{nc sec}  
The absolute order on a finite Coxeter group $W$ is the prefix order (or equivalently the subword order) on $W$ relative to the generating set $T$, the set of reflections in $W$.
(By contrast, the prefix order relative to the simple reflections $S$ is the weak order, while the subword order relative to $S$ is the Bruhat order.)
We will use the symbol $\le_T$ for the absolute order.
The \newword{$c$-noncrossing partitions} in a finite Coxeter group $W$ are the elements of $W$ contained in the interval $[1,c]_T$ in the absolute order on $W$.
For details on the absolute order and noncrossing partitions, see for example \cite[Chapter~2]{Armstrong}.
For our purposes, the key fact is a theorem of Brady and Watt.

Let $W$ be a finite Coxeter group of rank $n$ represented as a reflection group in $\reals^n$ and let $T$ be the set of reflections of $W$.
For each reflection $t\in T$, let $\beta_t$ be the corresponding positive root.
Given $w\in[1,c]_T$, define a cone 
\[F_c(w)=\set{\x\in\reals^n:\x\cdot\beta_t\le0\,\,\forall\,t\le_T w,\,\,\x\cdot\beta_t\geq0\,\,\forall\,t\le_T cw^{-1}}.\]
The following theorem combines \cite[Theorem~1.1]{BWassoc} with  \cite[Theorem~5.5]{BWassoc}.
\begin{theorem}\label{BWthm}
For $c$ bipartite, the map $F_c$ is a bijection from $[1,c]_T$ to the set of maximal cones in the $c$-Cambrian fan.
\end{theorem}
The astute reader will notice a difference between our definition of $F_c$ and the definition appearing in \cite[Section~1]{BWassoc}.
The set of reflections $t$ such that $t\le_Tw$ is the intersection of $T$ with some (not necessarily standard) parabolic subgroup of $W$.
The definition in \cite{BWassoc} imposes inequalities $\x\cdot\beta_t\le0$ only for those $\beta_t$ that are simple roots for that parabolic subgroup.
Our definition imposes additional inequalities, all of which are implied by the inequalities for the simple roots.
We similarly add additional redundant inequalities of the form $x\cdot\beta_t\geq0$.

Theorem~\ref{BWthm} suggests a definition of twin noncrossing partitions.
In fact, given Proposition~\ref{biCamb alt def}, two natural definitions suggest themselves.
Given $u,v\in[1,c]_T$, we call $(u,v)$ a pair of \newword{twin $c$-noncrossing partitions} if $F_c(u)\cap(-F_c(v))$ is full-dimensional.
Similarly, given $u\in[1,c]_T$ and $v\in[1,c^{-1}]_T$, we call $(u,v)$ a pair of \newword{twin $(c,c^{-1})$-noncrossing partitions} if $F_c(u)\cap F_{c^{-1}}(v)$ is full-dimensional.
Theorem~\ref{BWthm} now immediately implies the following theorem.
\begin{theorem}\label{bi nc}
For all $W$ and bipartite $c$, the $c$-bisortable elements are in bijection with pairs of twin $c$-noncrossing partitions and with pairs of twin $(c,c^{-1})$-noncrossing partitions.
\end{theorem}

\section{Bipartite $c$-bisortable elements and alternating arc diagrams}\label{type A sec}
In this section, we show how bipartite $c$-bisortable elements of type A are in bijection with certain objects called alternating arc diagrams.
We then prove the type-A enumeration of bipartite $c$-bisortable elements in Theorem~\ref{main thm} by counting alternating arc diagrams and prove the type-B enumeration by counting centrally symmetric alternating arc diagrams.

\subsection{Pattern avoidance}\label{pat av sec} 
The Coxeter group of type $A_n$ is the symmetric group $S_{n+1}$. 
We will write permutations $x$ in $S_{n+1}$ in their one-line notations $x_1\cdots x_{n+1}$.
In the weak order on permutations in $S_{n+1}$, there is a cover $x_1\cdots x_{n+1}\covered y_1\cdots y_{n+1}$ if and only if there exists $i$ such that $y_i=x_{i+1}>x_i=y_{i+1}$ and $y_j=x_j$ for $j\not\in\set{i,i+1}$.
We say that $x$ is covered by $y$ via a swap in positions $i$ and $i+1$.

The Cambrian congruences on $S_{n+1}$ are described in detail in \cite{cambrian}.
We quote part of the description here.
The simple generator $s_i$ for $A_n$ is the transposition $(i\,\,\,i\!+\!1)$, for $i = 1, 2, \ldots n$.
Each Coxeter element $c$ can be encoded by a coloring of the elements $2,\ldots,n$ that we call a \newword{barring}.  
Each element $i$ is either \newword{overbarred} and marked $\overline{i}$ if $s_i$ occurs before $s_{i-1}$ in every reduced word for $c$, or \newword{underbarred} and marked~$\underline{i}$ if $s_i$ occurs after $s_{i-1}$ in every reduced word for $c$.
Passing from $c$ to $c^{-1}$ means swapping overbarring with underbarring.  

We say $x$ is obtained from $y$ by a \newword{$\overline231\to\overline213$ move} if $x$ is covered by $y$ via a swap in positions $i$ and $i+1$, for some $i$, and if there exists an \emph{overbarred} element $x_j$ with $j<i$ and $x_i<x_j<x_{i+1}$.  
Similarly, $x$ is obtained from $y$ by a \newword{$31\underline2\to13\underline2$ move} if $x$ is covered by $y$ via a swap in positions $i$ and $i+1$, for some $i$, and if there exists an \emph{underbarred} element $x_j$ with $i+1<j$ and $x_i<x_j<x_{i+1}$.
Combining \cite[Proposition~5.3]{cambrian} and \cite[Theorem~6.2]{cambrian}, we obtain the following proposition: 
\begin{prop}\label{same class}
Suppose $x$ and $y$ are permutations in $S_{n+1}$ with $x\covered y$ in the weak order, and assume that the numbers $2,\ldots,n$ have been barred according to~$c$.  
Then $x$ and $y$ are in the same $c$-Cambrian congruence class if and only if $x$ is obtained from $y$ by a $\overline231\to\overline213$ move or a $31\underline2\to13\underline2$ move.
\end{prop}

As an immediate corollary, we see that a permutation $y$ is the bottom element of its $c$-Cambrian congruence class (i.e.\ is $c$-sortable) if and only if none of the permutations covered by $y$ are obtained from $y$ by a $\overline231\to\overline213$ move or a $31\underline2\to13\underline2$ move.
In other words, there is no subsequence $\overline bca$ of $y$ with $a<b<c$, with $c$ immediately preceding $a$, and with $b$ overbarred and no subsequence $ca\underline b$ of $y$ with $a<b<c$, with $c$ immediately preceding $a$, and with $b$ underbarred.
In this case, we say that $y$ \newword{avoids} $\overline231$ and $31\underline2$.

We can similarly describe bottom elements of $c$-biCambrian congruence classes (the $c$-bisortable elements), keeping in mind that passing from $c$ to $c^{-1}$ means swapping overbarring with underbarring:
An element $y$ is the bottom element of its $c$-biCambrian congruence class if and only if none of the permutations covered by $y$ are obtained from $y$ by a $\overline231\to\overline213$ or $31\underline2\to13\underline2$ move that is \emph{also} a $\underline231\to\underline213$ or $31\overline2\to13\overline2$ move.
(Compare \cite[Remark~34]{ChaPil}.)  
For $c$ linear, the $c$-bisortable permutations are the twisted Baxter permutations of \cite[Section~4.2]{rectangle}.
In general, $c$-bisortable permutations may be described by a complicated pattern-avoidance condition that we will only describe, in Propositions~\ref{avoidance} and~\ref{avoidance bivinc}, for the case of bipartite $c$, where it becomes much simpler.

\subsection{Noncrossing arc diagrams}\label{arc sec}  
We now review the notion of \newword{noncrossing arc diagrams} from~\cite{arcs}.
Beginning with $n+1$ distinct points on a vertical line, numbered $1,\ldots,n+1$ from bottom to top, we draw some (or no) curves called \newword{arcs} connecting the points.
Each arc moves monotone upwards from one of the points to another, passing either to the left or to the right of each point in between.
Furthermore no two arcs may intersect in their interiors, no two arcs share the same upper endpoint, and no two arcs may share the same lower endpoint.
We consider arc diagrams only up to their combinatorics, i.e.\ which pairs of points are joined by an arc and which points are left and right of each arc.

Given a permutation $x_1\cdots x_{n+1}$ in $S_{n+1}$, we define a noncrossing arc diagram $\delta(x_1\cdots x_{n+1})$.
Each descent $x_i>x_{i+1}$ becomes an arc $\alpha$ in $\delta(x_1\cdots x_{n+1})$ with lower endpoint $x_{i+1}$ and upper endpoint $x_i$.
For each integer $j$ with $x_{i+1}<j<x_i$ that occurs to the \emph{left} of $x_i$ in $x_1\cdots x_{n+1}$, the point $j$ is \emph{left} of the arc $\alpha$.
For each integer $j$ with $x_{i+1}<j<x_i$ that occurs to the \emph{right} of $x_{i+1}$ in $x_1\cdots x_{n+1}$, the point $j$ is \emph{right} of the arc $\alpha$.
It was shown in \cite[Theorem~3.1]{arcs} that $\delta$ is a bijection from permutations to noncrossing arc diagrams.  
More specifically, for each $k$, the map $\delta$ restricts to a bijection from permutations with $k$ descents to noncrossing arc diagrams with $k$ arcs.

A \newword{$c$-sortable arc} is an arc that belongs to $\delta(v)$ for some $c$-sortable permutation~$v$.
The following characterization of $c$-sortable arcs in terms of the barring associated to $c$ is immediate from the pattern-avoidance description above.  
(Compare \cite[Example~4.9]{arcs}.)
\begin{prop}\label{csort arc}
For $W=A_n$ and any $c$, the $c$-sortable arcs are the arcs that do not pass to the left of any underbarred element of $\set{2,\ldots,n}$ and do not pass to the right of any overbarred element of $\set{2,\ldots,n}$.
\end{prop}
In particular, since $c$ and $c^{-1}$ correspond to opposite barrings, the only arcs that are both $c$ and $c^{-1}$-sortable are the arcs that connect adjacent endpoints $i$ and $i+1$.
(This is a restatement of the type-A case of Remark~\ref{common walls} in terms of noncrossing arc diagrams.)

Combining the above descriptions of $c$-sortable and $c$-bisortable elements in terms of overbarred and underbarred elements, we obtain the following proposition.

\begin{proposition}\label{c-bisort arcs} 
For $W=A_n$ and any $c$, the map $\delta$ restricts to a bijection from $c$-bisortable permutations with $k$ descents to noncrossing arc diagrams on $n+1$ vertices with $k$ arcs, each of which is either $c$ or $c^{-1}$-sortable.
\end{proposition}
\begin{proof}
Suppose $x=x_1\cdots x_n$ is a permutation such that $\delta(x)$ has an arc that is neither $c$-sortable nor $c^{-1}$-sortable.
This arc has upper endpoint $x_i$ and lower endpoint $x_{i+1}$ for some $i$ and it fails the conclusion of Proposition~\ref{csort arc} for $c$ and for $c^{-1}$.
That is, it either passes left of an underbarred element or right of an overbarred element \emph{and} it either passes left of an overbarred element or right of an underbarred element.
Thus, switching $x_i$ with $x_{i+1}$ is both a $\overline231\to\overline213$ or $31\underline2\to13\underline2$ move \emph{and} a $\underline231\to\underline213$ or $31\overline2\to13\overline2$ move.
Therefore, $x$ is not $c$-bisortable.
The argument is easily reversed to prove the converse.
\end{proof}
Alternately, Proposition~\ref{c-bisort arcs} follows from the description of the $c$-biCambrian congruence as the meet of the $c$-Cambrian and $c^{-1}$-Cambrian congruences.

\subsection{Alternating arc diagrams}\label{alt sec}
We now consider the case where $c$ is bipartite.
Let $c_{+}$ be the product of the simple generators $s_i$ where $i$ is even, and $c_{-}$ be the product of the simple generators $s_i$ where $i$ is odd.
The bipartite Coxeter elements in $A_n$ are $c_+c_-$ and its inverse $c_-c_+$.
The barring of the numbers $2,\ldots, n$ associated to $c_+c_-$ has all even numbers overbarred and all odd numbers underbarred.
A \newword{right-even alternating arc} is an arc that passes to the right of even vertices and to the left of odd vertices.
A \newword{left-even alternating arc} is an arc that passes to the left of even vertices and to the right of odd vertices.
A \newword{right-even alternating arc diagram} is a noncrossing arc diagram all of whose arcs are right-even alternating, and \newword{left-even alternating arc diagrams} are defined analogously.
The following proposition is an immediate consequence of Proposition~\ref{csort arc}.

\begin{proposition}\label{c-arcs}  
Suppose $W=A_n$ and $c$ is the bipartite Coxeter element $c_+c_-$.
\begin{enumerate}
\item The map $\delta$ restricts to a bijection from $c$-sortable permutations to right-even alternating arc diagrams.
\item The map $\delta$ restricts to a bijection from $c^{-1}$-sortable permutations to left-even alternating arc diagrams.
\end{enumerate}
In each case, $\delta$ restricts further to send permutations with $k$ descents bijectively to arc diagrams with $k$ arcs.
\end{proposition}

An \newword{alternating arc} is an arc that is either right-even alternating or left-even alternating or both.
We call a noncrossing arc diagram consisting of alternating arcs an \newword{alternating arc diagram}. 
Figure~\ref{some} shows several alternating noncrossing arc diagrams.
From left to right, they are $\delta(5371624)$, $\delta(4631275)$, and $\delta(4275136)$.  
\begin{figure}
\includegraphics{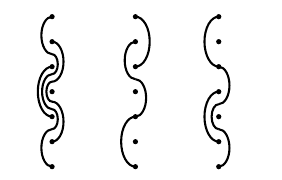}
\caption{Some alternating noncrossing arc diagrams}
\label{some}
\end{figure}
The following proposition is the bipartite case of Proposition~\ref{c-bisort arcs}.

\begin{prop}\label{avoid alt}
For $W=A_n$ and $c$ bipartite, the map $\delta$ restricts to a bijection from $c$-bisortable permutations with $k$ descents to alternating arc diagrams on $n+1$ points with $k$ arcs.
\end{prop}
Observe that an arc fails to be alternating if and only if it passes on the same side of two consecutive numbers.
Thus, we obtain the following simpler description of the pattern avoidance condition defining bipartite $c$-bisortable elements.
(Compare \cite[Remark~34]{ChaPil}.)  

\begin{prop}\label{avoidance}
If $c$ is the bipartite Coxeter element $c_+c_-$ of $A_n$, a permutation $x=x_1\cdots x_{n+1}$ is $c$-bisortable if and only if, for every descent $x_i>x_{i+1}$, there exists no $k$ with $x_{i+1}<k<k+1<x_i$ such that $k$ and $k+1$ are on the same side of the descent (i.e.\ $k$ and $k+1$ both left of $x_i$ or both right of $x_{i+1}$).
\end{prop}

The condition in Proposition~\ref{avoidance} is that $x$ avoids subsequences $dabc$, $dacb$, $bcda$, and $cbda$ with $a<b<c<d$, with $d$ and $a$ adjacent in \emph{position}, and with $b$ and $c$ being adjacent in \emph{value}.  
This is an instance of \newword{bivincular} pattern avoidance in the sense of \cite[Section~2]{BCDK}.
We will not review the notation for bivincular patterns from \cite{BCDK}, but we restate Proposition~\ref{avoidance} in that notation as follows:

\begin{prop}\label{avoidance bivinc}
For $c$ bipartite, a permutation is $c$-bisortable if and only if it avoids the bivincular patterns $(2341,\set{3},\set{2})$, $(3241,\set{3},\set{2})$, $(4123,\set{1},\set{2})$, and $(4132,\set{1},\set{2})$.
\end{prop}

\subsection{Counting alternating arc diagrams}\label{alt_arc_diagrams}
Let $[n]$ denote the set $\{1,2,\ldots , n\}$.
To prove the type-A enumeration of bipartite $c$-bisortable elements in Theorem~\ref{main thm}, we give a bijection $\pi$ from noncrossing alternating arc diagrams on $n+1$ vertices with $k$ arcs to pairs $(S,T)$ of subsets of $[n]$ with $|S|=|T|=k$.

Suppose that $\Sigma$ is an alternating arc diagram.
Whenever we encounter a right-even alternating arc in $\Sigma$ with endpoints $i<j$, we put $i$ into $S$ and $j-1$ into $T$; whenever we encounter a left-even alternating arc with endpoints $i<j$ we put $j-1$ into $S$ and $i$ into $T$.  
More precisely, suppose that $\Sigma$ is an alternating arc diagram with $k$ arcs.
Let $S'$ denote the set of numbers $i$ such that $i$ is bottom endpoint of a right-even alternating arc in $\Sigma$ and let $S''$ denote the set of numbers $j-1$ such that $j$ is the top endpoint of a left-even alternating arc in $\Sigma$.
Let $T'$ denote the set of numbers $j'-1$ such that $j'$ is the top endpoint of a right-even alternating arc in $\Sigma$ and let $T''$ denote the set of numbers $i'$ such that $i'$ is the bottom endpoint of a left-even alternating arc.
The map $\pi$ sends $\Sigma$ to the pair $(S'\cup S'', T'\cup T'')$.  

\begin{theorem}\label{alt arc bij}
The map $\pi$ is a bijection from the set of alternating arc diagrams on $n+1$ points to the set of pairs of subsets of $[n]$ of the same size.
For each $k$, the bijection restricts to a bijection from alternating arc diagrams with $k$ arcs to pairs of subsets of size $k$.
\end{theorem}

In preparation for the proof of Theorem~\ref{alt arc bij}, we will break each alternating diagram into smaller pieces.
Two alternating arcs with endpoints $i<j$ and $i'<j'$ \newword{overlap} if the intersection of the sets $\{i, \ldots, j-1\}$ and $\{i',\ldots, j'-1\}$ is nonempty.  
Informally, the arcs overlap if some part of one arc passes alongside of the other arc.
(If they only touch at their endpoints but don't pass alongside one another, then they do \emph{not} overlap). 
Given a collection $\E$ of arcs, we can define an ``overlap graph'' with vertices $\E$ and edges given by overlapping pairs in $\E$. 
We say that the collection $\E$ is \newword{overlapping} if this overlap graph is connected.
Each noncrossing arc diagram can be broken into overlapping components, maximal overlapping collections of arcs.
The definition of alternating arc diagrams and the definition of right-even and left-even alternating  arcs let us immediately conclude that two distinct arcs appearing in the same alternating arc diagram, one right-even alternating and one left-even alternating, cannot overlap.
We have proved the following fact.

\begin{prop}\label{overlap comp}
Each overlapping component of an alternating arc diagram fits exactly one of the following descriptions: (1) It consists of right-even alternating arcs that are not left-even alternating; (2) It consists of left-even alternating arcs that are not right-even alternating; or (3) it consists of a single arc that is right-even and left-even alternating (and thus connects two adjacent points).
\end{prop}

Proposition~\ref{overlap comp} implies that, on each overlapping component, the map $\pi$ collects all of the top endpoints of the arcs into one set, and all of the bottom endpoints into the other set.

Now we describe how to break an alternating diagram $\Sigma$ into its overlapping components.
Let $P(\Sigma)$ be the set of numbers $p\in[n+1]$ such that no arc in $\Sigma$ passes left or right of $p$.
(A point $p\in P(\Sigma)$ may still be an endpoint of one or two arcs.)
Write $P(\Sigma)=\set{p_0,\ldots,p_m}$ with $p_0<\cdots<p_m$.
In every case, $p_0=1$ and $p_m=n+1$.
For each $i$, we claim that an arc in $\Sigma$ has its lower endpoint in $\{p_{i-1}, p_{i-1}+1,\ldots, p_{i}-1\}$ if and only if it has its upper endpoint in $\{p_{i-1}+1, p_{i-1}+2,\ldots, p_i\}$.
Indeed, if an arc has a lower endpoint in $\{p_{i-1},p_{i-1}+1,\ldots, p_{i}-1\}$, then since it cannot pass on either side of $p_i$, it must end at a number in the set $\{p_{i-1}+1, p_{i-1}+2,\ldots, p_i\}$.
A similar argument proves the converse, so we have established the claim.
Let $\Sigma_i$ denote the set of arcs with lower endpoints in  $\{p_{i-1},p_{i-1}+1,\ldots, p_{i}-1\}$ (and thus with upper endpoints in  $\{p_{i-1}+1, p_{i-1}+2,\ldots, p_i\}$).  
By construction, $\Sigma_i$ is an overlapping component, and all overlapping components are $\Sigma_i$ for some $i$.
Let $(S_i,T_i)$ be the image of $\Sigma_i$ under $\pi$, so that $\pi(\Sigma) = (\bigcup_{i=1}^mS_i\,,\,\bigcup_{i=1}^mT_i)$.

We say that two arcs are \newword{compatible} if there is a noncrossing arc diagram containing both arcs.
Our next task is to understand for which pairs $(s,t)$ and $(s',t')$ there exists an overlapping pair of compatible \emph{alternating arcs}, one with endpoints $s$ and $t+1$ and one with endpoints $s'$ and $t'+1$.
Since the arcs must overlap but may not share the same bottom endpoint and may not share the same top endpoint, and taking without loss of generality $s<s'$, there are only two cases.
These cases are covered by the following two lemmas, which are easily verified.

\begin{lemma}\label{arc compat 1}
Suppose $s<s'\le t<t'$.
Then there exist two compatible alternating arcs, one with endpoints $s$ and $t+1$ and one with endpoints $s'$ and $t'+1$ if and only if $s'$ and $t$ have the same parity.
The pair of arcs can be chosen in exactly two ways, either both as right-even alternating arcs or both as left-even alternating arcs.
\end{lemma}
\begin{lemma}\label{arc compat 2}
Suppose $s<s'< t'<t$.
Then there exist two compatible alternating arcs, one with endpoints $s$ and $t+1$ and one with endpoints $s'$ and $t'+1$ if and only if $s'$ and $t'$ have opposite parity.
The pair of arcs can be chosen in exactly two ways, either both as right-even alternating arcs or both as left-even alternating arcs.
\end{lemma}

Given a pair $(S,T)$ of $k$-subsets of $[n]$, we will always write $S=\set{s_1,\ldots,s_k}$ with $s_1<\cdots<s_k$ and $T=\set{t_1,\ldots,t_k}$ with $t_1<\cdots<t_k$.
Define $Q(S,T)$ to be the set of numbers $q\in[n+1]$ such that, for all $j$ from $1$ to $k$, neither $s_j<q\le t_j$, nor $t_j<q\le s_j$.

\begin{lemma}\label{P Q lemma}
Let $\Sigma$ be an alternating arc diagram.
Then $Q(\pi(\Sigma))=P(\Sigma)$.
\end{lemma}
\begin{proof}
Write $(S,T)$ for $\pi(\Sigma)$.
If $p\in P(\Sigma)$, then no arc passes left or right of $p$. 
Thus there exists $k$ such that $s_j$ and $t_j$ are less than $p$ for all $j\le k$ and $s_j$ and $t_j$ are greater than or equal to $p$ for all $j>k$.
We see that $p\in Q(S,T)$.

Suppose that $q\in Q(S,T)$, and there exists some arc $\alpha$ that passes to the left or right of $q$.
The arc $\alpha$ belongs to some overlapping component of $\Sigma$, and each pair $s_i, t_i$ in the image of a different component satisfies $s_i, t_i < q$ or $s_i, t_i> q$.
Thus, we may as well assume that $\Sigma$ consists of a single overlapping component.  
Write $\pi(\Sigma) = (\set{s_1,\ldots,s_k},\set{t_1,\ldots,t_k})$ with $s_1<\cdots < s_k$ and $t_1<\cdots<t_k$.
Lemma~\ref{overlap comp} says that $\Sigma$ consists of either right-even overlapping arcs or left-even overlapping arcs.
Without loss of generality, we assume that $\Sigma$ consists of only right-even overlapping arcs, so that $\{s_1,\ldots, s_k\}$ is the set of bottom endpoints of those arcs.
Thus, $s_i \le t_i$ for each $i=1,2,\ldots, k$.
Let $s_i$ be the bottom endpoint of $\alpha$, and let $l$ be the largest number such that $s_l <q$.
We make two observations.
First, $\alpha$ must connect $s_i$ with $t_j+1$, where $j$ is strictly greater than $i$ (otherwise $s_i< q\le t_j\le t_i$), and $j$ is strictly greater than $l$ (otherwise $s_j < q \le t_j$).
Second, $t_{l+1} \ge q> t_{l}$, because $t_{l+1} \ge s_{l+1} \ge q >  t_l\ge s_l$.
We conclude that each number in the set of bottom endpoints $\{s_{l+1}, s_{l+2}, \ldots, s_k\}$ must connect with a number in the set $\{t_{l+1}+1,\ldots, t_k+1\}$.
Since $t_j +1$ is already connected to $s_i$, there is some number in the set $\{t_{l+1}+1, \ldots, t_k+1\}$ that is the top endpoint of two arcs, and that is a contradiction.
\end{proof}

We are now prepared to prove the main theorem of this section.

\begin{proof}[Proof of Theorem~\ref{alt arc bij}]
We first show that $\pi$ is well-defined.
Since each arc in $\Sigma$ contributes exactly one of its endpoints to $S'\cup S''$ and the other to $T'\cup T''$, both $S'\cup S''$ and $T'\cup T''$ have size $k$ as long as each contribution to $S'\cup S''$ is distinct and each contribution to $T'\cup T''$ is distinct.
Each contribution to $S'$ is distinct because no two arcs share the same lower endpoint, and each contribution to $S''$ is distinct because no two arcs share the same upper endpoint.
Proposition~\ref{overlap comp} implies that a right-even alternating arc with bottom endpoint $i$ and a \emph{distinct} left-even alternating arc with top endpoint $i+1$ are not compatible.
Thus the only elements of $S'\cap S''$ come from arcs that are both right-even alternating and left-even alternating, and we see that each contribution to $S'\cup S''$ is distinct.
The symmetric argument shows that each contribution to $T'\cup T''$ is distinct.
We have shown that $\pi$ is a well-defined map from alternating arc diagrams with $k$ arcs to pairs of $k$-element subsets of $[n]$.

We complete the proof by exhibiting an inverse $\eta$ to $\pi$.
Let $(S,T)$ be a pair of $k$-element subsets of $[n]$.
Write $Q(S,T)=\set{q_0,\ldots,q_m}$ with $q_0<\cdots<q_m$.
For each $i$ from $1$ to $m$, define $S_i=S\cap\set{q_{i-1},q_{i-1}+1,\ldots,q_i-1}$ and $T_i=T\cap\set{q_{i-1},q_{i-1}+1,\ldots,q_i-1}$.
We claim that $|S_i|=|T_i|$, and more specifically, that $s_j\in S_i$ if and only if $t_j\in T_i$.
Indeed, suppose $s_j\in S_i$, so that $q_{i-1}\le s_j<q_i$.
If $t_j<q_{i-1}$, then $t_j<q_{i-1}\le s_j$, contradicting the fact that $q_{i-1}\in Q(S,T)$.
If $t_j\ge q_i$, then $s_j<q_i\le t_j$, contradicting the fact that $q_i\in Q(S,T)$.
We conclude that $t_j\in T_i$.
The symmetric argument completes the proof of the claim.

Now, in light of Lemma~\ref{P Q lemma} and the definition of $\pi$, by subtracting $q_{i-1}-1$ from each element of $S_i$ and $T_i$, we reduce to the case where $m=1$ and thus $Q=\set{1,n+1}$ and $(S_1,T_1)=(S,T)$.
In particular, all of the arcs in the diagram $\eta(S,T)$ are right-even alternating, or all of the arcs are left-even alternating.
If $n=1$, then either $(S,T)=(\emptyset,\emptyset)$, in which case $\eta(S,T)$ has no arc, or $(S,T)=(\set{1},\set{1})$, in which case $\eta(S,T)$ has an arc connecting $1$ and $2$.

If $n>1$, then we observe that the element $1$ must be in $S$ or in $T$ but must not be in both.
Indeed, if $1$ is in neither set or in both, we see that $2\in Q(S,T)$, and this is a contradiction.
In particular, we will need to construct an arc whose lower endpoint is $1$ and whose upper endpoint is above $2$.
This arc will pass by $2$, and so it is either right-even alternating or left-even alternating (but not both).
If $1\in S$, then the corresponding arc is right-even alternating, and if $1\in T$ this arc is left-even alternating.
Without loss of generality, we assume $1\in S$, so that each $i$ in $S$ is a bottom endpoint and for each $j$ in $T$, $j+1$ is a top endpoint of a right-even alternating arc in $\eta(S,T)$.
To complete the proof, we show  that there is a unique way to pair off each bottom endpoint in $S$ with a top endpoint in $T$ so that the union of the resulting arcs is a noncrossing arc diagram.
Since the arcs in the diagram are all right-even alternating, we must pair each element of $S$ with a larger element of $T$.

We first decide which element of $T$ we should pair with $s_k$.
Because $s_k$ is the maximum element of $S$, Lemma~\ref{arc compat 1} implies that we must pair $s_k$ with some $t'$ such that $\set{t\in T:s_k<t<t',\, t-s_k\text{ odd}}$ is empty.
Similarly, Lemma~\ref{arc compat 2} implies that we must either pair $s_k$ with $t_k$ or pair $s_k$ with some $t'$ such that $t'-s_k$ is odd.
Furthermore, if we choose $t'$ according to those two rules, no matter how we pair the remaining elements of $S$ and $T$, the arcs produced will be compatible with the arc whose bottom endpoint is $s_k$.
We are forced to pair $s_k$ with $\min\set{t\in T:t\geq s_k,\, t-s_k\text{ odd}}$, or with $t_k$ if $\set{t\in T:t\geq s_k,\, t-s_k\text{ odd}}=\emptyset$.  
By induction on $k$, there is a unique way to pair the elements of $S\setminus\set{s_k}$ with the elements of $T\setminus\set{t'}$ to make a noncrossing alternating diagram.
Putting in the pair $(s_k,t')$ we obtain the unique pairing of elements of $S$ with elements of $T$ to make a noncrossing alternating diagram.
The base of the induction is where $k=1$.  
Here existence of a pairing is trivial and uniqueness comes from the requirement that the arc whose bottom endpoint is $1$ must be right-even alternating.
\end{proof}

The proof of Theorem~\ref{alt arc bij} completes the proof of Theorem~\ref{hard part finer} and Theorem~\ref{biNar thm} for type A.

\begin{remark}\label{type A insight}
The proof of Theorem~\ref{alt arc bij} provides a key insight that leads to our proof of Theorems~\ref{hard part} and~\ref{hard part finer}:
An alternating arc diagram decomposes into disjoint pieces such that each piece is either right-even alternating or left-even alternating.
We now describe a way of counting alternating arc diagrams by decomposing into left-even and right-even pieces (more coarsely than the decomposition into overlapping collections of arcs used in the proof of Theorem~\ref{alt arc bij}).

Recall that an arc is both right-even alternating and left-even alternating if and only if it connects consecutive points.
We call such an arc a \textit{simple arc} (and every other arc is called \textit{non-simple}).  
Given an alternating arc diagram $\Sigma$, let $R$ be the set of  points $p$ such that there exists a non-simple right-even alternating arc $\alpha$ in $\Sigma$ such that $\alpha$ passes alongside $p$, or $\alpha$ has $p$ as an endpoint.
Similarly, let $L$ be the set of $p'$ satisfying the above but with $\alpha$ left-even alternating.

We will refer to a set of integers of the form $\set{a,a+1,\ldots,b-1,b}$ as an \newword{interval}.
Let $R_1,\ldots, R_k$ be the maximal intervals contained in $R$, so that in particular $R$ is a disjoint union of the $R_i$, and any two of the $R_i$ have at least one point between them that is not in $R$.
Thus on each~$R_i$, we have an ``indecomposable piece'' of the diagram for $\Sigma$.
See Figure~\ref{fig: right-even double positive arc diagrams} for the right-even indecomposable pieces on $3$, $4$, or $5$ points.
\begin{figure}
\raisebox{12pt}{\includegraphics{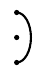}}\quad\quad
\raisebox{6pt}{\includegraphics{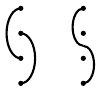}}\qquad
\includegraphics{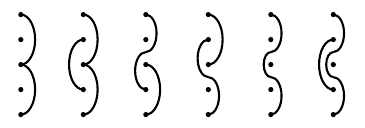}
\caption{
The indecomposable right-even alternating arc diagrams on $3$, $4$, or $5$ points
}
\label{fig: right-even double positive arc diagrams}
\end{figure}
(The restriction of $\Sigma$ to each $R_i$ is a union of overlapping collections of arcs, in the sense of Proposition~\ref{overlap comp}.
Each overlapping collection of arcs lives on an interval contained in $R_i$, and these intervals are pairwise disjoint except for intersecting at their endpoints.)
Symmetrically, $L$ breaks into an analogous collection of indecomposable pieces consisting of left-even alternating arcs.
Since the arcs of $\Sigma$ don't cross, each $R_i$ is disjoint from each $L_j$, except possibly at their endpoints.
The enumeration of alternating arc diagrams can be decomposed as a sum over all choices of $R$ and $L$ and their decompositions into pieces $R_1,\ldots, R_k$ and $L_i,\ldots,L_m$.
Each term in the sum is a product of: a power of $2$; a factor for each $R_i$ equal to the number of indecomposable pieces that can be constructed that interval; and an analogous factor for each $L_j$. 
The power of $2$ arises because there are ``gaps'' between intervals of $R$ or $L$ where we can fill in simple arcs or not.
Our proof of Theorems~\ref{hard part} and~\ref{hard part finer} generalizes this method, which can be carried out uniformly for all finite Coxeter groups.

Each indecomposable diagram with right-even alternating arcs corresponds (via~$\delta$) to a $c$-sortable permutation whose set of cover reflections has no simple reflections and also has full support.
(Cover reflections and supports will be defined in Section~\ref{can sec} for general Coxeter groups.
The cover reflections of a permutation are the transpositions $(i\,\,j)$ such that $i$ immediately precedes $j$ and $i>j$.
The support of an element is the set of simple reflections appearing in a reduced word for the element.
The requirement on supports here is that the union of the supports of the cover reflections is full.)
Similarly, each indecomposable diagram with left-even alternating arcs corresponds to a $c^{-1}$-sortable permutation whose set of cover reflections has no simple reflections and also has full support.
Thus the proof breaks bipartite biCambrian objects (the alternating arc diagrams) into pieces that belong to ordinary Catalan combinatorics.
More specifically, once we fix the type of arc (right-even or left-even alternating), the number of indecomposable diagrams on $m+1$ points is the number that in Section~\ref{double pos sec} will be called the double-positive Catalan number $\Cat\pp(A_m)$.

In the general setting, the role of the arcs in an alternating arc diagram is played by the \emph{canonical joinands} of a bipartite $c$-bisortable element. 
(The latter are join-irreducible $c$- or $c^{-1}$-sortable elements, and are defined in Section~\ref{can sec}.
For the connection between arcs and canonical join representations, see \cite[Section~3]{arcs}.)  
The fact that distinct right-even alternating and left-even alternating arcs may not overlap is a special case of the following fact, which we will prove uniformly in Section~\ref{sortable formula sec}:
Suppose $c$ is bipartite and $w$ is a $c$-bisortable element with a $c$-sortable canonical joinand $u$ and a $c^{-1}$-sortable canonical joinand $v$.
If $u=v$, then they equal a simple reflection, and if $u\neq v$, then the supports of $u$ and $v$ are disjoint.
Therefore, we can partition the set of canonical joinands of a $c$-bisortable element into a set of simple reflections, a set of non-simple $c$-sortable join-irreducible elements, and a set of non-simple $c^{-1}$-sortable join-irreducible elements.  
The set of non-simple $c$-sortable canonical joinands is a collection of ``indecomposable pieces'' that belong to ordinary Catalan combinatorics.
Just as the set $R$ broke into a disjoint union of the intervals $R_1, \ldots, R_k$, we break up this collection on $c$-sortable canonical joinands as follows:
An ``indecomposable piece'' corresponds to a subset of these canonical joinands that has full support on some irreducible parabolic subgroup of $W$.
The number of possible pieces for each standard parabolic subgroup is a double-positive Catalan number, so we obtain a formula counting $c$-bisortable elements in terms of  double-positive Catalan numbers.
We complete the proof by showing that the same formula also counts antichains in the doubled root poset.

\end{remark}

\begin{remark} 
Looking ahead to Section~\ref{double pos sec}, the previous remark implies an interpretation of the type-A double-positive Narayana number which---after some combinatorial manipulations that amount to changing from a bipartite Coxeter element to a linear Coxeter element---coincides with the interpretation given in \mbox{\cite[Theorem~1.1]{Ath-Sav}}.
\end{remark}

\subsection{Enumerating bipartite $c$-bisortable elements in type B}\label{type B sec} 
In this section, we use certain alternating arc diagrams to prove the enumeration of bipartite $c$-bisortable elements of type B given in Theorem~\ref{main thm}. 
In order to reuse much of our work from Section~\ref{alt_arc_diagrams}, we realize the weak order on $B_n$ as a sublattice of the weak order on $A_{2n-1}$, through the usual signed permutation model.

Let $x=x_{-n}\ldots x_{-1} x_1 \ldots x_n$ be a permutation of $\{\pm 1,\pm2, \ldots, \pm n\}$.  
Recall that $x$ is a \newword{signed permutation} if $x_i = -x_{-i}$ for each $i\in [n]$.
A signed permutation is completely determined by its abbreviated notation $x_1x_2\cdots x_n$.
We refer to the longer sequence $x_{-n}\ldots x_{-1} x_1 \ldots x_n$ as the \newword{full one-line notation} for $x$.
The weak order on $B_n$ is isomorphic to the set of the signed permutations, ordered so that $y_1\ldots y_n\covers x_1\ldots x_n$ if and only if one of the two following conditions is satisfied:
Either $y_i = x_{i+1} > x_i = y_{i+1}$ for $i,i+1\in [n]$ and $y_j = x_j$ for each $j\not\in \{i,i+1\}$, or $0<x_1=-y_1$ and $x_j = y_j$ for all $j \in \{2, 3,\ldots, n\}$.  
In the former case, the symmetry $y_i=-y_{-i}$ implies that $y_{-i-1}=x_{-i}> x_{-i-1} = y_{-i}$.  
In particular, in its full one-line notation, the signed permutation $y$ has two descents: $y_i>y_{i+1}$ and $y_{-i-1}>y_{-i}$. 
We say that such a pair of descents, or a single symmetric descent in positions $-1$ and $1$, is a \newword{type-B descent}.
(For more information on this realization of the  weak order on the type-$B$ Coxeter group see \cite[Section~8.1]{CombCoxeter}).

To motivate the definition of noncrossing arc diagrams of type B, we consider the action $y\mapsto w_0 y w_0$ on $A_{2n-1}$ where $w_0$ is the longest element.
(We describe $w_0\in A_{2n-1}$ below.  For the general definition of length, see Section~\ref{can sec}.)
We write each element $y$ in $A_{2n-1}$ as a permutation of $\{\pm 1,\ldots, \pm n\}$.
In the noncrossing arc diagram $\delta(y)$, we label the points $-n,-n+1,\ldots,-1,1,\ldots,n-1,n$ from bottom to top.
We place these points so that a half-turn rotation through the center of the diagram maps each point $i$ to the point~$-i$.
We call this rotation the \newword{central symmetry}.
The longest element $w_0$ in this copy of $A_{2n-1}$ is the permutation  $(-n)\cdots(-1)1\cdots n$.
Conjugation by $w_0$ acts by negating all of the entries of the full one-line notation of $y$ and reversing its order.
Thus, $y$ is fixed by the action of $w_0$ if and only if $y$ is a signed permutation.
On the level of noncrossing arc diagrams, the action of $w_0$ coincides with the central symmetry.

A \newword{centrally symmetric noncrossing arc diagram} is a  noncrossing arc diagram on the points $-n,\ldots,-1,1,\ldots, n$ that is fixed by the central symmetry.
The map $\delta$ restricts to a bijection from signed permutations to centrally symmetric noncrossing arc diagrams.
We use the term \newword{centrally symmetric arc} to describe either an arc that is fixed by the central symmetry or a pair of arcs that form an orbit under the symmetry.
For each $k$, the map $\delta$ restricts further to a bijection between signed permutations with $k$ type-B descents and centrally symmetric noncrossing arc diagrams with $k$ centrally symmetric arcs.
Since each signed permutation has at most one symmetric descent in the positions $-1$ and $1$, it follows that each centrally symmetric noncrossing arc diagram has at most one arc that is fixed by the central symmetry.


Now we describe the Cambrian and biCambrian congruences in type B.
The simple generators of $B_n$, are $s_0=(-1\,\,\,1)$ and $s_i=(-i\!-\!1\,\,\,-\!i)(i\,\,\,i\!+\!1)$ for $i=1,\ldots,n-1$, written in cycle notation as permutations of $\set{\pm 1, \ldots, \pm n}$.
A \newword{symmetric Coxeter element} of $A_{2n-1}$ is a Coxeter element that is fixed by the automorphism $y\mapsto w_0yw_0$.
Equivalently, the Coxeter element can be written as a product of some permutation of the elements $s_0,\ldots,s_{n-1}$ defined above.
This product in $A_{2n-1}$ can be interpreted as a Coxeter element of $B_n$, which we denote by $\tilde c$. 
A Coxeter element is symmetric if and only if it corresponds to a barring of $\set{\pm1,\ldots,\pm(n-1)}$ with the property that $i$ is overbarred if and only if $-i$ is underbarred.
Thus, a signed permutation avoids the pattern $\overline231$ if and only if it also avoids the pattern $31\underline2$ (in its full one-line notation).
The signed permutations avoiding $\overline231$ (and equivalently $31\underline2$) in their full notation are exactly the $\tilde c$-sortable elements by \cite[Theorem~7.5]{cambrian}.
Comparing with the description of $c$-sortable permutations following Proposition~\ref{same class}, we obtain the following proposition.

\begin{proposition}\label{b_sort}
Suppose $c$ is a symmetric Coxeter element of $A_{2n-1}$ and suppose $\tilde c$ is the corresponding Coxeter element of $B_n$.
A signed permutation is $\tilde c$-sortable in $B_n$ if and only if it is $c$-sortable as an element of $A_{2n-1}$.
\end{proposition}

The analogous result holds for $\tilde c$-bisortable elements.

\begin{proposition}\label{typeB_bisort}
Suppose $c$ is a symmetric Coxeter element of $A_{2n-1}$ and suppose $\tilde c$ is the corresponding Coxeter element of $B_n$.
A signed permutation is $\tilde c$-bisortable in $B_n$ if and only if it is $c$-bisortable as an element of $A_{2n-1}$.
\end{proposition}
\begin{proof} 
Suppose $w$ is a signed permutation.
If $w$ is $\tilde c$-bisortable, then Proposition~\ref{c join cinv} says that $w = u \join v$ for some $\tilde c$-sortable signed permutation $u$ and some \mbox{$\tilde c^{-1}$-sortable} signed permutation $v$.
Proposition~\ref{b_sort} says that, as elements of $A_{2n-1}$, $u$ is a $c$-sortable permutation and $v$ is a $c^{-1}$-sortable permutation.
It is well-known that the weak order on $B_n$ is a sublattice of the weak order on $A_{2n-1}$.
Indeed, in any finite Coxeter group, the map $y\mapsto w_0 y w_0$ is a rank-preserving \textit{lattice} automorphism.
For any lattice automorphism, the set of fixed points of the automorphism is a sublattice.
Thus, the join $u \join v$ is the same in $A_{2n-1}$ as in $B_n$, and Proposition~\ref{c join cinv} implies that $w$ is $c$-bisortable. 

On the other hand, if $w$ is $c$-bisortable as an element of $A_{2n-1}$, then as in Proposition~\ref{c join cinv}, we can write $w$ as $u \join v$, where $u$ is the $c$-sortable permutation $\pidown^c(w)$ and $v$ is the $c^{-1}$-sortable permutation $\pidown^{c^{-1}}(w)$.
Since conjugation by $w_0$ is a lattice automorphism fixing $w$, we obtain $w = (w_0 u w_0) \join (w_0 v w_0)$.
But $w_0 u w_0$ is $c$-sortable and below $w$, so $w_0 u w_0\le u$.
Since conjugation by $w_0$ is order preserving, we conclude that $w_0 u w_0 = u$.
Similarly $w_0 v w_0= v$.
Thus, by Proposition~\ref{b_sort}, $u$ is $\tilde c$-sortable and $v$ is $\tilde c^{-1}$-sortable in $B_n$. 
Since the weak order on $B_n$ is a sublattice of the weak order on $A_{2n-1}$, Proposition~\ref{c join cinv} says that $w$ is $\tilde c$-bisortable.
\end{proof}

A bipartite Coxeter element $\tilde c$ of $B_n$ is a symmetric bipartite Coxeter element of $A_{2n-1}$, so combining Propositions~\ref{avoid alt} and~\ref{typeB_bisort}, we immediately obtain the following proposition.

\begin{proposition}\label{avoid alt B}  
For $W=B_n$ and $\tilde c$ a bipartite Coxeter element, the map $\delta$ restricts to a bijection from $\tilde c$-bisortable signed permutations with $k$ descents to centrally symmetric alternating arc diagrams on $2n$ points with $k$ centrally symmetric alternating arcs.
\end{proposition}

Thus, to count the bipartite $c$-bisortable elements in $B_n$, it remains only to count centrally symmetric alternating arc diagrams.
The points in the noncrossing arc diagram for a permutation in $S_{2n}$ are labeled $1,\ldots,2n$ from bottom to top.
If we instead label the points $-n,\ldots,-1,1,\ldots,n$ from bottom to top, we can interpret the map $\pi$ as returning an ordered pair of subsets of $\set{-n,\ldots,-1,1,\ldots n-1}$.
Define $\pi_B$ to be the map on centrally symmetric alternating arc diagrams with $2n$ vertices that first does the map $\pi$ to obtain $(S,T)$ and then ignores $T$ and outputs only $S$.
The following theorem shows that the number of centrally symmetric alternating arc diagrams with $k$ centrally symmetric arcs is $\binom{2n-1}{2k}+\binom{2n-1}{2k-1}=\binom{2n}{2k}$, proving Theorem~\ref{biNar thm} for type $B_n$.

\begin{theorem}\label{type b-alt arc bij}
For each $k$, the map $\pi_B$ restricts to a bijection from centrally symmetric alternating arc diagrams with $k$ centrally symmetric arcs to subsets of $\set{-n,\ldots,-1,1,\ldots n-1}$ of size $2k$ or $2k-1$.
\end{theorem}
\begin{proof}
We first show that $\pi_B$ is a bijection from centrally symmetric alternating arc diagrams to subsets of $\set{-n,\ldots,-1,1,\ldots n-1}$.
Given $S\subseteq\set{\pm1,\ldots,\pm n}$, we write $-S-1$ for the set $\set{-i-1:i\in S}$, where we interpret $1-1$ to mean $-1$ in order to make $-S-1$ a subset of $\set{\pm1,\ldots,\pm n}$.
We show that $\pi_B$ is a bijection by showing that an alternating diagram $\Sigma$ is centrally symmetric if and only if $\pi(\Sigma)=(S,-S-1)$ for some $S$.

The terms ``right-even alternating'' and ``left-even alternating'' should be understood in terms of the labeling of points as $1,\ldots,2n$.
These terms become problematic when we label points as $-n,\ldots,-1,1,\ldots,n$.
(For example, whether a right-even alternating arc passes left or right of the point labeled $i$ depends on the sign of $i$, the parity of $i$, and the parity of $n$.)
Without worrying about these details, we make two easy observations:
 First, an alternating arc is right-even alternating if and only if its image under the central symmetry is right-even alternating.
Second, the central symmetry swaps top with bottom endpoints and  positive with negative endpoints.
These observations immediately imply that $\pi$ maps centrally symmetric alternating arc diagrams to pairs of the form $(S,-S-1)$.

These observations also immediately imply that if $\pi$ maps an alternating arc diagram $\Sigma$ to $(S,T)$ and $\Sigma'$ is the image of $\Sigma$ under the central symmetry, then $\pi$ maps $\Sigma'$ to $(-T+1,-S-1)$, where $-T+1$ is the set $\set{-i+1:i\in T}$, where we interpret $-1+1$ to mean $1$.
In particular, if $\pi$ maps $\Sigma$ to $(S,-S-1)$, then $\pi$ also maps $\Sigma'$ to $(S,-S-1)$.
Since we already know that $\pi$ is a bijection, we conclude that in this case $\Sigma$ must be centrally symmetric.
We have shown that $\Sigma$ is centrally symmetric if and only if $\pi(\Sigma)$ is of the form $(S,-S-1)$.
Therefore $\pi_B$ is a bijection.

It is now immediate that $\pi_B$ maps a centrally symmetric alternating arc diagram with $k$ centrally symmetric arcs to a $(2k-1)$-element set if the diagram has an arc that is fixed by the central symmetry or to a $2k$-element set if all of the arcs in the diagram come in symmetric pairs.
(Recall that the diagram has at most one arc fixed by the central symmetry.)
\end{proof}

\subsection{Simpliciality of the bipartite biCambrian fan in types A and B}\label{A B simp sec}
We now prove Theorem~\ref{simple A B}, which states that the bipartite biCambrian fan is simplicial in types A and B.
The proof of the type-A case of Theorem~\ref{simple A B} proceeds by combining results of~\cite{DIRRT} and~\cite{arcs}.

Some collections of noncrossing arc diagrams (including, we will see, the alternating arc diagrams), correspond to lattice quotients of the weak order.
More specifically, a collection of noncrossing arc diagrams may be the image, under  $\delta$, of the bottom elements of congruence classes of some congruence.
To describe when and how such a situation arises, we need the notion of a subarc.
For $i<j$ and $i'<j'$, an arc $\alpha$ connecting $i$ to $j$ is a \newword{subarc} of an arc $\alpha'$ connecting $i'$ to $j'$ if $i'\le i$ and $j'\ge j$ and if $\alpha$ and $\alpha'$ pass to the same side of every point between $i$ and~$j$.
It follows from \cite[Theorem~4.1]{arcs} and \cite[Theorem~4.4]{arcs} that a subset $D$ of the noncrossing arc diagrams on $n+1$ points is the image, under $\delta$, of the set of bottom elements for some congruence $\Theta$ if and only if all of the following conditions hold.
\begin{enumerate}[\rm(i)]
\item There exists a set $U$ of arcs such that a noncrossing arc diagram $\Sigma$ is in $D$ if and only if all arcs in $\Sigma$ are in $U$.
\item \label{subarc req} Any subarc of an arc in $U$ is itself also in $U$.\end{enumerate}
We will call $U$ the set of \newword{unremoved arcs} of the congruence $\Theta$.
If $C$ is any set of arcs and $U$ is the maximal set such that $U\cap C=\emptyset$ and condition~\eqref{subarc req} above holds, then we say that the congruence $\Theta$ is \newword{generated by removing the arcs $C$}.

An element $j$ of a finite lattice $L$ is join-irreducible if it covers exactly one element $j_*$.
A lattice congruence on $L$ \newword{contracts} a join-irreducible element $j$ if the congruence has $j\equiv j_*$.
A congruence is uniquely determined by the set of join-irreducible elements it contracts.
The join-irreducible elements of the weak order on $A_n$ are the permutations in $S_{n+1}$ with exactly one descent.
In particular, the map $\delta$ restricts to a bijection between join-irreducible elements in $S_{n+1}$ and noncrossing arc diagrams with exactly one arc.
(We will think of this restriction as mapping join-irreducible elements to arcs, rather than to singletons of arcs.)
Under this bijection, the join-irreducible elements \emph{not} contracted by a congruence $\Theta$ correspond to the arcs in $U$, where $U$ is the set of unremoved arcs of $\Theta$.
The congruence is \newword{generated by contracting} a set $J$ of join-irreducible elements if and only if it is generated by removing the arcs $\delta(J)$.

We call $j$ a \newword{double join-irreducible} element if it is join-irreducible and if the unique element $j_*$ covered by $j$ is either the bottom element of the lattice or is itself join-irreducible.
The following is a result of~\cite{DIRRT}. 
\begin{theorem}\label{one} Suppose $\Theta$ is a lattice congruence on the weak order on $A_n$.
Then the following three conditions are equivalent.
\begin{enumerate}[\rm(i)] 
\item The undirected Hasse diagram of the quotient lattice $A_n/\Theta$ is a regular graph.
\item $\F_\Theta(A_n)$ is a simplicial fan.
\item $\Theta$ is generated by contracting a set of double join-irreducible elements.  
\end{enumerate}
\end{theorem}

We now apply these considerations to alternating arc diagrams.
First, it is apparent that the set of alternating arc diagrams is the image of $\delta$ restricted to the set of bottom elements of a congruence.
(Indeed, this is the bipartite $c$-biCambrian congruence.)
It is also apparent that the congruence is generated by removing the arcs that connect $i$ to $i+3$ and that \emph{do not} alternate.
(That is they pass to the same side of $i+1$ and $i+2$.)
Applying the inverse of $\delta$, we see that the congruence is generated by contracting the join-irreducible elements 
\[1\cdots (i-1)(i+1)(i+2)(i+3)i(i+4)\cdots (n+1)\]
and
\[1\cdots (i-1)(i+3)i(i+1)(i+2)(i+4)\cdots (n+1)\]
for $i=1,\ldots,n-2$.
These are both double join-irrreducible elements, and thus Theorem~\ref{one} implies the type-A case of Theorem~\ref{simple A B}.

We now move to the type-B case of Theorem~\ref{simple A B}.
Just as in type-A, there is a correspondence between congruences on the weak order and certain sets of (centrally symmetric) noncrossing arc diagrams.
However, there is currently no analogue to Theorem~\ref{one} in type B\@.
Therefore, instead of arguing the type-B case as we argued the type-A case, we will use a folding argument to show that the type-A case implies the type-B case.

Say a lattice congruence of the weak order on $A_{2n-1}$ is \newword{symmetric under conjugation by $w_0$} if for all $x,y\in A_{2n-1}$ we have $x\equiv y$ modulo $\Theta$ if and only if $w_0xw_0\equiv w_0yw_0$ modulo $\Theta$.
\begin{prop}\label{sym bottoms}
If $\Theta$ is a lattice congruence of the weak order on $A_{2n-1}$ that is symmetric under conjugation by $w_0$, then its restriction to the sublattice $B_n$ is a congruence $\Theta'$.
An element of $B_n$ is the bottom element of its $\Theta'$-class if and only if it is the bottom element of its $\Theta$-class.
\end{prop}
\begin{proof}
It is also a well-known and easy fact that the restriction of a lattice congruence to any sublattice is a congruence on the sublattice, and the first assertion of the proposition follows.
One implication in the second assertion is immediate.
For the other implication, suppose $x\in B_n$ is the bottom element of its $\Theta'$-class and let $y=\pidown^\Theta(x)$, so that in particular $x\equiv y$ modulo $\Theta$.
Then because $\Theta$ is symmetric under conjugation by $w_0$, also $x=w_0xw_0\equiv w_0yw_0$ modulo $\Theta$.
Since $y$ is the bottom element of its $\Theta$-class, $y\le w_0 y w_0$.
Since conjugation by $w_0$ is order preserving, also $w_0yw_0\le y$, so $y=w_0 y w_0$.
Thus $y$ is in the $\Theta'$-class of $x$, and we conclude that $y=x$, so that $x$ is also the bottom element of its $\Theta$-class.
\end{proof}

\begin{proposition}\label{typeB simplicial} 
Suppose that $\Theta$ is a lattice congruence of the weak order on $A_{2n-1}$ and let $\Theta'$ denote its restriction to the weak order on $B_n$.
If $\F_{\Theta}(A_{2n-1})$ is simplicial and $\Theta$ is symmetric under conjugation by $w_0$, then $\F_{\Theta'}(B_n)$ is simplicial.
\end{proposition}
Before we proceed with the proof of Proposition~\ref{typeB simplicial} we define some useful terminology.
Recall that there is a linear functional $\lambda$ that orients the adjacency graph on maximal cones in $\F(W)$ to yield a partial order isomorphic to the weak order on $W$.
A facet  of a maximal cone is a \newword{lower wall} (with respect to $\lambda$) if passing through it to an adjacent maximal cone is the same as moving down by a cover in the weak order.
\newword{Upper walls} are defined dually.
The maximal cones of $\F_{\Theta}(W)$ similarly have lower and upper walls with respect to $\lambda$; passing from one cone to an adjacent cone through a lower wall corresponds to moving down by a cover in the lattice quotient induced by $\Theta$.
The lower walls of a maximal cone  in $\F_{\Theta}(W)$ are the lower walls of the smallest element in the corresponding $\Theta$-congruence class.
(Recall that each maximal cone in $\F_{\Theta}(W)$ is the union of the set of maximal cones in $\F(W)$ in the same $\Theta$-congruence class.)
Dually, the upper walls of a maximal cone in $\F_{\Theta}(W)$ are the upper walls of the cone corresponding to the largest element in the $\Theta$-congruence class.

\begin{proof}[Proof of Proposition~\ref{typeB simplicial}]
We begin by considering type $A_{2n-1}$ in the usual geometric representation in $\R^{2n}$.
However, to prepare for the type-B construction, we index the standard unit basis vectors of $\R^{2n}$ as $-n,\ldots,-1,1,\ldots,n$.
In this representation, there is a reflecting hyperplane $H_{ji}$, with normal vector $e_j-e_i$, for each $i<j$ with $i,j\in\set{\pm1,\ldots\pm n}$.
The maximal cone corresponding to the permutation $x_{-n}\cdots x_{-1}x_1\cdots x_n$ has a lower (respectively upper) wall contained in $H_{ji}$ if and only if there exists $r\in\set{-n,\ldots,-1,1,\ldots n-1}$ such that $x_r=j$ and $x_{r+1}=i$ (respectively, $x_{r+1}=j$ and $x_{r}=i$).
As the price for our choice of indices, when $r=-1$, we must interpret $r+1$ here to mean $1$.

Recall that the signed permutations of $B_n$ are exactly the permutations in $A_{2n-1}$ that are fixed under conjugation by $w_0$ and that the restriction of weak order to these $w_0$-fixed permutations is weak order on $B_n$. 
As an abuse of terminology, the linear map on $\R^{2n}$ that sends each vector $(v_{-n},\ldots,v_{-1},v_1,\ldots,v_n)$ to $-(v_n,\ldots,v_1,v_{-1},\ldots,v_{-n})$ will be called the conjugation action of $w_0$ on $\R^{2n}$.
Let $L$ be the linear subspace of $\R^{2n}$ consisting of vectors fixed by this action.
These are the vectors with $v_i=-v_{-i}$ for all $i$.
A permutation in $A_{2n-1}$ is fixed under conjugation by $w_0$ if and only if its corresponding cone in $\F(A_{2n-1})$ intersects $L$ in its relative interior, in which case the cone is also fixed under conjugation by $w_0$.
Thus, we obtain $\F(B_n)$ as the fan induced on $L$ by $\F(A_{2n-1})$, and the weak order on $B_n$ arises from that induced fan, ordered by the same linear functional $\lambda$ as $\F(A_{2n-1})$.
Moreover, $\F_{\Theta'}(B_n)$ is the fan induced on $L$ by $\F_{\Theta}(A_{2n-1})$.

Almost all of the lower walls of a $w_0$-fixed maximal cone $C$ in $\F_{\Theta}(A_{2n-1})$ intersect $L$ in pairs.
Specifically, Proposition~\ref{sym bottoms} implies that any such cone is associated to a signed permutation $x=x_{-n}\cdots x_{-1}x_1\cdots x_n$ that is the bottom element of its $\Theta$-class.
A descent $x_{-1}x_1$ of $x$ contributes a single lower wall to $C$, and thus a single lower wall to $C\cap L$.
We will say that such a lower wall is centrally symmetric. 
All other descents of $x$ come in symmetric pairs $x_{-i-1}x_{-i}$ and $x_ix_{i+1}$, contributing two lower walls to $C$.
However, these two walls have the same intersection with $L$ and thus contribute only one lower wall to $C\cap L$.
Similar dual statements hold for the upper walls.
Most importantly, among all of the walls of $C\cap L$, there are at most two that are centrally symmetric: at most one among the set of lower walls, and at most one among set of upper walls.

Since $\F_{\Theta}(A_{2n-1})$ is simplicial, $C$ has an odd number of walls.
In particular, this implies that among all of the walls for $C$, there is exactly one that is centrally symmetric wall.
Suppose that this wall is a lower wall. 
Then, $C$ has an odd number of lower walls, say $2k-1$, and their intersection with $L$ yields $k$ lower walls for the corresponding cone $C\cap L$ in $\F_{\Theta'}(B_n)$.
Since $\F_{\Theta}(A_{2n-1})$ is simplicial, there are $2n-2k$ upper walls, which intersect $L$  in pairs, to form $n-k$ upper walls in $\F_{\Theta'}(B_n)$.
Thus the cone associated to $C$ in $\F_{\Theta'}(B_n)$ has a total of $n$ walls.
The same argument (switching lower walls with upper walls) shows that if the centrally symmetric wall is an upper wall, the cone associated to $C$ in $\F_{\Theta'}(B_n)$ has $n$ walls.
We conclude that $\F_{\Theta'}(B_n)$ is simplicial.
\end{proof}

\begin{proof}[Proof of the type-B case of Theorem~\ref{simple A B}]
Let $c$ be a bipartite Coxeter element in $A_{2n-1}$ and let $\tilde c$ be the same element, thought of as a Coxeter element of $B_n$.
Recall that $\tilde c$ is also bipartite.

Using the bipartite case of Proposition~\ref{same class} (with $n$ replaced by $2n-1$), it is easily checked that $x\equiv y $ modulo $\Theta_c$ if and only if $w_0xw_0\equiv w_0yw_0$ modulo $\Theta_{c}$.
It follows that the $c$-biCambrian congruence is symmetric under conjugation by~$w_0$.
Since a congruence is uniquely determined by the set of bottom elements of its classes, Proposition~\ref{typeB_bisort} implies that the restriction of the $c$-biCambrian congruence to $B_n$ is the $\tilde c$-biCambrian congruence.
Thus the type-B case of the theorem follows from Proposition~\ref{typeB simplicial} and the type-A case of the theorem.
\end{proof}

\section{Double-positive Catalan numbers and biCatalan numbers}\label{double pos sec}  
For each finite Coxeter group $W$, the positive $W$-Catalan number $\Cat^+(W)$ is defined from the $W$-Catalan number $\Cat(W)$ by inclusion-exclusion.  
In this section, we review the definition of the positive $W$-Catalan number and define the double-positive $W$-Catalan number $\Cat\pp(W)$ from the positive $W$-Catalan number by inclusion-exclusion.
We then prove Theorems~\ref{hard part} and \ref{hard part finer} by showing how to count both antichains in the doubled root poset and bipartite $c$-bisortable elements by the same formula involving double-positive Catalan numbers.
Recall that these two theorems in particular establish that the terms ``biCatalan number'' and ``biNarayana number'' make sense.
As we prove these theorems, we obtain as a by-product a formula for the $W$-biCatalan numbers in terms of the double-positive Catalan numbers of parabolic subgroups of $W$.
This formula leads to a recursion for the $W$-biCatalan numbers.
Using a similar recursion for the $W$-Catalan numbers and a few other enumerative facts, we solve that recursion for $\biCat(D_n)$ to complete the proof of Theorem~\ref{enum thm}.
The recursions discussed here all have Narayana $q$-analogues, but we are not at this time able to solve the recursion to find a formula for $\biCat(D_n;q)$.
See Section~\ref{type D biNar sec} for a brief discussion of the type-D biNarayana numbers.

The $W$-Catalan number has a uniform formula $\Cat(W)=\prod_{i=1}^n\frac{h+e_i+1}{e_i+1}$, but this formula will not play a large role in this paper.
Similarly, the positive $W$-Catalan number is $\Cat^+(W)=\prod_{i=1}^n\frac{h+e_i-1}{e_i+1}$, but we will give, and use, the more simple-minded inclusion-exclusion definition of $\Cat^+(W)$.
The positive $W$-Catalan and positive $W$-Narayana numbers have interpretations in each setting of Coxeter-Catalan combinatorics.
(See for example \cite{Ath,ABMW,Ath-Tzan,ga,Haiman,Panyushev,sortable,Sommers}.)
In this paper, we give the usual interpretations in the settings of nonnesting partitions and $c$-sortable elements, specifically in Sections~\ref{antichain formula sec} and~\ref{sortable formula sec}.
We are not aware of any uniform formulas for the double-positive $W$-Catalan or $W$-Narayana numbers; we define these numbers by inclusion-exclusion.
Case-by-case formulas for the double-positive $W$-Catalan numbers are found in Theorem~\ref{Cat++ thm}.


The double-positive $W$-Narayana numbers appeared in \cite{Ath-Sav} as the local $h$-vector of the positive part of the cluster complex.
(See Remark~\ref{Ath connection}.)
As far as we know, \cite{Ath-Sav} was the first appearance of the double-positive $W$-Catalan/Narayana numbers and the only appearance before the current paper.

\subsection{Double-positivity}\label{double pos subsec}
We write $S$ for the set of simple reflections generating~$W$.
Given $J\subseteq S$, the notation $W_J$ stands for the subgroup of $W$ generated by $J$.
The subgroup $W_J$ is called a \newword{standard parabolic subgroup} of $W$ and is a Coxeter group in its own right with simple reflections $J$.
In particular, each $W_J$ has a Catalan number.
As usual, we define the \newword{positive $W$-Catalan number} to be
\begin{equation}\label{Cat+ def}
\Cat^+(W)=\sum_{J\subseteq S}(-1)^{|S|-|J|}\Cat(W_J).
\end{equation} 
As is not usual, we define the \newword{double-positive $W$-Catalan number} to be  
\begin{equation}\label{Cat++ def}
\Cat\pp(W)=\sum_{J\subseteq S}(-1)^{|S|-|J|}\Cat^+(W_J).
\end{equation}


We will prove the following formula for the biCatalan numbers.

\begin{theorem}\label{bicat d-p}
For any finite Coxeter group $W$ with simple generators $S$, 
\begin{equation}\label{bicat d-p formula}
\biCat(W)=\sum2^{|S|-|I|-|J|}\Cat\pp(W_{I})\Cat\pp(W_{J}),
\end{equation}
where the sum is over all ordered pairs $(I,J)$ of disjoint subsets of $S$.
\end{theorem}

We can prove a refinement of Theorem~\ref{bicat d-p} using the usual notion of positive Narayana numbers and a notion of double-positive Narayana numbers.
The \newword{positive $W$-Narayana numbers} are
\begin{equation}\label{Nar+ def}
\Nar^+_k(W)=\sum_{J\subseteq S}(-1)^{|S|-|J|}\Nar_k(W_J).
\end{equation}
We define the \newword{double-positive $W$-Narayana number} to be
\begin{equation}\label{Nar++ def}
\Nar\pp_k(W)=\sum_{J\subseteq S}(-1)^{|S|-|J|}\Nar^+_{k-|S|+|J|}(W_J).
\end{equation}
In all of the settings where the Narayana numbers appear, it is apparent that $\Nar_k(W)=0$ whenever $k<0$ or $k$ is greater than the rank of $W$.
These definitions establish that $\Nar^+_k(W)=\Nar\pp_k(W)=0$ as well for those values of $k$.

Defining $\Cat^+(W;q)=\sum_k\Nar^+_k(W)q^k$ and $\Cat\pp(W;q)=\sum_k\Nar\pp_k(W)q^k$, equations \eqref{Nar+ def} and \eqref{Nar++ def} correspond to 
\begin{equation}\label{q Cat+ def}
\Cat^+(W;q)=\sum_{J\subseteq S}(-1)^{|S|-|J|}\Cat(W_J;q).
\end{equation}
and
\begin{equation}\label{q Cat++ def} 
\Cat\pp(W;q)=\sum_{J\subseteq S}(-q)^{|S|-|J|}\Cat^+(W_J;q).
\end{equation}
Taking $\biCat(W;q)=\sum_k\biNar_k(W)q^k$, we will prove the following $q$-analog of Theorem~\ref{bicat d-p}.

\begin{theorem}\label{biCat GF d-p}
For any finite Coxeter group $W$ with simple generators $S$,
\begin{equation}\label{GF d-p formula}
\biCat(W;q)=\sum\, q^{|M|}\Cat\pp(W_{I};q)\Cat\pp(W_{J};q),
\end{equation}
where the sum is over all ordered triples $(I,J,M)$ of pairwise disjoint subsets of~$S$.
\end{theorem}

The following theorem is equivalent to Theorem~\ref{biCat GF d-p}.
\begin{theorem}\label{binar d-p}
For any finite Coxeter group $W$ with simple generators $S$ and any~$k$,
\begin{equation}\label{binar d-p formula}
\biNar_k(W)=\sum\,\sum_{i=0}^{k-|M|}\Nar\pp_i(W_{I})\Nar\pp_{k-|M|-i}(W_{J}),
\end{equation}
where the outer sum is over all ordered triples $(I,J,M)$ of pairwise disjoint subsets of~$S$.
(If $|M|>k$, then the inner sum is interpreted to be zero.)
\end{theorem}

To prove these theorems, as well as Theorems~\ref{hard part} and~\ref{hard part finer}, we establish (in Propositions~\ref{q antichain} and~\ref{q bisortable formula}) that the right side of \eqref{GF d-p formula} counts antichains $A$ in the doubled root poset with weight $q^{|A|}$ and also counts bipartite $c$-bisortable elements $v$ with weight $q^{\des(v)}$.
Once these counts are established, Theorems~\ref{hard part} and~\ref{hard part finer} follow, and in particular the definitions of the biCatalan and biNarayana numbers are validated.
Also, Theorem~\ref{biCat GF d-p} holds, leading immediately to Theorems~\ref{bicat d-p} and~\ref{binar d-p}.

\subsection{Counting twin nonnesting partitions}\label{antichain formula sec}
We now recall the interpretations of the positive Catalan and Narayana numbers and give the interpretations of double-positive Catalan and Narayana numbers in the nonnesting setting. 
(Results in \cite{Ath-Tzan,Sommers} give the same interpretations, but accomplish much more, by establishing bijections and counting formulas.
By contrast, here we are only making simple assertions about inclusion-exclusion.)
After giving these interpretations, we prove that the formula in Theorem~\ref{binar d-p} counts $k$-element antichains in the doubled root poset.

Since it is customary to talk about the ``$W$-Catalan number'' rather than the ``$\Phi$-Catalan number,'' we will make statements about ``the root poset of $W$,'' when $W$ is a crystallographic Coxeter group.
This is harmless because, although the map from crystallographic root systems to Coxeter groups is not one-to-one, for each crystallographic Coxeter group, all corresponding crystallographic root systems have isomorphic root posets.
Correspondingly, when $W_J$ is a standard parabolic subgroup of $W$, we will say that a root or set of roots is ``contained in $W_J$'' if it is contained in the subset of $\Phi$ forming a root system for $W_J$.
An antichain that is not contained in any proper parabolic $W_J$ has full support, in the sense of Section~\ref{nn sec}.

For any $J\subseteq S$, the number of antichains in the root poset for $W$ that are contained in $W_J$ is $\Cat(W_J)$.  
By inclusion-exclusion, we conclude that: 
\begin{prop}\label{i-e nn}
The number of antichains in the root poset for~$W$ with full support is $\Cat^+(W)$.
The number of $k$-element antichains in the root poset for~$W$ with full support is $\Nar^+_k(W)$.
\end{prop}

For $J\subseteq S$, the map $A\mapsto A\setminus\set{\alpha_i:i\in J}$ is a bijection from the set of antichains containing the simple roots $\set{\alpha_i:i\in J}$ to the set of antichains in the root poset for $W_{S\setminus J}$.

Using this bijection, we prove the following proposition.
\begin{prop}\label{i-e nn 2}
The number of antichains in the root poset for~$W$ containing no simple roots is $\Cat^+(W)$.
The number of $k$-element antichains in the root poset for~$W$ containing no simple roots is $\Nar^+_{n-k}(W)$.
\end{prop}
\begin{proof}
The bijection mentioned above implies that the generating function for antichains containing the simple roots $\set{\alpha_i:i\in J}$ (and possibly additional simple roots) is $q^{|J|}\Cat(W_{S\setminus J};q)$.
By inclusion-exclusion, the generating function for $k$-element antichains containing no simple roots is $\sum_{J\subseteq S}(-q)^{|S|-|J|}\Cat(W_J;q)$.
On the other hand, starting with \eqref{q Cat+ def}, replacing $q$ by $q^{-1}$, multiplying through by $q^{|S|}$ (i.e.\ $q^n$), and using the known symmetry $q^{|J|}\Cat(W_J;q^{-1})=\Cat(W_J;q)$ of the coefficients of $\Cat(W_J;q)$, we obtain 
\[\sum_{k=0}^n\Nar^+_{n-k}(W)q^k=\sum_{J\subseteq S}(-q)^{|S|-|J|}\Cat(W_J;q).\qedhere\]
\end{proof}

The bijection described above restricts to a bijection from the set of antichains \emph{with full support} containing the simple roots $\set{\alpha_i:i\in J}$ to the set of antichains \emph{with full support} in the root poset for $W_{S\setminus J}$.
Thus, a similar inclusion-exclusion argument yields the following proposition.

\begin{prop}\label{i-e nn 3}
The number of antichains in the root poset for~$W$ with full support containing no simple roots is $\Cat\pp(W)$.
The number of $k$-element antichains in the root poset for~$W$ with full support containing no simple roots is $\Nar\pp_k(W)$.
\end{prop}


\begin{remark}\label{Ath connection}
The polynomials $\Cat\pp(W;q)$ appeared in \cite{Ath-Sav}, where Athanasiadis and Savvidou showed that $\Cat\pp(W;q)$ is the local $h$-vector of the positive part of the cluster complex, as we now explain.
We refer to \cite{Ath-Sav} for the relevant definitions, which we will not need here.
In light of \cite[Theorem~1.5]{Ath-Tzan} and Proposition~\ref{i-e nn 2}, the polynomial $h(\Delta_+(\Phi),x)$ appearing in \cite{Ath-Sav} is $x^{|S|}\Cat^+(W;x^{-1})$, where $(W,S)$ is the Coxeter system associated to $\Phi$.
Thus the assertion of \cite[Proposition~2.5]{Ath-Sav} is that the local $h$-vector of the positive part of the cluster complex is $\sum_{J\subseteq S}(-1)^{|S|-|J|}x^{|J|}\Cat^+(W;x^{-1})$.
But since the local $h$-vector is symmetric by \cite[Theorem~3.3]{Stanley}, we can replace $x$ by $x^{-1}$ and multiply by $x^{|S|}$ to show that the local $h$-vector is $\sum_{J\subseteq S}(-x)^{|S|-|J|}\Cat^+(W;x)=\Cat\pp(W;x)$.
\end{remark}

We now prove the key result on antichains in the doubled root poset.
\begin{prop}\label{q antichain}
For any finite Coxeter group $W$ with simple generators $S$, the generating function $\sum_Aq^{|A|}$ for antichains $A$ in the doubled root poset is 
\begin{equation}\label{q antichain formula}
\sum\, q^{|M|}\Cat\pp(W_{I};q)\Cat\pp(W_{J};q),
\end{equation}
where the sum is over all ordered triples $(I,J,M)$ of \emph{pairwise disjoint} subsets of~$S$.
\end{prop}
\begin{proof} 
In light of Proposition~\ref{i-e nn 3}, the proposition amounts to the following assertions:
First, there is a bijection from antichains $A$ in the doubled root poset to triples $(B,C,M)$ such that $B$ and $C$ are antichains in the root poset for $W$, each containing no simple roots, and the sets $I=\supp(B)$, $J=\supp(C)$ and $M$ are pairwise disjoint.
Second, under this bijection, $|B|+|C|+|M|=|A|$.
Every antichain $A$ in the doubled root poset consists of some set $B$ of positive non-simple roots in the top root poset, some set $C$ of positive non-simple roots in the bottom root poset, and some set $M$ of simple roots.  
The sets $I$, $J$, and $M$ are pairwise disjoint because $A$ is an antichain. 
The map $A\mapsto(B,C,M)$ is the desired bijection.
\end{proof}

It will be useful to have a similar formula for antichains in the (not doubled) root poset, which are known to be counted by $\Cat(W)$.

\begin{theorem}\label{Cat GF d-p}
For any finite Coxeter group $W$ with simple generators $S$.
\begin{equation}\label{Cat GF d-p formula}
\Cat(W;q)=\sum\, q^{|J|}\Cat\pp(W_{I};q),
\end{equation}
where the sum is over all ordered pairs $(I,J)$ of \emph{disjoint} subsets of~$S$.
\end{theorem}
\begin{proof}  
Every antichain $A$ in the root poset consists of some set $B$ of positive non-simple roots and some set $C$ of simple roots.
Writing $I$ and $J$ for the supports of $B$ and $C$, again $I$ and $J$ are disjoint.
By Proposition~\ref{i-e nn 3}, each pair $(I,J)$ of disjoint subsets of $S$ contributes $q^{|J|}\Cat\pp(W_{I};q)$ to the count.
\end{proof}

The following is an immediate consequence of Proposition~\ref{i-e nn 3} and will also be useful.

\begin{prop}\label{Cat++ reducible}
If $W$ is reducible as $W_1\times W_2$, then 
\begin{equation}\label{Cat++ reducible formula}
\Cat\pp(W;q)=\Cat\pp(W_1;q)\Cat\pp(W_2;q).
\end{equation}
\end{prop}

\subsection{Canonical join representations and lattice congruences}\label{can sec}  
To count bipartite $c$-bisortable elements, we will use a canonical factorization in the weak order called the canonical join representation.
In this section, we focus exclusively on the lattice-theoretic tools that we will use in the following sections to complete the proof of Theorem~\ref{binar d-p}. 
Additional background specific to lattice congruences can be found in Section~\ref{twin sec}.

The canonical join representation is a ``minimal'' expression for an element as a join of join-irreducible elements.
The construction is somewhat analogous to prime factorizations of integers.
Indeed, in the divisibility poset for positive integers, where $p\le q$ if and only if $p|q$, the canonical join representation coincides with prime factorization.
For our purposes, the canonical join representation is useful because of how it interacts with lattice congruences.
Recall that a lattice congruence $\Theta$ contracts a join-irreducible element $j$ if $j$ is equivalent modulo $\Theta$ to the unique element that it covers.
Each congruence $\Theta$ of a finite lattice is determined by the set of join-irreducible elements that it contracts.
In particular, we can see which elements of $W$ are $c$-sortable or $c$-bisortable by looking at their canonical join representations (much as we looked at the arcs in their arc diagrams in types A and B). 

The canonical join representation of an element $a$ is an expression $a=\Join A$ such that $A$ is minimal in two senses, among sets joining to $a$.
First, the join $\Join A$ is \newword{irredundant}, meaning that there is no proper subset $A'\subset A$ with $\Join A'=\Join A$.
Second, $A$ has the smallest possible elements (in terms of the partial order on $L$).
Specifically, a subset $A$ of $L$ \newword{join-refines} a subset $B$ of $L$ if for each $a\in A$ there is an element $b\in B$ such that $a\le b$.
Join-refinement is a preorder on the subsets of $L$ that restricts to a partial order on the set of antichains.
The \newword{canonical join representation} of $a$, if it exists, is the unique minimal antichain $A$, in the sense of join-refinement, that joins irredundantly to $a$.
We sometimes write $\Can(a)$ for $A$. 
The elements of $A$ are called the \newword{canonical joinands} of $a$.
It follows immediately that each canonical joinand is join-irreducible.

Not every finite lattice admits a canonical join representation for each of its elements.
For example, in the diamond lattice $M_3$, which has five elements, three of which are atoms, the largest element does not have a canonical join representation.
Many interesting lattices do admit canonical join representations, including all finite distributive lattices and, as we will see, the weak order on finite Coxeter groups.
The next proposition establishes the promised connection between canonical join representations and lattice congruences.
(The last assertion in the proposition also follows from \cite[Proposition~6.3]{shardint}.)

\begin{proposition}\label{can cong}
Suppose $L$ is a finite lattice such that each element in $L$ has a canonical join representation, and suppose that $\Theta$ is a lattice congruence on $L$.
If $j$ is a canonical joinand of $a\in L$ and $j$ is not contracted by $\Theta$, then $j$ is a canonical joinand of $\pidown^\Theta(a)$ in $L$.
Moreover, if $\pidown^{\Theta}(a)= a$ then none of the canonical joinands of $a$ are contracted by $\Theta$.
\end{proposition}
The assertion that $j$ is a canonical joinand of $\pidown^\Theta(a)$ in $L$ implies also that $j$ is a canonical joinand of $\pidown^\Theta(a)$ in $\pidown^\Theta (L)$.
(Since $\pidown^\Theta(L)$ is a join-sublattice of $L$, every join-representation of $\pidown^\Theta(a)$ in $\pidown^\Theta(L)$ is also a join-representation of $\pidown^\Theta(a)$ in $L$.)

\begin{proof}
Throughout the proof, we write $\set{j_1,\ldots j_k}$ for $\Can(a)$ with $j=j_1$.
Recall that the lattice quotient $L/\Theta$ is isomorphic to the subposet of $L$ induced by the set $\pidown^{\Theta}(L)$.
Suppose $j$ is not contracted by $\Theta$, so that $\pidown^{\Theta}(j)= j$.
Recall that $\pidown^{\Theta}$ is a lattice homomorphism, so $\pidown^\Theta(a)=\Join_{i=1}^k\pidown^\Theta(j_1) = j\join\left(\Join_{i=2}^k\pidown^\Theta(j_i)\right)$, (where the joins are all taken in the lattice quotient $L/\Theta$).
Since $L/\Theta$ is also a join-sublattice of $L$, the join in $L/\Theta$ coincides with the join in $L$.
Thus $\pidown^\Theta(a)$ is equal to $j\join\left(\Join_{i=2}^k\pidown^\Theta(j_i)\right)$ in $L$.
Write $B$ for the set $\Can(\pidown^\Theta(a))$.
Thus $B$ join-refines $\set{j}\cup\set{\pidown^\Theta(j_2),\ldots,\pidown^\Theta(j_k)}$.
If no element of $B$ is less or equal to~$j$, then this join-refinement implies that each element of $B$ is below some element of $\set{\pidown^\Theta(j_2),\ldots,\pidown^\Theta(j_k)}$, so that $\pidown^\Theta(a)\le\Join_{i=2}^k\pidown^\Theta(j_i)$.
Since also $\pidown^\Theta(a)$ is equal to $j\join\left(\Join_{i=2}^k\pidown^\Theta(j_i)\right)$, we see that $j \le \Join_{i=2}^k\pidown^\Theta(j_i)$.
Recall that $\pidown^\Theta (j_i) \le j_i$ for each $i$, so we have $j\le \Join_{i=2}^k j_i$.
This contradicts the fact that $\Join_{i=1}^k j_i$ is irredundant.
We conclude that there is some $j'\in B$ with $j'\le j$.
Observe that $\left(\Join B\right)\join \left(\Join_{i=2}^k j_i\right) = a$ because ${j_1=j\le\pidown^\Theta(a)\le a}$.
Thus, $\{j_1,\ldots j_k\}$ join-refines $B\cup \{j_2,\ldots j_k\}$.
Since $j$ is incomparable to each $j_i$, there is some $j''\in B$ such that $j\le j''$.  
But $B$ is an antichain, so $j'=j''=j$, and thus $j\in B$ as desired.

Now suppose that $\pidown^{\Theta}(a)= a$.
Then $a=\Join_{i=1}^n \pidown^{\Theta}(j_i)$, so $\set{j_1,\ldots j_k}$ join-refines $\set{\pidown^\Theta(j_1),\ldots,\pidown^\Theta(j_k)}$.
Thus, for each $j_i$, there is some $j_m$ with $j_i\le \pidown^{\Theta}(j_m)$.
But $\pidown^{\Theta}(j_m) \le j_m$, and since $\set{j_1,\ldots j_k}$ is an antichain, we have $j_i=j_m$, and thus also $j_i=\pidown^{\Theta}(j_i)$.
\end{proof}

We will use the following easy proposition, which appears as \cite[Proposition~2.2]{arcs}.
\begin{prop}\label{can cplx}
Suppose $L$ is a finite lattice and $J\subset L$.
If $\Join J$ is the canonical join representation of some element of $L$ and if $J'\subseteq J$, then $\Join J'$ is the canonical join representation of some element of $L$.  
\end{prop}

Next we consider canonical join representations in the weak order.
Before we begin, we briefly review some relevant terminology.
For each $w\in W$, the \newword{length} of $w$, denoted $l(w)$, is the number of letters in a reduced (that is, a shortest possible) word for $w$ in the alphabet~$S$.
The covers in the (right) weak order on $W$ are $w\covers ws$ whenever $w\in W$ and $s\in S$ have $l(ws) < l(w)$.
In this case, the simple generator $s$ is a \newword{descent} of~$w$.
Let $T$ denote the set of reflections in $W$.
An \newword{inversion} of $w$ is a reflection $t$ such that $l(tw)<l(w)$. 
We denote the set of inversions of $w$ by $\inv(w)$. 
A \newword{cover reflection} of $w$ is an inversion $t$ of $w$ such that $tw=ws$ for some $s\in S$. 
Thus, the cover reflections of $w$ are in bijection with the descents of $w$.
We write $\cov(w)$ for the set of cover reflections of $w$.
The following proposition is quoted from \cite[Theorem~8.1]{typefree}.
\begin{proposition}\label{coxcjr}
Fix a finite Coxeter group $W$, and an element $w\in W$.
The canonical join representation of $w$ exists and is equal to $\Join j_t$ where $t$ ranges over the set of cover reflections of $w$, and $j_t$ is the unique smallest element below $w$ that has $t$ as an inversion.
In particular, $w$ has $\des(w)$ many canonical joinands.
\end{proposition}

Recall that the \newword{support} of $w$, written $\supp(w)$, is the set of simple reflections appearing in a reduced word for $w$, and is independent of the choice of reduced word for $w$.
The following lemma is an immediate consequence of the fact that every standard parabolic subgroup $W_J$ is a lower interval in the weak order on $W$.

\begin{lemma}\label{supp}
For each $w\in W$, the support of $w$ equals $\bigcup_{j\in\Can(w)}\supp(j)$.
\end{lemma}

For each element $w$ and standard parabolic subgroup $W_J$, there is a unique largest element below $w$ that belongs to $W_J$.
We write $w_J$ for this element and $\pidown^J$ for the map that sends $w$ to $w_J$.
In \cite[Corollary~{6.10}]{congruence}, it was shown that the fibers of $\pidown^J$ constitute a lattice congruence of the weak order. 
We write $\Theta_J$ for this congruence.
Since $\pidown^J$ sends each element to the bottom if its fiber, it is a lattice homomorphism from $W$ to $\pidown^J(W)$, which equals $W_J$.

\begin{lemma}\label{disjoint joinands}  
Suppose that $A_1$ and $A_2$ are antichains with disjoint support such that $\Join A_1$ and $\Join A_2$ are both canonical join representations in the weak order on~$W$.
Then $\Join(A_1\cup A_2)$ is a canonical join representation.
\end{lemma}
\begin{proof}
We write $A$ for $A_1\cup A_2$.
First we show that $\Join A$ is irredundant.
By way of contradiction, assume that there is some $j\in A$ such that $\Join A = \Join (A\setminus \{j\})$.
We may as well take $j\in A_1$.
We write $J$ for the support of $A_1$.
Since the support of each join-irreducible element $j'$ in $A_2$ is disjoint from $J$, and since support decreases weakly in the weak order, we conclude that $\pidown^{J}(j')$ is the identity element.
Since $\pidown^J$ is a lattice homomorphism, $\pidown^{J}(\Join A) = \Join A_1$ and $\pidown^{J}(\Join (A\setminus \{j\}))= \Join(A_1\setminus \{j\})$.
We conclude that ${\Join A_1 = \Join(A_1\setminus \{j\})}$, contradicting the fact that $\Join A_1$ is a canonical join representation.

Next we show that $\Can(\Join A)$ is contained in $A$.
Assume that $j''$ is a canonical joinand of $\Join A$.
There is some $j\in A$ such that $j''\le j$.
Assume that $j\in A_1$, so that $\supp(j'')\subset J$.
Thus, $\pidown^{J}(j'') = j''$. 
Proposition~\ref{can cong} says $j''$ is a canonical joinand of $\pidown^{J}(\Join A) = \Join A_1$.
Because $A$ is an antichain, $j''= j$.
Since $\Join A$ is irredundant, and $A$ contains $\Can(\Join A)$, we conclude that $A$ is equal to $\Can(\Join A)$.
\end{proof}

Observe that if $s\in S$ is a cover reflection of $w$ then Proposition~\ref{coxcjr} implies that $s$ is also a canonical joinand of $w$ because simple reflections are atoms in the weak order.
We immediately obtain the following useful fact.
\begin{lemma}\label{simple_cjr}
Each $w\in W$ has $\Can(w)\cap S=\cov(w)\cap S$.
\end{lemma}

In much of what follows, for $s\in S$, we will use the abbreviation $\br{s}$ to mean $S\setminus\set{s}$.
It is known (see for example \cite[Lemma~2.8]{sort_camb}) that if $w\in W_\br{s}$, then ${\cov(w\join s)} = \cov(w)\cup \{s\}$.
We close this section with a lemma extends this statement to canonical join representations.
\begin{lemma}\label{cover_ref}
If $w\in W_\br{s}$, then $\Can(w\join s) =\Can(w)\cup \{s\}$. 
\end{lemma}
\begin{proof}
Since support is weakly decreasing in the weak order, each $j\in \Can(w)$ has support contained in $\br{s}$.
Lemma~\ref{disjoint joinands} says that $\Join \left(\Can(w)\cup \{s\}\right)$ is a canonical join representation.
\end{proof}

\subsection{Canonical join representations of $c$-bisortable elements}\label{canonical bisort sec}
In this section we focus on canonical join representations of $c$-sortable elements and $c$-bisortable elements.
Our goal is to prove the following result:
\begin{proposition}\label{disjoint_simple_support} 
Fix a bipartite $c$-bisortable element $w$ and the corresponding twin $(c,c^{-1})$-sortable elements $(u,v)=(\pidown^c(w),\pidown^{c^{-1}}(w))$.
Then 
\begin{enumerate}[\qquad\rm(1)]
\item $\Can(w)\cap S=\Can(u)\cap \Can(v)$   
\item $\Can(w)$ is the disjoint union  $(\Can(u)\setminus S)\uplus(\Can(v)\setminus S)\uplus(\Can(w)\cap S)$
\item \label{pw dis} The sets $\supp(\Can(u)\setminus S)$,\, $\supp(\Can(v)\setminus S)$ and $\Can(w)\cap S$ are pairwise disjoint.  
\end{enumerate}
\end{proposition}

We begin with an easy application of Proposition~\ref{can cong} (the first item below can also be found as \cite[Proposition~8.2]{typefree}).
\begin{prop}\label{c or cinv}
For any Coxeter element $c$ and $w\in W$:
\begin{enumerate}
\item $w$ is $c$-sortable if and only if each of its canonical joinands is $c$-sortable.
\item $w$ is $c$-bisortable if and only if each of its canonical joinands is either $c$- or $c^{-1}$-sortable. 
\end{enumerate}
\end{prop}
\begin{proof}  
The first assertion follows immediately from Proposition~\ref{can cong}.
Recall the notation $\Theta_c$ for the $c$-Cambrian congruence and write $\Theta$ for the $c$-biCambrian congruence.
Since $\Theta = \Theta_c\meet \Theta_{c^{-1}}$, a join-irreducible element in $W$ is contracted by $\Theta$ if and only if it is contracted by $\Theta_c$ \textit{and} by $\Theta_{c^{-1}}$.
The second assertion follows.
\end{proof}

Recall from Section~\ref{clus sec} that a simple reflection $s$ is initial in a Coxeter element~$c$ if there is a reduced word $a_1\ldots a_n$ for $c$ with $a_1=s$.
Similarly $s$ is \newword{final} in $c$ if there is a reduced word $a_1\ldots a_n$ for $c$ with $a_n=s$.
In much of what follows, the key property of a bipartite Coxeter element is that every $s\in S$ is either initial or final in $c$.

The following lemma is the combination of \cite[Propositions~3.13, 5.3, and~5.4]{typefree}.
Recall that $v_\br{s}$ is the largest element in $W$ below $v$ that belongs to $W_\br{s}$.
\begin{lemma}\label{s_initial_or_final} 
Fix a $c$-sortable element $v$ in $W$ and a simple reflection $s\in S$.\begin{enumerate}
\item If $s$ is final in $c$ and $v\ge s$, then $v_\br{s}$ is $cs$-sortable and $v=s\join v_\br{s}$.
\item  If $s$ be initial in $c$ and $s\in\cov(v)$, then $v_\br{s}$ is $sc$-sortable and $v=s\join v_\br{s}$.
\end{enumerate}
\end{lemma}

Observe that if $v$ satisfies the conditions of either item in  Lemma~\ref{s_initial_or_final}, then by Lemma~\ref{cover_ref}, $\Can(v) = \{s\}\cup\Can(v_\br{s})$.
The following two lemmas are a straightforward application of Lemma~\ref{s_initial_or_final}.
Lemma~\ref{only ji} is a restatement of Remark~\ref{common walls}, for the special case where $c$ is bipartite, and, for this special case, we give a simpler proof.

\begin{lemma}\label{j eq s lemma}
If $j$ is a $c$-sortable join-irreducible element and $s$ is final in $c$ with $j\ge s$, then $j=s$.
\end{lemma}

\begin{proof}
The first assertion of Lemma~\ref{s_initial_or_final} says that $j= s\join j_\br{s}$.
Since $j$ is join-irreducible and not equal to $j_\br{s}$, we conclude that $j=s$.
\end{proof}

\begin{lemma}\label{only ji}
If $c$ is a bipartite Coxeter element and $j$ is a join-irreducible element that is both $c$-sortable and $c^{-1}$-sortable, then $j$ is a simple reflection. 
\end{lemma}
\begin{proof}
Because $j$ is join-irreducible, it is not the identity, so there is some $s\in S$ such that $j\ge s$. 
Since $c$ is bipartite, we can assume without loss of generality that $s$ is final in~$c$. 
(If not, then replace $c$ with $c^{-1}$.)
Thus $j=s$ by Lemma~\ref{j eq s lemma}.
\end{proof}

Putting together Lemma~\ref{j eq s lemma} and Lemma~\ref{only ji}, we obtain an explicit description of $\pidown^{c^{-1}}(j)$, for bipartite $c$-sortable join-irreducible elements.
\begin{lemma}\label{pidown}  
Suppose that $c$ is a bipartite Coxeter element and $j$ is a $c$-sortable join-irreducible element.
Let $S'$ denote the set of simple reflections $s$ such that $j\ge s$.  
Then $\pidown^{c^{-1}}(j) = \Join S'$, which equals $\prod S'$, the product in $W$.
Moreover, this join is a canonical join representation.
\end{lemma}
\begin{proof}

The statement of the lemma is obvious if $j$ is a simple reflection, so we assume that $j$ is not simple.
Thus, Lemma~\ref{only ji} implies that $j$ is not $c^{-1}$-sortable, so $\pidown^{c^{-1}}(j)$ is strictly less than $j$. 

If any $s\in S'$ is final in $c$, then Lemma~\ref{j eq s lemma} says that $j=s$, contradicting our assumption.
Thus, since $c$ is bipartite, each $s\in S'$ is initial.
In particular, the elements of $S'$ pairwise commute, so that the notation $\prod S'$ makes sense and equals $\Join S'$.
Moreover, since $\Join S'$ is an irredundant join of atoms, it is a canonical join representation.
Since each simple reflection is both $c$- and $c^{-1}$-sortable, Proposition~\ref{c or cinv} says that this element is $c^{-1}$-sortable.
We conclude that $\pidown^{c^{-1}}(j)\ge \Join S'$. 

Suppose that $j'$ is a canonical joinand of $\pidown^{c^{-1}}(j)$.
There is some simple reflection $s$ such that $j'\ge s$.
Since also $j'\le \pidown^{c^{-1}}(j)\le j$, we conclude that $s\in S'$.
Every element of $S'$ is initial in $c$ and thus final in $c^{-1}$, so again by Lemma~\ref{j eq s lemma}, $j'=s$.
We conclude that $\Can(\pidown^{c^{-1}}(j))\subseteq S'$.
Thus $\pidown^{c^{-1}}(j) = \Join S'$. 
\end{proof}

Recall that Lemma~\ref{disjoint joinands} says that if $j$ and $j'$ are join-irreducible elements with disjoint support, then $j\join j'$ is canonical.
In Lemma~\ref{simpler_disjoint_support} below, we prove that when $j$ is bipartite $c$-sortable and $j'$ is bipartite $c^{-1}$-sortable, the converse is also true.
We begin with the case when $j'$ is a simple reflection.
\begin{lemma}\label{simple_covers}  
Given a bipartite Coxeter element $c$, a $c$-sortable join-irreducible element $j$ and a simple reflection $s\in\supp(j)$, there exists no element $w\in W$ with both $s$ and $j$ in $\Can(w)$.
\end{lemma}
\begin{proof}  
In light of Proposition~\ref{can cplx}, to prove this proposition, it is enough to show that no element can have $s\join j$ as its canonical join representation.
Suppose to the contrary that there is an element $v$ with canonical join representation $s\join j$.
By Proposition~\ref{c or cinv}, $v$ is $c$-sortable.  
Also $s\join j$ is irredundant, so $j$ and $s$ are incomparable.  
Since $c$ is bipartite, $s$ is either initial or final in $c$, so Lemma~\ref{s_initial_or_final} says that $v = s\join v_\br{s}$.
Since $v=s\join j$ is a canonical join representation, we see that $j\le v_\br{s}$, contradicting the hypothesis that $s$ is in the support of $j$.
\end{proof}

\begin{lemma}\label{simpler_disjoint_support}
Fix a bipartite Coxeter element $c$ in $W$.
Suppose that $j$ is a $c$-sortable join-irreducible element and that $j'$ is a $c^{-1}$-sortable join-irreducible element.
Suppose that $j \join j'$ is a canonical join representation for some element of~$W$.
Then $j$ and $j'$ have disjoint support.
\end{lemma}
\begin{proof} 
Suppose that $s\in \supp(j)\cap\supp(j')$, and assume without loss of generality that $s$ is initial in $c$.  
It is immediate from the definition of $c$-sortable elements that $s\le j$.
(See for example \cite[Proposition~2.29]{typefree}.)
Since $s$ is a $c^{-1}$-sortable element, also $s\le\pidown^{c^{-1}}(j\join j')$.
By Lemma~\ref{s_initial_or_final}(1)  and Lemma~\ref{cover_ref}, $s$ is a canonical joinand of $\pidown^{c^{-1}}(j\join j')$.
But also Proposition~\ref{can cong} says that $j'$ is a canonical joinand of $\pidown^{c^{-1}}(j\join j')$.
We have reached a contradiction to Lemma~\ref{simple_covers}, and we conclude that $\supp(j)\cap\supp(j')=\emptyset$.
\end{proof}

Finally, we prove Proposition~\ref{disjoint_simple_support}.  
\begin{proof}[Proof of Proposition~\ref{disjoint_simple_support}]
Lemma~\ref{only ji} implies that $\Can(w)$ is the disjoint union $(\Can(w)\cap S)\uplus J_+\uplus J_-$ such that $J_+$ is the set of $c$-sortable join-irreducible elements in $\Can(w)\setminus S$ and $J_-$ is the set of $c^{-1}$-sortable join-irreducible elements in ${\Can(w)\setminus S}$.
Moreover, by Lemma~\ref{simpler_disjoint_support}, these sets have pairwise disjoint support.
For each $j\in J_-$, write $S'_j$ for the set of simple reflections $s$ such that $s\le j$, and $S'= \bigcup S'_j$, where the union ranges over all $j\in J_-$.
Lemma~\ref{pidown} says that ${\pidown^c(j) = \Join S'_j}$.
Since $\pidown^c$ is a join-homomorphism, $\pidown^c(\Join J_-) = \Join S'$.
Thus, applying the map $\pidown^{c}$ to the join $\Join[(\Can(w)\cap S)\uplus J_+\uplus J_-]$, we see that ${\Join[(\Can(w)\cap S)\uplus J_+ \uplus S']}$ is a join representation of $u$.
Since $S'$ is contained in the support of $J_-$, the sets $\Can(w)\cap S$, $J_+$, and $S'$ also have pairwise disjoint support.
Proposition~\ref{can cplx} says that both $\Join \Can(w)\cap S$ and $\Join J_+$ are canonical join representations.
Since $\Join S'$ is an irredundant join of atoms, it is also a canonical join representation.
Thus, by Lemma~\ref{disjoint joinands},  $\Join[(\Can(w)\cap S)\uplus J_+ \uplus S']$ is the canonical join representation of $u$.
The symmetric argument gives the canonical join representation of $v$.
We conclude that $\Can(w)\cap S = \Can(u)\cap \Can(v)$, ${J_+=\Can(u)\setminus S}$, and $J_-=\Can(v) \setminus S$.
The proposition follows.
\end{proof}

\subsection{Counting bipartite $c$-bisortable elements}\label{sortable formula sec} 
In this section, we prove that the formulas in Theorem~\ref{binar d-p} counts bipartite $c$-bisortable elements, thus completing the proofs of Theorems~\ref{hard part}, \ref{hard part finer}, \ref{bicat d-p}, \ref{biCat GF d-p} and~\ref{binar d-p}.
We begin by interpreting the double-positive Catalan and Narayana numbers in the $c$-sortable setting.
We define \newword{positive $c$-sortable elements} to be the set of $c$-sortable elements not contained in any standard parabolic subgroup of $W$.
Equivalently, these are the $c$-sortable elements whose support is not contained in any proper subset of $S$. 
As the name suggests, positive $c$-sortable elements are counted by the positive Catalan numbers.
The following analogue of Proposition~\ref{i-e nn} is the combination of \cite[Corollary~9.2]{sortable} and \cite[Corollary~9.3]{sortable}.
\begin{proposition}\label{positive_sortable} 
For any Coxeter element $c$ of $W$, the number of positive $c$-sortable elements in $W$ is $\Cat^+(W)$.
The number positive $c$-sortable elements with $k$ descents is $\Nar^+_k(W)$.
\end{proposition}

We define \newword{clever $c$-sortable elements} to be $c$-sortable elements which have no simple canonical joinands.
We continue to let $\br{s}$ stand for $S\setminus \{s\}$.
To count clever $c$-sortable elements we will use Lemma~\ref{s_initial_or_final} to define a map from $c$-sortable elements $v$ with simple cover reflection $s$ to $c'$-sortable elements in the standard parabolic subgroup $W_\br{s}$, where $c'$ is the restriction of $c$ to $W_\br{s}$.
Our next task is to show that, for bipartite $c$, clever $c$-sortable elements are analogous, enumeratively, to antichains in the root poset having no simple roots:

\begin{proposition}\label{i-e sortable}  
Fix a bipartite Coxeter element $c$ of $W$.
\begin{enumerate}
\item The number of clever $c$-sortable elements is $\Cat^+(W)$. 
\item The number of positive, clever $c$-sortable elements is $\Cat\pp(W)$.
\item The number of positive, clever $c$-sortable elements with exactly $k$ descents is $\Nar\pp_k(W)$.
\end{enumerate}
\end{proposition}

We emphasize that while Proposition~\ref{positive_sortable} holds for arbitrary $c$, Proposition~\ref{i-e sortable} holds only for bipartite $c$.
The proof of Proposition~\ref{i-e sortable} will use inclusion-exclusion and the following technical lemma.
\begin{lemma}\label{tech lemma}
For bipartite $c$ and $J\subseteq S$, let $c'$ be the restriction of $c$ to $W_{S\setminus J}$.
\begin{enumerate}
\item The map $\pidown^{S\setminus J}:v\mapsto v_{S\setminus J}$ is a bijection from $c$-sortable elements of $W$ with ${J\subseteq\Can(v)}$ to $c'$-sortable elements of $W_{S\setminus J}$.   
Also, $\Can(v_{S\setminus J})=\Can(v)\setminus J$.
\item The map restricts to a bijection from positive $c$-sortable elements of $W$ with $J\subseteq\Can(v)$ to positive $c'$-sortable elements of $W_{S\setminus J}$.
\item The map restricts further to a bijection from positive $c$-sortable elements of $W$ with $J\subseteq\Can(v)$ and with exactly $k$ descents to positive $c'$-sortable elements of $W_{S\setminus J}$ with exactly $k-|J|$ descents.
\end{enumerate}
\end{lemma}
\begin{proof}
Suppose that $v$ is $c$-sortable, and $J\subseteq\Can(v)$.
Lemma~\ref{simple_covers} says that the support of each canonical joinand $j$ in $\Can(v)\setminus J$ is contained in $S\setminus J$.
(Lemma~\ref{simple_covers} applies to the non-simple elements of $\Can(v)$. Clearly, each simple reflection $s\in \Can(v)\setminus J$ is supported on $S\setminus J$.)
On the one hand, $\pidown^{S\setminus J}(j) = j$ for each $j\in \Can(v)\setminus J$.
On the other hand, $\pidown^{S\setminus J}(s)$ is the identity element for each $s$ in $J$.
Since $\pidown^{S\setminus J}$ is a lattice homomorphism, $\pidown^{S\setminus J}(\Join \Can(v)) = \Join [\Can(v) \setminus J]$.
Proposition~\ref{can cong} implies that $\Join [\Can(v) \setminus J]$ is the canonical join representation of $\pidown^{S\setminus J}(v) = v_{S\setminus J}$.
Lemma~\ref{c or cinv} says that $v_{S\setminus J}$ is $c'$-sortable.

To complete the proof of the first assertion, we construct an inverse map.
Suppose that $v'$ is a $c'$-sortable element in $W_{S\setminus J}$.
Lemma~\ref{supp} says that the support of each canonical joinand $j\in \Can(v')$ is contained in $S\setminus J$.
Lemma~\ref{disjoint joinands} says that the join $\Join [\Can(v')\cup J]$ is a canonical join representation for some element $v\in W$.
Lemma~\ref{c or cinv} says that $v$ is $c$-sortable.
We conclude that the map sending $v'$ to $\Join [\Can(v')\cup J]$ is a well-defined inverse.

Lemma~\ref{supp}, Lemma~\ref{simple_covers}, and the fact that $\Can(v_{S\setminus J})=\Can(v)\setminus J$ imply that $v$ is positive in $W$ if and only if $v_{S\setminus J}$ is positive in $W_{S\setminus J}$.
The second assertion follows.
The third assertion then follows from Proposition~\ref{coxcjr} and the fact that $\Can(v_{S\setminus J})=\Can(v)\setminus J$.
\end{proof}

Finally, we complete the proof of that bipartite $c$-bisortable elements are counted by the formula in Theorem~\ref{biCat GF d-p}.
\begin{prop}\label{q bisortable formula}
For any finite Coxeter group $W$ with simple generators $S$, the generating function $\sum_vq^{\des(v)}$ for bipartite $c$-bisortable elements is
\[\sum\, q^{|M|}\Cat\pp(W_{I};q)\Cat\pp(W_{J};q),\]
where the sum is over all ordered triples $(I,J,M)$ of \emph{pairwise disjoint} subsets of~$S$.
\end{prop}
\begin{proof}
Similarly to the proof of Proposition~\ref{q antichain}, the proposition amounts to establishing a bijection from bipartite $c$-bisortable elements $w$ to triples $(u',v',M)$ such that $u'$ is a clever $c$-sortable element, $v'$ is a clever $c^{-1}$-sortable element, and the sets $I=\supp(u')$, $J=\supp(v')$, and $M$ are disjoint subsets of $S$, and then showing that $\des(w)=\des(u')+\des(v')+|M|$.

Given a bipartite $c$-bisortable element $w$, write $(u,v)$ for the corresponding pair $(\pidown^c(w),\pidown^{c^{-1}}(w))$ of twin $(c,c^{-1})$-sortable elements.
Proposition~\ref{disjoint_simple_support}(2) says that $\Can(w)$ is the disjoint union $\left(\Can(u)\setminus S\right)\uplus \left(\Can(v)\setminus S\right)\uplus\left( \Can(w)\cap S\right)$.
Proposition~\ref{disjoint_simple_support}(\ref{pw dis}) says that the sets $I=\supp(\Can(u)\setminus S)$, $J=\supp(\Can(v)\setminus S)$, and $M=\Can(w)\cap S$ are pairwise disjoint subsets of $S$.
By Proposition~\ref{can cplx}, $\Join \Can(u)\setminus S$ is the canonical join representation of a positive, clever $c$-sortable element $u'$ in~$W_I$.
Similarly, $\Join \Can(v)\setminus S$ is the canonical join representation of a positive, clever $c^{-1}$-sortable element $v'$ in $W_J$.  
Applying Proposition~\ref{coxcjr} several times, we see that $\des(w) = \des(u')+\des(v')+|M|$.

We will show that this map $w\mapsto(u',v',M)$ is a bijection by showing that the map $(u',v',M)\mapsto u'\join v'\join (\Join M)$ is the inverse.
On one hand, given $w$, construct $(u',v',M)$ as above.
Then $w$ equals $\Join\Can(w)$, which equals 
\[\left(\Join\Can(u)\setminus S\right)\join \left(\Join\Can(v)\setminus S\right)\join\left(\Join\Can(w)\cap S\right)=u'\join v'\join (\Join M).\]
On the other hand, given a triple $(u',v',M)$ satisfying the description above, set ${w=u'\join v'\join (\Join M)}$. 
Since $u'$, $v'$ and $M$ have pairwise disjoint support, we conclude that $\Can(u')$, $\Can(v')$, and $M$ also have pairwise disjoint support.
Lemma~\ref{disjoint joinands} says that $\Join \Can(u')\uplus \Can(v') \uplus M$ is the canonical join representation of $w$.
By Lemma~\ref{c or cinv}(1), each canonical joinand of $u'$ is $c$-sortable and each canonical joinand of $v'$ is $c^{-1}$-sortable.
Since each simple generator is both $c$- and $c^{-1}$-sortable, we conclude that each canonical joinand of $w$ either either $c$- or $c^{-1}$-sortable.
By Lemma~\ref{c or cinv}(2), $w$ is $c$-bisortable.
Thus, the map $(u',v', M) \mapsto u'\join v'\join (\Join M)$ is a well-defined.

Lemma~\ref{only ji} says that $\Can(u')\uplus M$ is equal to the set of $c$-sortable canonical joinands of $w$.
Since $u'$ is clever, $\Can(u')$ is equal to the set of $c$-sortable canonical joinands in $\Can(w)\setminus S$.
Similarly, $\Can(v')$ is the set of $c^{-1}$-sortable canonical joinands in $\Can(w)\setminus S$, and $\Can(w)\cap S = M$.
Define $u=\pidown^c(w)$ and $v=\pidown^{c^-1}(w)$.
Proposition~\ref{disjoint_simple_support}(2) says that $\Can(w)=(\Can(u)\setminus S)\uplus (\Can(v) \setminus S)\uplus (\Can(w)\cap S)$.
Comparing this to the expression $\Can(w)=\Can(u')\uplus\Can(v')\uplus M$, we see that $\Can(u)\setminus S=\Can(u')$, that $\Can(v)\setminus S=\Can(v')$, and that $\Can(w)\cap S=M$.
Thus the map described above takes $w$ back to $(u',v',M)$.
\end{proof}

\begin{remark}\label{uniform uniform}
The proof given here that twin nonnesting partitions are in bijection with bipartite $c$-bisortable elements would be uniform if there were a uniform proof connecting $c$-sortable elements and nonnesting partitions.
The opposite is true as well: 
Suppose one proved uniformly that a given map $\phi$ is a bijection from antichains in the doubled root poset to bipartite $c$-bisortable elements and also that $\phi$ preserves the triples $(I,J,M)$ appearing in Propositions~\ref{q antichain} and~\ref{q bisortable formula}.
Then the restriction of $\phi$ to antichains in the root poset (i.e.\ those with $J=\emptyset$) is a bijection from antichains in the root poset to $c$-sortable elements.
\end{remark}

\begin{remark}\label{other c}  
The methods of this section don't apply well to the case where $c$ is not bipartite, because the main structural results of the section, Propositions~\ref{disjoint_simple_support} and~\ref{i-e sortable}, can fail when $c$ is not bipartite.
We now describe how both propositions fail for linear $c$ in type $A_3$.
By analogy to Proposition~\ref{avoid alt}, the $c$-bisortable elements for linear $c$ are in bijection with noncrossing arc diagrams such that every arc either passes only left of points or passes only right of points.
(Each arc in the diagram corresponds to a canonical joinand.
See Remark~\ref{type A insight} or \cite[Example~4.10]{arcs}.)
Taking $w=3241$ and $u$ as in Proposition~\ref{disjoint_simple_support}, the noncrossing arc diagram $\delta(w)$ has a right arc connecting $1$ to $4$ and an arc (which is both a left arc and a right arc) connecting $2$ to~$3$.
Thus $\Can(w)\cap S=\set{s_2}$.
Also $u=w=3241$, so $\Can(u)\setminus S$ corresponds to the arc connecting $1$ to $4$, which has support $\set{s_1,s_2,s_3}$, contradicting Proposition~\ref{disjoint_simple_support}\eqref{pw dis}.
The $c$-sortable elements are in bijection with noncrossing arc diagrams such that every arc only passes right of points.
From there, we easily see that there is only $1$ positive, clever $c$-sortable element, contradicting Proposition~\ref{i-e sortable}(2).
\end{remark}

\subsection{BiCatalan and Catalan formulas}\label{recursions sec}
In this section and the next, we prepare to prove the formula for $\biCat(D_n)$ in Theorem~\ref{enum thm}, thus completing the proof of that theorem.
Specifically, the proof requires combining a very large number of identities relating $q$-analogs of biCatalan numbers, Catalan numbers, and double-positive Catalan numbers that we quote or prove here.
In this section, we give recursions for the $q$-analogs of $W$-biCatalan and $W$-Catalan numbers for irreducible finite Coxeter groups, in which $q$-analogs of double-positive Catalan numbers appear as coefficients.

\begin{prop}\label{q biCat recursion}  
For an irreducible finite Coxeter group $W$ and a simple generator $s\in S$, the $q$-analog of the $W$-biCatalan number satisfies
\begin{multline}\label{q biCat recursion formula}
\biCat(W;q)=(1+q)\biCat(W_{S\setminus\set{s}};q)\\
+2\sum_{S_0}\Cat\pp(W_{S_0};q)\prod_{i=1}^m\left[\frac12\biCat(W_{S_i};q)+\frac{1+q}2\biCat(W_{S_i\setminus\set{s_i}};q)\right],
\end{multline}
where the sum is over all connected subgraphs $S_0$ of the diagram for $W$ with $s\in S_0$, the connected components of the complement of $S_0$ in the diagram are $S_1,\ldots,S_m$, and each $s_i$ is the unique vertex in $S_i$ that is connected by an edge to a vertex in~$S_0$.
\end{prop}

\begin{proof}
For fixed $s$, we break the formula in Theorem~\ref{biCat GF d-p} into four sums, according to whether $s$ is in $S\setminus(I\cup J\cup M)$, in $M$, in $I$, or in $J$.
The sum of terms with $s\in S\setminus(I\cup J\cup M)$ equals $\biCat(W_{S\setminus\set{s}};q)$.
The sum of terms with $s\in M$ equals $q\cdot\biCat(W_{S\setminus\set{s}};q)$.

Consider next the sum of terms with $s\in I$, and in each term let $S_0$ be the connected component of the diagram containing $s$.
Using \eqref{Cat++ reducible formula}, we can reorganize the sum according to $S_0$ to obtain
\[\sum_{S_0}\Cat\pp(W_{S_0};q)\sum q^{|M|}\Cat\pp(W_{I'};q)\Cat\pp(W_J;q),\]
where the $S_0$-sum is as described in the statement of the proposition and the inner sum is over all ordered triples $(I',J,M)$ of disjoint subsets of $S\setminus S_0$ such that no element of $I'$ is connected by an edge of the diagram to an element of $S_0$.
Again using \eqref{Cat++ reducible formula}, we factor the inner sum further to obtain
\[\sum_{S_0}\Cat\pp(W_{S_0};q)\prod_{i=1}^m\left[\sum q^{|M_i|}\Cat\pp(W_{I_i};q)\Cat\pp(W_{J_i};q)\right],\]
where the $S_i$ and $s_i$ are as in the statement of the proposition and the inner sum runs of over all ordered triples $(I_i,J_i,M_i)$ of pairwise disjoint subsets of $S_i$ with $s_i\not\in I_i$.
The sum for each $i$ can be broken up into a sum over terms with $s_i \in J_i$ and terms with $s_i \not\in J_i$.
Splitting the sum over terms with $s_i \not\in J_i$ in half, we obtain three sums:
\begin{multline*}
\sum_{s_i\in J_i} q^{|M_i|}\Cat\pp(W_{I_i};q)\Cat\pp(W_{J_i};q)\\
+\frac12\sum_{s_i\not\in J_i} q^{|M_i|}\Cat\pp(W_{I_i};q)\Cat\pp(W_{J_i};q)\\
+\frac12\sum_{s_i\not\in J_i} q^{|M_i|}\Cat\pp(W_{I_i};q)\Cat\pp(W_{J_i};q)
\end{multline*}
The symmetry between $I$ and $J$ on the right side of Theorem~\ref{biCat GF d-p} lets us recognize the sum of the first two terms as $\frac12\biCat(W_{S_i};q)$, recalling that $s\not\in I_i$ throughout.
The third term is $\frac{1+q}2\biCat(W_{S_i\setminus\set{s_i}};q)$.
We see that the sum of terms with $s\in I$ is the sum in the proposed formula, without the factor $2$ in front.
By symmetry, the sum of terms with $s\in J$ is the same sum, so we obtain the factor $2$ in the sum and we have established the desired formula.
\end{proof}

We obtain the following recursion for $\biCat(D_n;q)$ from Proposition~\ref{q biCat recursion}.
The notation $D_2$ means $A_1\times A_1$ and $D_3$ means $A_3$.


\begin{prop}\label{q biCat recursion D}  
For $n\ge 3$,   
\begin{multline}\label{q biCat recursion D formula}
\biCat(D_n;q)=(1+q)\biCat(D_{n-1};q)\\
+\sum_{i=1}^{n-3}\Cat\pp(A_i;q)\left(\biCat(D_{n-i};q)+(1+q)\biCat(D_{n-i-1};q)\right)\\
+2(1+q)^2\Cat\pp(A_{n-2};q)+4(1+q)\Cat\pp(A_{n-1};q)+2\Cat\pp(D_n;q)
\end{multline}
\end{prop}
\begin{proof}
In Proposition~\ref{q biCat recursion}, take $s$ to be a leaf of the $D_n$ diagram whose removal leaves the diagram for $D_{n-1}$.
The sum over $S_0$ splits into several pieces.
First, the $S_0$ for which the diagram on $S\setminus\set{S_0}$ is of type $D_k$ for $k\ge 3$ give rise to terms $\sum_{i=1}^{n-3}\Cat\pp(A_i;q)\left(\biCat(D_{n-i};q)+(1+q)\biCat(D_{n-i-1};q)\right)$.
Next, the term for which the diagram on $S\setminus\set{S_0}$ is of type $D_2$ is $2\Cat\pp(A_{n-2})(\frac12(1+q)+\frac{1+q}2\cdot1)^2$, which simplifies to $2(1+q)^2\Cat\pp(A_{n-2})$.
The two terms for which the diagram on $S\setminus\set{S_0}$ is of type $A_1$ \emph{each} contribute $2(1+q)\Cat\pp(A_{n-1})$.
Finally, the term with $S_0=S$ is $2\Cat\pp(D_n;q)$.
\end{proof}

We obtain the following recursion for $\biCat(B_n;q)$ from Proposition~\ref{q biCat recursion} similarly.
Here and throughout the paper, we interpret $B_0$ and $B_1$ to be synonyms for $A_0$ and $A_1$.
\begin{prop}\label{q biCat recursion B}  
For $n\ge1$, 
\begin{multline}\label{q biCat recursion B formula}
\biCat(B_n;q)=\\(1+q)\biCat(B_{n-1};q)+2\Cat\pp(B_n;q)+2(1+q)\Cat\pp(A_{n-1};q)\\
+\sum_{i=1}^{n-2}\Cat\pp(A_i;q)\left[\biCat(B_{n-i};q)+(1+q)\biCat(B_{n-i-1};q)\right].
\end{multline}
\end{prop}
\begin{proof}
In Proposition~\ref{q biCat recursion}, take $s$ to be a leaf of the $B_n$ diagram whose removal leaves the diagram for $B_{n-1}$.
The terms with $|S_0|$ from $1$ to $n-2$ are in the summation in \eqref{q biCat recursion B formula}, but we separate out the terms with $|S_0|=n-1$ and $|S_0|=n$.
For the term with $|S_0|=n-1$, we use the facts that $\biCat(B_1;q)=(1+q)$ and that $\biCat(B_0;q)=1$.
\end{proof}

Similarly, we obtain the following recursion for $\biCat(A_n)$ by taking $s$ to be either leaf of the diagram.

\begin{prop}\label{q biCat recursion A}  
For $n\ge1$, 
\begin{multline}\label{q biCat recursion A formula}
\biCat(A_n;q)=(1+q)\biCat(A_{n-1};q)+2\Cat\pp(A_n;q)\\
+\sum_{i=1}^{n-1}\Cat\pp(A_i;q)\left[\biCat(A_{n-i};q)+(1+q)\biCat(A_{n-i-1};q)\right].
\end{multline}
\end{prop}

Next we gather some formulas involving the $q$-Catalan numbers.
We begin with the usual recursion for the type-A Catalan numbers, although this $q$-version may be less widely familiar.
It is easily obtained through the interpretation of $\Cat(A_n;q)$ as the descent generating function for $231$-avoiding permutations in $S_{n+1}$, by breaking up the count according to the first entry in the permutation.
We omit the details.
\begin{prop}\label{Cat A q}  
For $n\ge1$, 
\begin{equation}\label{Cat A q formula}
\Cat(A_n;q)=(1+q)\Cat(A_{n-1};q)+q\sum_{i=1}^{n-1}\Cat(A_{i-1};q)\Cat(A_{n-i-1};q).
\end{equation}
\end{prop}

Furthermore, using known formulas for the Narayana numbers, we obtain a recursion that relates the $q$-Catalan number in types A and D.
\begin{prop}\label{Cat D A q}
For $n\ge2$,
\begin{equation}\label{Cat D A formula}
\Cat(D_n;q) = \frac{n+1}2(1+q)\Cat(A_{n-1};q)-\Bigl(\frac{n-1}2+q+\frac{n-1}2q^2\Bigr)\Cat(A_{n-2}).
\end{equation}
\end{prop}
\begin{proof}
Taking the coefficient of $q^k$ on both sides, we see that \eqref{Cat D A formula} is equivalent to 
\begin{multline}\label{Nar D A formula}
\Nar_k(D_n) = \frac{n+1}2\left(\Nar_k(A_{n-1})+\Nar_{k-1}(A_{n-1})\right)\\-\frac{n-1}2\Nar_k(A_{n-2})-\Nar_{k-1}(A_{n-2})-\frac{n-1}2\Nar_{k-2}(A_{n-2}).
\end{multline}
This can be verified using the known formulas for the type-A and type-D Narayana numbers.
(See for example in \cite[(9.1)]{gcccc} and \cite[(9.3)]{gcccc}, putting $m=1$ in both formulas).
\end{proof}

Next, we give a recursion for $\Cat(W;q)$ analogous to \eqref{q biCat recursion formula}.
The proof follows the outline of the proof of Proposition~\ref{q biCat recursion}, using Theorem~\ref{Cat GF d-p} instead of Theorem~\ref{biCat GF d-p}.
This proof is simpler than the proof of Proposition~\ref{q biCat recursion}, so we omit the details.

\begin{prop}\label{q Cat recursion}  
For an irreducible finite Coxeter group $W$ and a simple generator~$s$, the $q$-analog of the $W$-Catalan number satisfies
\begin{multline}\label{q Cat recursion formula}
\Cat(W;q)=(1+q)\Cat(W_{S\setminus\set{s}};q)\\
+\sum_{S_0}\Cat\pp(W_{S_0};q)\prod_{i=1}^m(1+q)\Cat(W_{S_i\setminus\set{s_i}};q),
\end{multline}
where the sum is over all connected subgraphs $S_0$ of the diagram for $W$ with $s\in S_0$, the connected components of the complement of $S_0$ in the diagram are $S_1,\ldots,S_m$, and each $s_i$ is the unique vertex in $S_i$ that is connected by an edge to a vertex in~$S_0$.
\end{prop}

The following three propositions give the type-A, type-B, and type-D cases of \eqref{q Cat recursion formula}.

\begin{prop}\label{q Cat recursion A}
For $n\ge 0$,
\begin{multline}\label{q Cat recursion A formula}
\Cat(A_n;q)=\Cat\pp(A_n;q)+(1+q)\sum_{i=0}^{n-1}\Cat\pp(A_i;q)\Cat(A_{n-i-1};q).
\end{multline}
\end{prop}
\begin{proof}
If $n=0$, then the formula is $\Cat(A_0;q)=\Cat\pp(A_0;q)$, which says ${1=1}$.
Otherwise, taking $s$ to be a leaf of the $A_n$ diagram in \eqref{q Cat recursion formula}, the sum over $S_0$ has terms $\sum_{i=1}^{n-1}\Cat\pp(A_i;q)(1+q)\Cat(A_{n-i-1};q)$ and $\Cat\pp(A_n;q)$.
Because $\Cat\pp(A_0;q)=1$, we can merge the first term into the sum.
\end{proof}

\begin{prop}\label{q Cat recursion B 2}
For $n\ge 0$,
\begin{multline}\label{q Cat recursion B 2 formula}  
\Cat(B_n;q)=\Cat\pp(B_n;q)+(1+q)\sum_{i=0}^{n-1}\Cat\pp(A_i;q)\Cat(B_{n-i-1};q) 
\end{multline}
\end{prop}
\begin{proof}
The formula holds for $n=0$ and $n=1$.
For $n>1$, take $s$ to be the leaf whose deletion leaves a diagram of type $B_{n-1}$ in \eqref{q Cat recursion formula}, and rearrange the formula as in the proof of Proposition~\ref{q Cat recursion A formula}.
\end{proof}

\begin{prop}\label{q Cat recursion D}  
For $n\ge3$, 
\begin{multline}\label{q Cat recursion D formula}
\Cat(D_n;q)=\\
(1+q)\Cat(A_{n-1};q)+(1+q)\Cat\pp(A_{n-1};q)+\Cat\pp(D_n;q)\\
+(1+q)^2\sum_{i=1}^{n-2}\Cat\pp(A_i;q)\Cat(A_{n-i-2};q)\\
+(1+q)\sum_{i=3}^{n-1}\Cat\pp(D_i;q)\Cat(A_{n-i-1};q).\\
\end{multline}
\end{prop}
\begin{proof}
Start with Proposition~\ref{q Cat recursion}, taking $s$ to be a leaf of the $D_n$ diagram whose removal leaves the diagram for $A_{n-1}$.
For $S_0$ not containing the leaf symmetric to $s$, we get terms
$(1+q)^2\sum_{i=1}^{n-2}\Cat\pp(A_i;q)\Cat(A_{n-i-2};q)$ and $(1+q)\Cat\pp(A_{n-1};q)$.
(The $i=1$ term in the sum would be wrong, except that $\Cat\pp(A_1)=0$.)
For $S_0$ containing the leaf symmetric to $s$, we get $(1+q)\sum_{i=3}^{n-1}\Cat\pp(D_i;q)\Cat(A_{n-i-1};q)$ and $\Cat\pp(D_n;q)$.
\end{proof}

\subsection{The double-positive Catalan numbers}\label{Cat++ sec}
In this section, we consider the double-positive Catalan numbers for the classical reflection groups, and establish some identities for $\Cat\pp(A_n;q)$, $\Cat\pp(B_n;q)$, and $\Cat\pp(D_n;q)$ that will be useful for proving the type-D case of Theorem~\ref{enum thm}.
As an aside at the end of this section, we give formulas and computed values of the numbers $\Cat\pp(W)$ for all finite types.
(See Theorem~\ref{Cat++ thm}.)

\begin{remark}   
Athanasiadis and Savvidou, in \cite[Theorem~1.2]{Ath-Sav}, gave formulas for the polynomials $\Cat\pp(W;q)$ for each $W$ of finite type by explicitly determining coefficients $\xi_i$ such that $\Cat\pp(W;q)=\sum_{i=0}^{\lfloor n/2\rfloor}\xi_iq^i(1+q)^{n-2i}.$
Similar formulas for the relevant polynomials $\Cat(W;q)$ are known \cite[Propositions~11.14--11.15]{PRW}, so the identities we need can in principle be obtained by manipulating the formulas from \cite{Ath-Sav,PRW}.
Indeed, Proposition~\ref{Cat++ D q} is easily obtained in this way, but such proofs of Propositions~\ref{Cat++ A q} and~\ref{Cat++ Bn-1} appear to be more complicated.
\end{remark}

To motivate the propositions that follow, we list here some examples of the double-positive Catalan numbers for the classical reflection groups.   
{\small
\[\begin{array}{c||c|c|c|c|c|c|c||c|c|c|c|c||c|c|c|c}
\!W\!&\!A_0\!&\!A_1\!&\!A_2\!&\!A_3\!&\!A_4\!&\!A_5\!&\!A_6\!&\!B_2\!&\!B_3\!&\!B_4\!&\!B_5\!&\!B_6\!&\! D_4\!&\!D_5\!&\!D_6\!&\! D_7\\\hline
&&&&&&&&&&&&&&&&\\[-9pt]
\!\Cat\pp(W)\!&\!1\!&\!0\!&\!1\!&\!2\!&\!6\!&\!18\!&\!57\!&\!2\!&\!6\!&\!22\!&\!80\!&\!296\!&\!10\!&\!42\!&\!168\!&\!660
\end{array}\,\]
}


From inspection of these numbers, several interesting relationships appear.
First, the data suggests that $2\Cat\pp(A_n)+\Cat\pp(A_{n-1}) = \Cat(A_{n-1})$.  Below, we establish a $q$-analog of this identity.  

\begin{prop}\label{Cat++ A q}
For $n\ge1$, 
\begin{equation}\label{Cat++ A q formula}
(1+q)\Cat\pp(A_n;q)+q\Cat\pp(A_{n-1};q)=q\Cat(A_{n-1};q).
\end{equation}
\end{prop}
\begin{proof}  
If $n=1$, then the identity is $(1+q)\cdot0+q\cdot1=q\cdot1$.
If $n>1$, then by induction, we can replace $(1+q)\Cat\pp(A_i;q)$ with $q(\Cat(A_{i-1};q)-\Cat\pp(A_{i-1};q))$ in the terms $i>1$ of \eqref{q Cat recursion A formula} and observe that $\Cat\pp(A_0;q)=1$ to obtain
\begin{multline*}
\Cat(A_n;q)=\Cat\pp(A_n;q)+(1+q)\Cat(A_{n-1};q)\\
+q\sum_{i=1}^{n-1}\Cat(A_{i-1};q)\Cat(A_{n-i-1};q)\\
-q\sum_{i=1}^{n-1}\Cat\pp(A_{i-1};q)\Cat(A_{n-i-1};q).
\end{multline*}
The first sum, by Proposition~\ref{Cat A q}, is $(\Cat(A_n;q)-(1+q)\Cat(A_{n-1};q))$.
The second sum can be reindexed to $q\sum_{i=0}^{n-2}\Cat\pp(A_{i};q)\Cat(A_{n-i-2};q)$, which, by Proposition~\ref{q Cat recursion A}, equals $\frac{q}{1+q}(\Cat(A_{n-1};q)-\Cat\pp(A_{n-1};q))$.
We obtain
\begin{multline*}
\Cat(A_n;q)=\Cat\pp(A_n;q)+(1+q)\Cat(A_{n-1};q)\\
+\Cat(A_n;q)-(1+q)\Cat(A_{n-1};q)\\
-\frac{q}{1+q}(\Cat(A_{n-1};q)-\Cat\pp(A_{n-1};q)),
\end{multline*}
which simplifies to the desired identity.
\end{proof}

The data also suggests that $\Cat\pp(D_n)=(n-2)\Cat(A_{n-2})$.
Indeed, the following is a $q$-analog.

\begin{prop}\label{Cat++ D q}
For $n\ge2$, 
\begin{equation}\label{Cat++ D q formula}
\Cat\pp(D_n;q) = (n-2)q\Cat(A_{n-2};q).
\end{equation}
\end{prop}
\begin{proof}
For $n=2$, the identity is $q+q^2=(3-2)q(1+q)$.
If $n\ge3$, then we start with \eqref{q Cat recursion D formula}.
The first summation in the formula can be rewritten, using \eqref{q Cat recursion A formula}, as 
\begin{equation}\label{firstsum}
(1+q)\bigl(\Cat(A_{n-1};q)-(1+q)\Cat(A_{n-2};q)-\Cat\pp(A_{n-1};q)\bigr).
\end{equation}
By induction, the second summation can be rewritten as 
\begin{equation}\label{secondsum}
(1+q)q\sum_{i=3}^{n-1}(i-2)\Cat\pp(A_{i-2};q)\Cat(A_{n-i-1};q).
\end{equation}
To further simplify \eqref{secondsum}, we use \eqref{Cat A q formula} to calculate
\begin{align*}
&(n-3)\bigl(\Cat(A_{n-1};q)-(1+q)\Cat(A_{n-2};q)\bigr)\\
&\qquad=(n-3)q\sum_{i=1}^{n-2}\Cat(A_{i-1};q)\Cat(A_{n-i-2};q)\\
&\qquad=q\sum_{i=1}^{n-2}\bigr((i-1)\Cat(A_{i-1};q)\Cat(A_{n-i-2};q)\\&\qquad\qquad\qquad\qquad\qquad\qquad\qquad\qquad+(n-i-2)\Cat(A_{i-1};q)\Cat(A_{n-i-2};q)\bigr)\\
&\qquad=q\sum_{i=1}^{n-2}(i-1)\Cat(A_{i-1};q)\Cat(A_{n-i-2};q)\\&\qquad\qquad\qquad\qquad\qquad\qquad\qquad+q\sum_{i=1}^{n-2}(n-i-2)\Cat(A_{i-1};q)\Cat(A_{n-i-2};q)
\end{align*}
Both sums can be reindexed to agree with \eqref{secondsum}, except for the initial factor $(1+q)$.
Thus \eqref{secondsum} equals $\frac{n-3}2(1+q)(\Cat(A_{n-1};q)-(1+q)\Cat(A_{n-2};q))$.
Finally, we use \eqref{Cat D A formula} to rewrite the $\Cat(D_n;q)$.
We obtain
\begin{multline}
\frac{n+1}2(1+q)\Cat(A_{n-1};q)-\Bigl(\frac{n-1}2+q+\frac{n-1}2q^2\Bigr)\Cat(A_{n-2})=\\
(1+q)\Cat(A_{n-1};q)+(1+q)\Cat\pp(A_{n-1};q)+\Cat\pp(D_n;q)\\
+(1+q)\bigl(\Cat(A_{n-1};q)-(1+q)\Cat(A_{n-2};q)-\Cat\pp(A_{n-1};q)\bigr)\\
+\frac{n-3}2(1+q)(\Cat(A_{n-1};q)-(1+q)\Cat(A_{n-2};q)).
\end{multline}
This can be rearranged to say $\Cat\pp(D_n;q) = (n-2)q\Cat(A_{n-2};q)$.
\end{proof}

In order to establish a needed identity for double-positive Catalan numbers of type B, we need a recursion for the $q$-Catalan number that comes from a completely different direction.
The $q$-Catalan numbers $\Cat(W;q)$ encode the $h$-vector of the generalized associahedron for $W$.
(See, for example, \cite[Section~5.2]{rsga}.)
For each Coxeter group $W$ of rank $n$ and each $i$ from $0$ to $n$, define $f_i$ to be the number of simplices in the simplicial generalized associahedron having exactly $i$ vertices (and thus dimension $i-1$).
Define a polynomial
\[f(W;x)=\sum_{k=0}^nf_k(W)x^k.\]
The following is \cite[Proposition~3.7]{ga}.
\begin{prop}\label{ga prop}
If $W$ is reducible as $W_1\times W_2$, then $f(W;x)=f(W_1;x)f(W_2;x)$.
If $W$ is irreducible with Coxeter number $h$, then 
\begin{equation}\label{ga prop formula}
\dod{f(W;x)}{x}=\frac{h+2}2\sum_{s\in S}f(W_{S\setminus\set{s}};x)
\end{equation}
\end{prop}

Since $f(W)$ encodes the $f$-vector of the generalized associahedron and $\Cat(W;q)$ encodes the $h$-vector, \eqref{ga prop formula} implies a formula for $\Cat(W;q)$. 
Since $f(W)$ is has coefficients reversed from the $f$-polynomial usually used to define $h$-vectors, the formula for $\Cat(W;q)$ is somewhat more complicated than \eqref{ga prop formula}.

\begin{prop}\label{q Cat FZ}
For an irreducible Coxeter group $W$ with rank $n\ge0$ and Coxeter number $h$, the $q$-analog of the Catalan number satisfies  
\begin{equation}\label{q Cat FZ formula}
n\Cat(W;q)+(1-q)\dod{}{q}\Cat(W;q)=\frac{h+2}2\sum_{s\in S}\Cat(W_{S\setminus\set{s}};q).
\end{equation}
\end{prop}
\begin{proof}

We begin with the right side of \eqref{q Cat FZ formula} and replace $q$ by $x+1$ throughout.
The result is $\frac{h+2}2\sum_{s\in S}\rev(f(W_{S\setminus\set{s}};x))$, where $\rev$ is the operator that reverses the coefficients of a polynomial.
In other symbols: $x^{n-1}\frac{h+2}2\sum_{s\in S}f(W_{S\setminus\set{s}};x^{-1})$
Using \eqref{ga prop formula}, the quantity becomes $x^{n-1}\dod{f(W;x^{-1})}{(x^{-1})}$.

Similarly, $\Cat(W;x+1)=\rev(f(W;x))=x^nf(W;x^{-1})$, so $f(W;x^{-1})=x^{-n}\Cat(W;x+1)$.
Thus the right side of \eqref{q Cat FZ formula} equals 
\begin{align*} 
&x^{n-1}\dod{}{(x^{-1})}\left[x^{-n}\Cat(W;x+1)\right]\\
&=x^{n-1}\dod{}{x}\left[x^{-n}\Cat(W;x+1)\right](-x^2)\\
&=-x^{n+1}\left[-nx^{-n-1}\Cat(W;x+1)+x^{-n}\dod{}{x}\Cat(W;x+1)\right]\\
&=n\Cat(W;x+1)-x\dod{}{x}\Cat(W;x+1)\\
&=n\Cat(W;x+1)-x\dod{}{(x+1)}\Cat(W;x+1)
\end{align*}
Replacing $x$ by $q-1$ throughout, we obtain the left side of \eqref{q Cat FZ formula}.
\end{proof}

The type-B version of \eqref{q Cat FZ formula} is the following recursion:
\begin{prop}\label{q Cat FZ B}
For $n\ge 0$, 
\begin{equation}\label{q Cat FZ B formula}
n\Cat(B_n;q)+(1-q)\od{}{q}\Cat(B_n;q)=(n+1)\sum_{i=1}^n\Cat(A_{i-1};q)\Cat(B_{n-i};q).
\end{equation}
\end{prop}
The following formula is obtained using known formulas for Narayana numbers of types A and B.
\begin{prop}\label{mystery prop}
For $n\ge0$,
\begin{equation}\label{mystery formula}  
\sum_{i=1}^n\Cat(A_{i-1};q)\Cat(B_{n-i};q)=n\Cat(A_{n-1};q).
\end{equation}
\end{prop}
\begin{proof}
By \eqref{q Cat FZ B formula}, the assertion is equivalent to 
\begin{equation}\label{Cat B q ODE formula}
n\Cat(B_n;q)+(1-q)\od{}{q}\Cat(B_n;q)=n(n+1)\Cat(A_{n-1};q).
\end{equation}
Taking the coefficient of $q^k$ on both sides, we see that \eqref{Cat B q ODE formula} is equivalent to 
\begin{equation}\label{Nar B q ODE formula}
(n-k)\Nar_k(B_n)+(k+1)\Nar_{k+1}(B_n)=n(n+1)\Nar_k(A_{n-1}).
\end{equation}
This can be verified using the formulas for the type-A and type-B Narayana numbers, found for example in \cite[(9.1)]{gcccc} and \cite[(9.2)]{gcccc} (setting $m=1$ in both formulas from~\cite{gcccc}).
\end{proof}

Using \eqref{q biCat recursion A formula} and the observation that $\biCat(A_n;q)=\Cat(B_n;q)$, then applying \eqref{q Cat recursion B 2 formula} twice, (where, in the first instance $n$ is replaced by $n+1$ in \eqref{q Cat recursion B 2 formula}), we obtain the following formula.
\begin{prop}\label{gratuitous prop}
For $n\ge1$,
\begin{multline}\label{gratuitous formula}
\Cat(B_n;q)=(1+q)\Cat(B_{n-1};q)-(1+q)\Cat\pp(A_{n-1};q)\\+\Cat\pp(B_n;q)+(1+q)\Cat\pp(B_{n-1};q).
\end{multline}
\end{prop}

Next, we obtain the following formula.

\begin{prop}\label{Cat B q combination}
For $n\ge2$, 
\begin{multline}\label{Cat B q combination formula}  
(1+q)\Cat(B_n;q)=\\(1+q+q^2)\Cat(B_{n-1};q)
+(n-1)q(1+q)\Cat(A_{n-2};q)\\
+q\Cat\pp(B_{n-1};q) 
+ (1+q)\Cat\pp(B_n;q).
\end{multline}
\end{prop}
\begin{proof}
Using \eqref{Cat++ A q formula} to replace each instance of $(1+q)\Cat\pp(A_i;q)$ in \eqref{q Cat recursion B 2 formula} 
with $q(\Cat(A_{i-1};q)-\Cat\pp(A_{i-1};q))$ for $i>0$ and splitting into two sums, we obtain:
\begin{multline*}
\Cat(B_n;q)=(1+q)\Cat(B_{n-1};q)+q\sum_{i=1}^{n-1}\Cat(A_{i-1};q)\Cat(B_{n-i-1};q)\\
-q\sum_{i=1}^{n-1}\Cat\pp(A_{i-1};q)\Cat(B_{n-i-1};q) + \Cat\pp(B_n;q)
\end{multline*}
We use \eqref{mystery formula} with $n$ replaced by $n-1$ to evaluate the first sum.
We reindex the second sum and evaluate it using \eqref{q Cat recursion B 2 formula} with $n$ replaced by $n-1$.
\begin{multline*}
\Cat(B_n;q)=(1+q)\Cat(B_{n-1};q)+q(n-1)\Cat(A_{n-2};q)\\
-\frac{q}{1+q}(\Cat(B_{n-1};q)-\Cat\pp(B_{n-1};q)) 
+ \Cat\pp(B_n;q).
\end{multline*}
We multiply through by $(1+q)$ and simplify to obtain \eqref{Cat B q combination formula}.
\end{proof}

Solving both \eqref{Cat B q combination formula} and \eqref{gratuitous formula} for $(1+q)\Cat\pp(B_n;q)$ and combining them, then solving for $(1+q+q^2)\Cat\pp(B_{n-1};q)$, we obtain the key result for $\Cat\pp(B_{n-1})$.
\begin{prop}\label{Cat++ Bn-1}
\begin{multline}\label{Cat++ Bn-1 formula}  
(1+q+q^2)\Cat\pp(B_{n-1};q)=-q\Cat(B_{n-1};q)\\+(n-1)q(1+q)\Cat(A_{n-2};q)
+(1+q)^2\Cat\pp(A_{n-1};q)
\end{multline}
\end{prop}

We have established all the propositions we will need for the type-D case of Theorem~\ref{enum thm}.
We now provide formulas and values for the double-positive $W$-Catalan numbers, although they are not necessary for the Coxeter-biCatalan combinatorial results of this paper.

\begin{theorem}\label{Cat++ thm}
The double-positive $W$-Catalan numbers are
\begin{multline*}\begin{array}{c||c|c}
W		&A_n\quad (n\ge0)										&B_n\quad(n\ge1)						\\\hline
\Cat\pp(W)&\sum_{k=0}^n(-1)^k\frac{k+1}{n+1}\binom{2n-k}{n}	
&
\sum_{k=0}^n(-1)^k\binom{2n-k-1}{n-1}
		\\
\end{array}\\[10pt]
\begin{array}{|c|c|c|c|c|c|c|c}
D_n	\quad(n\ge2)			&E_6	&E_7	&E_8	&F_4	&H_3	&H_4	&I_2(m)\\\hline
\frac{n-2}n\binom{2n-2}{n-1}	&265	&1728	&13816	&49		&16		&208	&m-2
\end{array}
\end{multline*}
\end{theorem}
In particular, the numbers $\Cat\pp(A_n)$ are OEIS \cite{OEIS} sequence number A000957 (the Fine numbers).
The type-B and type-D double-positive Catalan numbers are sequences A014301 and A276666.

\begin{proof}
The proof for $\Cat\pp(A_n)$ can be found for example in \cite[Sections~3--4]{DS}, where both the $q=1$ specialization of \eqref{Cat++ A q formula} and the formula $\sum_{k=0}^n(-1)^k\frac{k+1}{n+1}\binom{2n-k}{n}$ are shown to describe the Fine numbers.

%
%
To obtain the formula for $\Cat\pp(B_n)$, we start from \eqref{Cat++ Bn-1 formula}, set $q=1$ and shift the index $n$ to obtain
\begin{equation}\label{shifted}
3\Cat\pp(B_n)=-\Cat(B_n)+2n\Cat(A_{n-1})+4\Cat\pp(A_n)
\end{equation}
We use the $q=1$ specialization of \eqref{Cat++ A q formula}
to rewrite $4\Cat\pp(A_n)$, as $2\Cat(A_{n-1})-2\Cat\pp(A_{n-1})$, use known formulas for $\Cat(B_n)$, $\Cat(A_{n-1})$, and $\Cat\pp(A_{n-1})$, and reindex a sum to obtain
\begin{equation}\label{rewrite and known}
3\Cat\pp(B_n)=\binom{2n}n+2\sum_{k=0}^n(-1)^k\frac{k}{n}\binom{2n-k-1}{n-1}
\end{equation}
Verifying the desired formula for $\Cat\pp(B_n)$ can thus be reduced to verifying the identity
\begin{equation}\label{verifying}
\sum_{k=0}^n(-1)^k\frac{3n-2k}n\binom{2n-k-1}{n-1}=\binom{2n}n.
\end{equation}
Rewriting $\frac{3n-2k}n\binom{2n-k-1}{n-1}$ as $2\binom{2n-k}n-\binom{2n-k-1}{n-1}=\binom{2n-k}n+\binom{2n-k-1}n$, we see that the alternating sum telescopes, collapsing to $\binom{2n}n$ as desired.

The formula for $\Cat\pp(D_n)$ follows from \eqref{Cat++ D q formula} and the usual formula for $\Cat(A_n)$.
The values for the exceptional types are easily computed.
\end{proof}

\subsection{The Type D biCatalan number}\label{type D sec}  
We now complete the proof of Theorem~\ref{enum thm} by proving the following theorem.

\begin{theorem}\label{D biCat}
For $n\ge 2$, the $D_n$-biCatalan number is 
\begin{equation}\label{D biCat formula}
\biCat(D_n)=6\cdot4^{n-2} - 2\binom{2n-4}{n-2}.
\end{equation}
\end{theorem}

Since we have already established the type-A and type-B cases of Theorem~\ref{enum thm}, Theorem~\ref{D biCat} is the assertion that $\biCat(D_n) = 3\biCat(B_{n-1}) - 2\biCat(A_{n-2})$.
In preparation for the proof, we let $X=X(q)$ and $Y=Y(q)$ be any rational functions of $q$ and define, for each $n\ge2$, a rational function $Z_n=Z_n(q)$ given by
\[Z_n=\biCat(D_n;q)-X\biCat(B_{n-1};q)+Y\biCat(A_{n-2};q).\]
Combining \eqref{q biCat recursion D formula}, \eqref{q biCat recursion B formula}, and \eqref{q biCat recursion A formula}, we obtain the following recursion for $Z_n$ for $n\ge3$.
\begin{multline}\label{Zn recursion formula}
Z_n=(1+q)Z_{n-1}+\sum_{i=1}^{n-3}\Cat\pp(A_i;q)\left(Z_{n-i}+(1+q)Z_{n-i-1}\right)\\
+2\bigl((1+q)^2-X(1+q)+Y\bigr)\Cat\pp(A_{n-2})+4(1+q)\Cat\pp(A_{n-1})\\
+2\Cat\pp(D_n;q)-2X\Cat\pp(B_{n-1};q)
\end{multline}

One way to obtain a formula for $q$-biCatalan numbers $\biCat(D_n;q)$ would be to find a choice of $X$ and $Y$ that makes this recursion for $Z_n$ into something that can be solved.
We have thus far been unable to find a choice of $X$ and $Y$ that works.
Instead, we will prove Theorem~\ref{D biCat} by showing that if $X(1)=3$ and $Y(1)=2$, then $Z_n(1)=0$ for all $n\ge2$.
In the proof that follows, we take convenient choices of $X$ and $Y$ but delay specializing $q$ to $1$ until the end, because specializing earlier does not make the manipulations much easier, and because we hope that perhaps we are still getting closer to a formula for $\biCat(D_n;q)$.

\begin{proof}[Proof of Theorem~\ref{D biCat}]
Substituting \eqref{Cat++ Bn-1 formula} and \eqref{Cat++ D q formula} into \eqref{Zn recursion formula}, taking $X$ to be ${1+q+q^2}$, and taking $Y$ to be $2q-q^2+q^3$, we obtain
\begin{multline}\label{Zn recursion formula specialized}
Z_n=(1+q)Z_{n-1}+\sum_{i=1}^{n-3}\Cat\pp(A_i;q)\left(Z_{n-i}+(1+q)Z_{n-i-1}\right)\\
+2q(1-q)\Cat\pp(A_{n-2};q)+2(1-q)(1+q)\Cat\pp(A_{n-1};q)\\
-2q\bigl(1+(n-1)q\bigr)\Cat(A_{n-2};q)+2q\Cat(B_{n-1};q)
\end{multline}
We next apply \eqref{Cat++ A q formula} to rewrite the two double-positive $q$-Catalan numbers in \eqref{Zn recursion formula specialized} as a single $q$-Catalan number.
\begin{multline}\label{Zn recursion formula specialized more}
Z_n=(1+q)Z_{n-1}+\sum_{i=1}^{n-3}\Cat\pp(A_i;q)\left(Z_{n-i}+(1+q)Z_{n-i-1}\right)\\
+2q(1-q)\Cat(A_{n-1};q)-2q\bigl(1+(n-1)q\bigr)\Cat(A_{n-2};q)+2q\Cat(B_{n-1};q)
\end{multline}

Finally specializing $q$ to $1$ and using the fact that $\Cat(B_{n-1})=n\Cat(A_{n-2})$ for $n\ge3$ (which is immediate from the well-known formulas for the type-A and type-B Catalan numbers), we see that
\begin{equation}\label{Zn recursion formula specialized even more}
Z_n(1)=2Z_{n-1}(1)+\sum_{i=1}^{n-3}\Cat\pp(A_i)\left(Z_{n-i}(1)+2Z_{n-i-1}(1)\right)
\end{equation}  
We easily verify that $Z_2(1)=0$, and thus we have a simple inductive proof that $Z_n(1)=0$ for all $n\ge2$.
Since we chose $X$ and $Y$ to have $X(1)=3$ and $Y(1)=2$, we obtain the desired identity $\biCat(D_n) = 3\biCat(B_{n-1}) - 2\biCat(A_{n-2})$.
\end{proof}

\subsection{Type-D biNarayana numbers}\label{type D biNar sec}
Computational evidence suggests the following modest conjecture on the type-D biNarayana number $\biNar_k(D_n)$.

\begin{conj}\label{D biNar conj}
The type-D biNarayana number $\biNar_k(D_n)$ is a polynomial in~$n$ (for $n\ge2$) of degree $2k$ and leading coefficient $\displaystyle\frac{4^k}{(2k)!}$.
\end{conj}

If Conjecture~\ref{D biNar conj} is true, then the following table shows $\frac{(2k)!}{2^k}\cdot\biNar_k(D_n)$ for small $k$.
\[\displaystyle\begin{array}{ll|ll}
k&&&\displaystyle\frac{(2k)!}{2^k}\cdot\biNar_k(D_n)\\[7pt]\hline&&\\[-9pt]
0&&&1\\[2pt]\hline&\\[-9pt]
1&&&2n^2-3n\\[2pt]\hline&\\[-9pt]
2&&&4n^4-20n^3+35n^2-7n-24\\[2pt]\hline&\\[-9pt]
3&&&8n^6-84n^5+365n^4-705n^3+212n^2+1104n-1080\\[2pt]\hline&\\[-9pt]
4&&&16n^8-288n^7+2268n^6-9576n^5+20349n^4\\
&&&\qquad\qquad\qquad\qquad\qquad\qquad-8022n^3-54133n^2+104826n-60480
\end{array}\]
The $k=1$ case is verified by Proposition~\ref{biNar 1}, and with some effort, the $k=2$ case can be proved as well.

\section*{Acknowledgments} \label{sec:ack}
Bruno Salvy's and Paul Zimmermann's package \texttt{GFUN} \cite{GFUN} was helpful in guessing a formula for the $D_n$-Catalan number.
John Stembridge's packages \texttt{posets} and \texttt{coxeter/weyl} \cite{StembridgePackages} were invaluable in counting antichains in the doubled root poset, in checking the distributivity of the doubled root poset, and in verifying the simpliciality of the bipartite biCambrian fan.  
The authors thank Christos A. Athanasiadis, Christophe Hohlweg, Richard Stanley, Salvatore Stella, and Bernd Sturmfels for helpful suggestions and questions.



\end{document}